\documentclass[A4paper]{amsart}

\usepackage[all]{xy}
\usepackage{amsmath}
\usepackage{amssymb} 
\usepackage{amsfonts}
\usepackage{multicol} 
\usepackage{hyperref}
\usepackage{lscape}
\usepackage{wasysym}

\newtheorem{theorem}{Theorem}[section]
\newtheorem{proposition}[theorem]{Proposition}
\newtheorem{corollary}[theorem]{Corollary}
\newtheorem{lemma}[theorem]{Lemma}
\newtheorem*{theorem*}{Theorem}
\newtheorem*{proposition*}{Proposition}
\newtheorem*{corollary*}{Corollary}
\newtheorem*{lemma*}{Lemma}
\theoremstyle{definition}
\newtheorem{definition}[theorem]{Definition}
\newtheorem{punto}[theorem]{}
\newtheorem{example}[theorem]{Example}
\newtheorem{remark}[theorem]{Remark}
\newtheorem*{remark*}{Remark}

\newtheorem*{definition*}{Definition}
\newtheorem{aussage}[theorem]{}
\newtheorem*{aussage*}{}

\numberwithin{equation}{section}

\def\ox{\otimes}
\def\cat{\mathsf}
\def\·{\cdot}
\def\mathbbm{\mathbf}
\def\other{^{{\prime}}}
\def\grtilde{\tilde}

\newcommand{\tagarray}{\mbox{}\refstepcounter{equation}$(\theequation)$}
\def\cdotdot{\cdot\!\!\cdot\,}

\begin{document}

\title{On the category of weak bialgebras} 
\author{Gabriella B\"ohm}
\address{Wigner Research Centre for Physics, Budapest,
H-1525 Budapest 114, P.O.B.\ 49, Hungary}
\email{bohm.gabriella@wigner.mta.hu}
\author{Jos\'e G\'omez-Torrecillas}
\address{Departamento de \'Algebra, Universidad de Granada, 
E-18071 Granada, Spain} 
\email{gomezj@ugr.es}
\thanks{Research partially supported by the Hungarian Scientific Research
 Fund OTKA --- grant no. K 108384 ---, the Spanish Ministerio de 
 Ciencia en Innovaci\'on and the European Union --- grant MTM2010-20940-C02-01
 and integrated grants from the FPI programme BES-2011-044383 and
 EEBB-I-12-05437 --- and the Nefim Funds of Wigner RCP, Budapest.\\ 
 All authors are grateful to the referee for highly valuable comments. 
 G. B. gratefully acknowledges a generous invitation and warm hospitality at
 Universidad de Granada in February of 2012. 
 J. G.-T. thanks the members of the Wigner RCP for their kind invitation and
 for the warmest hospitality during his visits in May of 2012 and February of
 2013. 
 E. L.-C. fully thanks her coauthors for their great support and the members
 of the Wigner RCP for the warmest hospitality and kindness experienced
 during her stay in April, May and June of 2012 and her visit in January and
 February of 2013.} 
\author{Esperanza L\'opez-Centella}
\address{Departamento de \'Algebra, Universidad de Granada,
E-18071 Granada, Spain}
\email{esperanza@ugr.es}

\begin{abstract}Weak (Hopf) bialgebras are described as (Hopf) bimonoids in
 appropriate duoidal (also known as 2-monoidal) categories. This
 interpretation is used to define a category $\mathsf{wba}$ of weak
 bialgebras over a given field. As an application, the ``free vector space''
 functor from the category of small categories with finitely many objects to
 $\mathsf{wba}$ is shown to possess a right adjoint, given by taking
 (certain) group-like elements. This adjunction is proven to
 restrict to the full subcategories of groupoids and of weak Hopf algebras,
 respectively. As a corollary, we obtain equivalences between the category of
 small categories with finitely many objects and the category of pointed
 cosemisimple weak bialgebras; and between the category of small groupoids
 with finitely many objects and the category of pointed cosemisimple weak
 Hopf algebras. 
\end{abstract}
\date{Oct, 2013}
\maketitle

\section*{Introduction.}
Some generalizations of (Hopf) bialgebras --- which have been studied
intensively on their own right --- were shown to be instances of (Hopf)
bimonoids in appropriately constructed braided (or even symmetric) monoidal
categories. This was done, for example, in \cite{CaeDLo} for Turaev's group
(Hopf) bialgebras \cite{Tur} and in \cite{CaeGoy} for Makhlouf and
Silvestrov's hom (Hopf) bialgebras \cite{SilMak}. Such a description allows
for a unified treatment of all these structures, it conceptually explains the
origin of some results obtained earlier by other means and it also makes
available the general theory of (Hopf) bimonoids in braided monoidal
categories.

{\em Weak (Hopf) bialgebras} in \cite{BohmEtAll:1999}, however, do not seem to
be (Hopf) bimonoids in any braided monoidal category. Our first aim in this
paper is to describe them rather as (Hopf) bimonoids in so-called {\em
duoidal categories}. 

Duoidal categories were introduced by Aguiar and Mahajan in
\cite{Aguiar&Mahajan} under the original name `2-monoidal category'. These are 
categories with two, possibly different, monoidal structures. They are
required to be compatible in the sense that the functors and natural
transformations defining the first monoidal structure, are comonoidal with
respect to the second monoidal structure. Equivalently, the functors and
natural transformations defining the second monoidal structure, are monoidal
with respect to the first monoidal structure. Whenever both monoidal
structures coincide, we re-obtain the notion of braided monoidal category. 
More details will be recalled in Section \ref{sec:prelims}. A {\em bimonoid}
in a duoidal category is a monoid with respect to the first monoidal structure
and a comonoid with respect to the second monoidal structure. The
compatibility axioms are formulated in terms of the coherence morphisms
between the monoidal structures, see \cite{Aguiar&Mahajan} (and a short review
in Section \ref{sec:prelims}). In the spirit of \cite{Booker&Street}, a
bimonoid is said to be a {\em Hopf monoid} provided that it induces a right
Hopf comonad in the sense of \cite{BrLaVi}, see Section \ref{sec:prelims}. 

An inspiring example in \cite[Example 6.43]{Aguiar&Mahajan} says that small
categories can be described as bimonoids in an appropriately chosen duoidal
category: in the category of spans over a given set (the set of objects). This 
construction is re-visited in Section \ref{sec:cat}. By
this motivation we aim to find an appropriate duoidal category whose
bimonoids are `quantum categories'; that is, weak
bialgebras. Recall that weak bialgebras are examples of Takeuchi's
$\times_R$--bialgebras \cite{Takeuchi}, equivalently, of Lu's bialgebroids
\cite{Lu}; such that the base algebra $R$ carries a separable Frobenius
structure \cite{Szlach:2001,Schauenburg:2003}. Bialgebroids whose base algebra
$R$ is central, were described in \cite[Example 6.44]{Aguiar&Mahajan} as
bimonoids in the duoidal category of $R$--bimodules. It was also discussed
there that arbitrary bialgebroids are beyond this framework because the
candidate --- Takeuchi's $\times_R$--operation --- does not define a monoidal
product in general. 

In Section \ref{sec:wba} we study the category of bimodules over $R\ox R^{op}$
for a separable Frobenius algebra $R$. Observing that in this case Takeuchi's 
$\times_R$--product becomes isomorphic to some (twisted) bimodule tensor
product over $R\ox R^{op}$, we equip this category with a duoidal
structure. Moreover, we show that its bimonoids are precisely the weak
bialgebras whose base algebra is isomorphic to $R$. 

This interpretation of weak bialgebras as bimonoids allows us to define a
category $\mathsf{wba}$ of weak bialgebras (by applying a more
general construction in Section \ref{sec:cat_bmd}). Morphisms, from a
weak bialgebra $H$ with separable Frobenius base algebra $R$, to a weak
bialgebra $H'$ with separable Frobenius base algebra $R'$, are pairs of
coalgebra maps $R\to R'$ and $H\to H'$ with additional properties that ensure
that they induce a morphism of monoidal comonads --- in the sense of
\cite{Szlachanyi:2003} --- from the monoidal comonad induced by $H$ on the
category of $R\ox R^{op}$--bimodules to the monoidal comonad induced by $H'$ on
the category of $R'\ox R^{\prime op}$--bimodules. 

The vector space spanned by any small category with finitely many objects
carries a weak bialgebra structure \cite{Bohm:2009,NykVai}. This turns out to
yield the object map of a functor $\mathsf k$ from the category $\mathsf{cat}$
of small categories with finitely many objects to $\mathsf{wba}$, see Section
\ref{sec:functor_k}. In Sections \ref{sec:gr-like} and \ref{sec:adjoint} we
show that it possesses a right adjoint: For the interval category
$\mathbbm{2}$ and any weak bialgebra $H$, we consider the set $\mathsf g
(H):=\mathsf{wba}(\mathsf k (\mathbbm 2),H)$. In general, it is isomorphic to
a subset of the set of so-called `group-like elements'; that is, 
of coalgebra maps from the base field to $H$ (not to be mixed with the weakly
group-like elements in \cite{BohmEtAll:1999} and \cite{Vecs}). 
In favorable situations --- for example, if $H$ is cocommutative
or $H$ is a weak Hopf algebra --- $\mathsf g (H)$ is proven to be isomorphic
to the set of group-like elements. 
For any weak bialgebra $H$, $\mathsf g (H)$ is interpreted as the morphism set of
a category and it is shown to obey $\mathsf{wba}(\mathsf k (A),H)\cong 
\mathsf{cat}(A,\mathsf g(H))$, for any small category $A$ with finitely many
objects. The unit of this adjunction is a natural isomorphism. The component
of the counit at some weak bialgebra $H$ is an isomorphism if and only if $H$
is pointed cosemisimple (as a coalgebra). So we obtain an equivalence between
$\mathsf{cat}$ and the full subcategory in $\mathsf{wba}$ of all pointed
cosemisimple weak bialgebras. 

Defining Hopf monoids in duoidal categories as bimonoids that induce right
Hopf comonads in the sense of \cite{BrLaVi}, the Hopf monoids in the duoidal
category of spans turn out to be precisely the small groupoids. In the duoidal
category of bimodules over $R\ox R^{op}$, for a separable Frobenius algebra
$R$, Hopf monoids turn out to be precisely the weak Hopf algebras with base
algebra isomorphic to $R$. In Section \ref{sec:Hopf} we show that the
adjunction between $\mathsf{cat}$ and $\mathsf{wba}$ restricts to an
adjunction between the category $\mathsf{grp}$ of small groupoids with
finitely many objects, and the full subcategory $\mathsf{wha}$ in
$\mathsf{wba}$ of all weak Hopf algebras. Consequently, the equivalence
between $\mathsf{cat}$ and the full subcategory in $\mathsf{wba}$ of all
pointed cosemisimple weak bialgebras restricts to an equivalence between
$\mathsf{grp}$ and the full subcategory in $\mathsf{wha}$ of all pointed
cosemisimple weak Hopf algebras. This extends the well-known relation between
groups and pointed cosemisimple Hopf algebras, see for example \cite{Abe}. 

\tableofcontents

\section{Preliminaries.}
\label{sec:prelims}

Let $(\cat{C},\otimes, I)$ be a monoidal category, with underlying category
$\cat{C}$, monoidal product $\otimes$ and unit $I$. For any objects
$A,B,C$ of $\cat{C}$, we will denote by $\alpha$ the associator natural
isomorphism 
$$
\xymatrix{
\alpha_{A,B,C}:(A\otimes B)\otimes C\ar[r] & 
A\otimes (B\otimes C),} 
$$
and by $\lambda$ and $\rho$ the unit natural isomorphisms 
\begin{equation*}
\lambda_A: I\otimes A\rightarrow A, \qquad \rho_A: A\otimes I\rightarrow A. 
\end{equation*}
The composition of morphisms will be denoted by juxtaposition. 

\subsection{Duoidal categories}
A \emph{duoidal category} (introduced in \cite{Aguiar&Mahajan} under the name
of $2$-monoidal category\label{duoidalcatdefinition}) is a five tuple
$(\cat{C},\circ,I,\bullet,J)$, where $(\cat{C},\circ,I)$ and
$(\cat{C},\bullet, J)$ are monoidal categories, along with a transformation
(called \emph{the interchange law}) 
\begin{equation}\label{duoidalinterchangelaw}
\gamma_{A,B,C,D}:(A\bullet B)\circ (C\bullet D)\rightarrow (A\circ
C)\bullet(B\circ D) 
\end{equation}
which is natural in $A,B,C$ and $D$, and three morphisms
\begin{equation}\label{duoidalmorphisms}
\mu_{J}: J \circ J\rightarrow J,\qquad \Delta_{I}: I\rightarrow I\bullet I,
\qquad \tau: I\rightarrow J 
\end{equation}
such that the axioms below are satisfied.

\textit{Compatibility of units}. The units $I$ and $J$ are compatible in the
sense that $(J,\mu_J,\tau)$ is a monoid in $(\cat{C},\circ,I)$ and
$(I,\Delta_I,\tau)$ is a comonoid in $(\cat{C},\bullet,J)$. Equivalently, the
following diagrams commute: 
\begin{center}
\begin{tabular}{ll}
\tagarray\qquad\hspace{12pt}
\raisebox{.8cm}{\xymatrix@R=8pt@C=26pt{
I\ar[rr]^-{\Delta_I} \ar[dd]_-{\Delta_I}& & 
I\bullet I\ar[d]^-{\Delta_I\bullet I}\\
& & (I \bullet I)\bullet I \ar[d]^-{\alpha^{\bullet}_{I,I,I}} \\
I\bullet I \ar[rr]_-{I\bullet \Delta_I} & & 
I\bullet (I \bullet I)}}\qquad&
\tagarray\qquad
\raisebox{.8cm}{\xymatrix@R=30pt{
J\bullet I\ar[drr]_-{\lambda_I^{\bullet}} & & 
I\bullet I \ar[ll]_-{\tau\bullet I}\ar[rr]^-{I\bullet \tau} & & 
I\bullet J\ar[dll]^-{\rho^{\bullet}_I}\\
& & I \ar[u]_-{\Delta_I} & & }}\qquad \\
\tagarray\qquad
\raisebox{.8cm}{\xymatrix@R=8pt{
(J\circ J) \circ J\ar[rr]^-{\mu_J\circ J}
\ar[d]_-{\alpha^{\circ}_{J,J,J}} & &
J\circ J\ar[dd]^-{\mu_J}\\
J \circ (J\circ J) \ar[d]_-{J\circ \mu_J} & &\\
J\circ J \ar[rr]_-{\mu_J}& & J}}\qquad&
\tagarray\qquad
\raisebox{.8cm}{\xymatrix@R=30pt{
I\circ J\ar[drr]_-{\lambda^\circ_J}\ar[rr]^-{\tau\circ J} & & 
J\circ J\ar[d]^-{\mu_J} & & J\circ I\ar[ll]_-{J\circ \tau}
\ar[dll]^-{\rho^\circ_J}\\
& & J & &}} \qquad
\end{tabular}
\end{center}

\textit{Associativity}. The following diagrams commute, for any objects
$A,B,C,D,E,F$: 
\begin{center}
\begin{tabular}{c}
\tagarray\label{diag:1associativity}
\qquad
\raisebox{1.5cm}{\xymatrix{
((A\bullet B)\circ (C\bullet D))\circ (E\bullet F) 
\ar[d]_-{\gamma\circ (E\bullet F)}\ar[r]^-{\alpha^{\circ}} & 
(A\bullet B)\circ ((C\bullet D)\circ(E\bullet F))
\ar[d]^-{(A\bullet B)\circ \gamma}\\ 
((A\circ C)\bullet (B\circ D)) \circ (E\bullet F) 
\ar[d]_-{\gamma} & 
(A\bullet B)\circ ((C\circ E)\bullet (D\circ F))\ar[d]^-{\gamma}\\
((A\circ C)\circ E)\bullet 
((B\circ D)\circ F)\ar[r]_-{\alpha^{\circ}\bullet \alpha^{\circ}} & 
(A\circ (C\circ E))\bullet (B\circ (D\circ F))
}}\\
\\
\tagarray\label{diag:2associativity}
\qquad
\raisebox{1.5cm}{\xymatrix{
((A\bullet B)\bullet C)\circ ((D\bullet E)\bullet F) 
\ar[d]_-{\gamma}\ar[r]^-{\alpha^{\bullet}\circ \alpha^{\bullet}} & 
(A\bullet (B\bullet C))\circ (D\bullet(E\bullet F))\ar[d]^-{\gamma}\\
((A\bullet B)\circ (D\bullet E)) \bullet (C\circ F) 
\ar[d]_-{\gamma\bullet (C\circ F)} & 
(A\circ D)\bullet ((B\bullet C)\circ (E\bullet F))
\ar[d]^-{(A\circ D) \bullet\gamma}\\
((A\circ D)\bullet (B\circ E))\bullet (C\circ F)\ar[r]_-{\alpha^{\bullet}} & 
(A\circ D)\bullet ((B\circ E)\bullet (C\circ F))
}}\\
\end{tabular}
\end{center}

\textit{Unitality}. The following diagrams commute, for any objects $A,B$: 
\begin{center}
\begin{tabular}{ll}
\tagarray\label{diag:unit1}
\quad \raisebox{.8cm}{\xymatrix@C=15pt{
I\circ(A\bullet B)\ar[rr]^-{\Delta_I\circ (A\bullet B)} 
\ar[d]_-{\lambda_{A\bullet B}^{\circ}}&& 
(I\bullet I)\circ (A\bullet B)\ar[d]^-{\gamma}\\
A\bullet B && 
(I\circ A)\bullet (I\circ B) 
\ar[ll]^-{ \lambda_A^{\circ}\bullet \lambda_B^{\circ}} 
}}\quad &
\tagarray\label{diag:unit2}
\quad\raisebox{.8cm}{\xymatrix@C=15pt{
(A\bullet B)\circ I\ar[rr]^-{(A\bullet B)\circ \Delta_I} 
\ar[d]_-{ \rho_{A\bullet B}^{\circ}}&& 
(A\bullet B)\circ (I\bullet I)\ar[d]^-{\gamma}\\
A\bullet B && 
(A\circ I)\bullet (B\circ I)
\ar[ll]^-{\rho_A^{\circ} \bullet \rho_B^{\circ}}
}}\\
\tagarray\label{diag:unit3}
\ \raisebox{.8cm}{\xymatrix@C=15pt{
J\bullet(A\circ B)\ar[d]_-{\lambda_{A\circ B}^{\bullet}} && 
(J\circ J)\bullet (A\circ B)\ar[ll]_-{\mu_J\bullet (A\circ B)}\\
A\circ B && 
(J\bullet A)\circ (J\bullet B)
\ar[ll]^-{\lambda_{A}^{\bullet} \circ\lambda_B^{\bullet}}\ar[u]_-{\gamma}
}}&
\tagarray\label{diag:unit4}
\quad \raisebox{.8cm}{\xymatrix@C=15pt{
(A\circ B)\bullet J\ar[d]_-{\rho_{A\circ B}^{\bullet}} && 
(A\circ B)\bullet (J\circ J)\ar[ll]_-{(A\circ B)\bullet \mu_J}\\
A\circ B && 
(A\bullet J)\circ (B\bullet J)
\ar[ll]^-{\rho_{A}^{\bullet} \circ\rho_B^{\bullet}}\ar[u]_-{\gamma}
}}
\end{tabular}
\end{center}

A \emph{bimonoid} in $\cat{C}$ is a quintuple $(H,\mu, \eta,\Delta,\epsilon)$
where $(H,\mu,\eta)$ is a monoid in $(\cat{C},\circ,I)$, $(H,\Delta,\epsilon)$
is a comonoid in $(\cat{C},\bullet,J)$ and both structures are compatible in
the sense that the following four diagrams commute.
\begin{center}
\begin{tabular}{ll}
\tagarray\label{diag:bimon1}
\qquad\raisebox{.8cm}{\xymatrix@C=2pt{
(H\bullet H)\circ (H\bullet H)\ar[rr]^-{\gamma} && 
(H\circ H)\bullet (H\circ H)\ar[d]^-{\mu\bullet {\mu}}\\
H\circ H\ar[r]_-{\mu}\ar[u]^-{\Delta\circ\Delta} 
& H\ar[r]_-{\Delta} & H\bullet H}}\qquad&
\tagarray\label{diag:bimon2}
\qquad\raisebox{.8cm}{\xymatrix@C=30pt{
H\circ H\ar[d]_-{\mu}\ar[r]^-{\epsilon\circ \epsilon} & 
J\circ J\ar[d]^-{\mu_{J}}\\
H\ar[r]_-{\epsilon} & J}}\qquad\\
\tagarray\label{diag:bimon3}
\qquad\hspace{25pt} \raisebox{.8cm}{\xymatrix@C=76pt{
I\ar[r]^-{\eta}\ar[d]_-{\Delta_{I}} &
H\ar[d]^-{\Delta}\\
I\bullet I\ar[r]_-{\eta\bullet\eta} &
H\bullet H}}\hspace{22pt}&
\tagarray\label{diag:bimon4}
\qquad\raisebox{.8cm}{\xymatrix{ & H \ar[dr]^-{\epsilon}& \\
I\ar[ur]^-{\eta}\ar[rr]_-{\tau} && J}}\qquad
\end{tabular}
\end{center}

A \emph{morphism of bimonoids} is a morphism of the underlying monoids and
comonoids.\\ 

\subsection{Monoidal comonads}
Let $\textbf{T}=(T,\delta,\varepsilon)$ be a comonad on a monoidal category
$(\cat{C}, \otimes, I)$ with comultiplication $\delta : T \to T^2$ and counit
$\varepsilon: T \to \cat{C}$. 
We call $\textbf{T}$ a \emph{monoidal comonad} if the endofunctor
$T:\mathsf{C}\to \mathsf{C}$ is monoidal and $\delta$ and $\varepsilon$ are
monoidal natural transformations. 
(We use the term `monoidal functor' in the sense of \cite[Section XI.2]{CWM}.
That is, we mean by it the existence of a morphism ${T_0}: I\to TI$ and a
natural transformation $T_2: T(-)\ox T(-)\to T((-)\ox (-))$ obeying the
evident associativity and unitality conditions. Some authors (for instance, the
authors of \cite{Aguiar&Mahajan}) use the name `lax monoidal' functor for the
same notion. Also for the dual notion, both names `comonoidal functor' (used
in this paper) and `colax monoidal functor' (used in \cite{Aguiar&Mahajan}) do
coexist.) As in \cite[Section XI.2]{CWM}, monoidality of $\delta$ and
$\varepsilon$ means commutativity of the diagrams 
$$
\xymatrix@C=15pt{
 TA\ox TB\ar[rr]^{{T_2}}\ar[d]_{\delta_A\ox \delta_B} && 
T(A\ox B)\ar[d]^{\delta_{A\ox B}}&
I\ar[rr]^{{T_0}}\ar@{=}[d]&& 
TI\ar[d]^{\delta_I}\\
 T^2A\ox T^2B\ar[r]_-{{ T_2}} &
 T(TA\ox TB) \ar[r]_-{ T{T_2}}& 
T^2(A\ox B)&
I\ar[r]_-{{ T_0}} & 
 TI\ar[r]_-{ TT_0}&
T^2I\\
TA\ox TB\ar[rr]^{{T_2}}\ar[d]_{\epsilon_A\ox \epsilon_B} &&
T(A\ox B)\ar[d]^{\epsilon_{A\ox B}}&
I\ar[rr]^{{T_0}}\ar@{=}[rrd] && 
TI\ar[d]^{\epsilon_I}\\
A\ox B\ar@{=}[rr] && 
A\ox B &
&& I}
$$
for any objects $A,B$. 
(Comonoidal monads, the dual concept, were introduced in \cite{Moerdijk:2002}
under the name of Hopf monads.) A bimonoid $H$ in a duoidal category
$(\mathsf{C},\circ,I,\bullet,J)$ induces monoidal comonads $(-)\bullet H$ and
$H\bullet(-)$ on $(M,\circ,I)$ and comonoidal monads $(-)\circ H$ and $H\circ
(-)$ on $(M,\bullet,J)$, see \cite{Booker&Street}. 

Let $\textbf{T'}=(T',\delta',\varepsilon')$ be a second monoidal comonad on a
monoidal category $(\cat{C}',\ox',I')$. Recall from \cite[Definition
 3.1]{Szlachanyi:2003}, that a morphism from $\textbf{T}$ to $\textbf{T'}$ is
a pair $(F:\cat{C}\rightarrow \cat{C}',\Phi:FT\rightarrow T'F)$ where $F$ is a
comonoidal functor and $\Phi$ is a comonad morphism (in the sense of \cite[$\S
 1$]{Street:1972}) rendering commutative also the diagrams 
\begin{equation}\label{comonadmorphism}
\xymatrix@C=2pt{
F(TA\ox TB)\ar[d]_-{F_2}\ar[rrr]^-{FT_2} &&&
FT(A\ox B)\ar[rrrr]^-{\Phi(A\ox B)}&&&&
T'F(A\ox B)\ar[d]^-{T'F_2}&&
FI\ar[d]_-{FT_0}\ar[rrrrrr]^-{F_0}&&&&&&
I'\ar[d]^-{T'_0}\\
FTA\ox' FTB\ar[rrrr]_-{\Phi A\ox' \Phi B}&&&&
T'FA\ox' T'FB\ar[rrr]_-{T'_2}&&&
T'(FA\ox' FB)&&
FTI\ar[rrr]_-{\Phi I}&&&
T'F I \ar[rrr]_-{T' F_0} &&&
T' I',}
\end{equation}
for any objects $A,B$ of $\cat{C}$.

A monoidal comonad $\textbf{T}$ is said to be a {\em right Hopf comonad}
whenever for any objects $A$ and $B$,
$$
\xymatrix@C=30pt{
TA\ox TB \ar[r]^-{\delta_A \ox TB}&
T^2A\ox TB\ar[r]^-{T_2}&
T(TA\ox B)}
$$
is an isomorphism (natural in the objects $A$ and $B$). 

\subsection{Separable Frobenius algebras}
Let $k$ be a field and $R$ be an algebra over $k$. Recall that $R$ is said to
be a \emph{Frobenius algebra} if there exists a Frobenius structure $(\psi,
e)$ consisting of a $k$--linear map $\psi : R \to k$ --- called the {\em
Frobenius functional} --- and an element $e=\sum_i e_i \ox f_i \in R\ox
R$ (the summation symbol will be omitted for brevity) --- called the
{\em Frobenius element} --- such that 
\begin{equation}
r=\psi(re_i)f_i = e_i\psi(f_ir)\qquad \forall r \in R.\label{eq:Frob-1}
\end{equation}
A Frobenius algebra $R$ is necessarily finite dimensional with finite
dual basis $e_i\ox \psi(f_i-)\in R \ox \mathsf{Hom}(R,k)$.

Any Frobenius algebra $R$ possesses a coalgebra structure with
$R$--bilinear coassociative comultiplication $\delta:R\rightarrow R\ox R$, 
$r \mapsto re=er$ and counit $\psi:R\rightarrow k$. Moreover, the
category of right (respectively left) $R$--modules is isomorphic to the
category of right (respectively left) $R$--comodules (see for example
\cite{Abrams}). Thus, each right $R$--module $M$ is endowed with a
right $R$--comodule structure via the coaction $m \mapsto me_i \ox
f_i$. Conversely, any $R$--comodule $m\mapsto m_0\ox m_1$ is an $R$--module
via the action $m\ox r\mapsto m_0\psi(m_1r)$. 

For any Frobenius algebra $R$ there is a unique algebra automorphism $\theta :
R \to R$, the so-called {\em Nakayama automorphism}, obeying 
\begin{equation}
\psi(sr) = \psi(\theta(r)s)\qquad \forall r,s \in R.\label{eq:Frob0}
\end{equation}
From \eqref{eq:Frob-1} and \eqref{eq:Frob0} we get the explicit forms 
\begin{equation}\label{Nakayamaexplicitform}
 \theta(r)=\psi(e_ir)f_i, \qquad \theta^{-1}(r)=\psi(rf_i)e_i
\end{equation}
 of the Nakayama automorphism and its inverse, 
and the following identities for all $r\in R$.
\begin{eqnarray}
\label{eq:Frob1} 
re_i \ox f_i &=& e_i \ox f_ir\\
\label{eq:Frob2} 
e_ir \ox f_i&=&e_i\ox\theta(r) f_i\\
\label{eq:Frob3}
 e_i \ox \theta(f_i) & = 
 & \theta^{-1}(e_i) \ox f_i\\
\label{eq:Frob4}
 \theta(e_i)\ox f_i = & f_i\ox e_i & = 
 e_i \ox \theta^{-1}(f_i)
\end{eqnarray}

An algebra $R$ is \emph{separable} if its multiplication is a split
epimorphism of $R$--bimodules (see for example
\cite{IngDeMey}). Equiva\-lently, there exists an element $\sum_i a_i\otimes 
b_i\in R\ox R$, called a separability structure for $R$, such that (omitting
again the summation symbol) 
\begin{equation}\label{separability}
ra_i \ox b_i = a_i \ox b_ir \quad \forall r\in R \quad \text{and} 
\qquad a_ib_i=1.
\end{equation}
In this case the element 
$\sum_i a_i\ox b_i$ is an idempotent in the enveloping algebra 
$R^e:=R\ox R^{op}$, that is, 
\begin{equation}\label{idempotencyFrob}
a_ia_j\ox b_jb_i=a_i\ox b_i. 
\end{equation}

We say that $R$ is a \emph{separable Frobenius algebra} if there exists a
Frobenius structure $(\psi,e)$ for $R$ such that $e=\sum_i e_i \ox f_i \in
R^e$ is also a separability structure (see \cite{Szlach:2001}). In this case
the canonical epimorphism $M\ox N \twoheadrightarrow M\ox_R N$ is split by 
$$
M\ox_R N \rightarrowtail M\ox N,\qquad
m\ox_R n \mapsto \sum_i me_i\ox f_in,
$$
for any right $R$--module $M$ and any left $R$--module $N$.
It is evident that the opposite algebra $R^{op}$ and the tensor product $R\ox
S$ of separable Frobenius algebras $R, S$ are also separable Frobenius
algebras.

\subsection{Weak bialgebras}
Recall from \cite{BohmEtAll:1999} that a \emph{weak bialgebra}
$(H,\mu,\eta,\Delta,\epsilon)$ over a field $k$ is an associative unital
$k$--algebra $(H,\mu,\eta)$ and a coassociative counital $k$--coalgebra
$(H,\Delta,\epsilon)$ such that 
\begin{eqnarray}
\Delta(ab)&=&\Delta(a)\Delta(b),\label{multiplicativitycomultiplication}\\
\epsilon(ab_1)\epsilon(b_2c)=&\epsilon(abc)&=\epsilon(ab_2)\epsilon(b_1c)
\label{weakmultiplicativitycounit}\\
(\Delta(1)\ox 1)(1\ox \Delta(1))=&\Delta^2(1)&=(1\ox \Delta(1))(\Delta(1)\ox
1)
\label{weakcomultiplicativityunit}
\end{eqnarray}
for all $a,b,c \in H$.
Here (and throughout) we use a simplified version of
Heynemann-Sweedler's notation: $\Delta(a)=a_1\ox a_2$, implicit summation
understood. The axioms expressing weak multiplicativity of the counit
(that is, \eqref{weakmultiplicativitycounit}) can be 
rephrased equivalently as
\begin{equation}\label{weakmultiplicativitycounitalternativeaxioms}
\epsilon(a1_1)\epsilon(1_2c)=\epsilon(ac)=\epsilon(a1_2)\epsilon(1_1c),
\end{equation}
for all $a,c\in H$. Indeed, \eqref{weakmultiplicativitycounit} $\Rightarrow$
\eqref{weakmultiplicativitycounitalternativeaxioms} is evident. Conversely,
for any $b\in H$, 
$$
1_1\ox \Delta(1_2b)=
1_1\ox 1_2b_1\ox 1_3b_2=
1_1\ox 1_21_{1'}b_1 \ox 1_{2'}b_2=
1_1\ox 1_2b_1 \ox b_2,
$$
by the multiplicativity and the coassociativity of $\Delta$, by
\eqref{weakcomultiplicativityunit}, and by the multiplicativity of $\Delta$
again. Therefore,
$$
\epsilon(ab_1)\epsilon(b_2c)=
\epsilon(a1_1)\epsilon(1_2b_1)\epsilon(b_2c)=
\epsilon(a1_1)\epsilon((1_2b)_1)\epsilon((1_2b)_2c)=
\epsilon(a1_1)\epsilon(1_2bc)=
\epsilon(abc),
$$
where the first and the last equalities follow by
\eqref{weakmultiplicativitycounitalternativeaxioms} and the penultimate one
follows by counitality of $\Delta$. One can argue symmetrically about the
other axiom in \eqref{weakmultiplicativitycounit}.

Frequently, we will write $H$ to denote the weak bialgebra, understanding
that the structure is given. In a weak bialgebra $H$ the counital (idempotent)
maps $H\rightarrow H$ defined by the formulae 
\begin{equation}\label{sqcapR&L}
\begin{array}{ll}
 \sqcap^L(h)=\varepsilon(1_{1}h)1_{2}, \qquad&
 \sqcap^R(h)=1_{1}\varepsilon(h1_{2})\\ 
\overline{\sqcap}^L(h)=\varepsilon(h1_{1})1_{2},\qquad& 
\overline{\sqcap}^R(h)=1_{1}\varepsilon(1_{2}h)
\end{array} 
\end{equation}
play an important role. The following identities are immediate consequences of
the above definitions. 
\begin{equation}\label{piLbarpiR=piLbar}
\begin{array}{llll}
\overline{\sqcap}^L\sqcap^R=\overline{\sqcap}^L,
&
\qquad\overline{\sqcap}^R\sqcap^L=\overline{\sqcap}^R,
&
\qquad\sqcap^R\overline{\sqcap}^L=\sqcap^R,
&
\qquad\sqcap^L\overline{\sqcap}^R=\sqcap^L,
\\
\overline{\sqcap}^R\sqcap^R=\sqcap^R,
&
\qquad\overline{\sqcap}^L\sqcap^L=\sqcap^L,
&
\qquad\sqcap^R\overline{\sqcap}^R=\overline{\sqcap}^R,
&
\qquad\sqcap^L\overline{\sqcap}^L=\overline{\sqcap}^L.
\end{array}
\end{equation}

By (\cite[Proposition 2.4]{BohmEtAll:1999}), the coinciding images of
$\sqcap^R$ and $\overline{\sqcap}^R$; and the coinciding images of $\sqcap^L$
and $\overline{\sqcap}^L$ are unital subalgebras of $H$ --- they are
called the `right' and `left' subalgebra, respectively --- and they
commute with each other. That is, for all $h,h'\in H$ 
\begin{equation}\label{piRpiLcommute}
\sqcap^R(h)\sqcap^L(h')=\sqcap^L(h')\sqcap^R(h).
\end{equation}
Moreover,
\begin{equation}\label{delta1}
\Delta(1) = \sqcap^R(1_1)\ox 1_2=
1_1\ox \sqcap^L(1_2)
=\sqcap^R(1_2)\ox \sqcap^L(1_1).
\end{equation}
By (\cite[Proposition 2.11]{BohmEtAll:1999}, \cite{Szlach:2001}), they are
separable Frobenius (co)algebras with respective separability Frobenius
structures 
$$
(\epsilon_{|\sqcap^R(H)}, 1_1 \ox \sqcap^R(1_2))\qquad
\qquad\textrm{and}\qquad\qquad
(\epsilon_{|\sqcap^{ L}(H)}, \sqcap^L(1_1) \ox 1_2).
$$
Corestriction yields coalgebra maps $\sqcap^R:H\to
\sqcap^R(H)$ and $\sqcap^L:H\to \sqcap^L(H)$; symmetrically, anti-coalgebra
maps $\overline \sqcap^R:H\to \sqcap^R(H)$ and $\overline \sqcap^L:H\to
\sqcap^L(H)$. Moreover, the maps $\sqcap^{R}$ and $\overline{\sqcap}^{L}$; and
also the maps $\overline \sqcap^{R}$ and ${\sqcap}^{L}$, induce mutually 
inverse anti-algebra and anti-coalgebra iso\-morphisms between the separable
Frobenius (co)algebras $\sqcap^{R}(H)$ and ${\sqcap}^{L}(H)$, see
for example \cite[Proposition 1.18]{BohmEtAll:}. 

By (\cite[Lemma 2.5]{{BohmEtAll:1999}}), $\sqcap^R$ is a right
$\sqcap^R(H)$--module map and $\sqcap^L$ is a left $\sqcap^L(H)$--module
map. Symmetrically, $\overline{\sqcap}^R$ is a left ${\sqcap}^R(H)$--module
map and $\overline{\sqcap}^L$ is a right ${\sqcap}^L(H)$--module map. In
formulae, the following identities hold true for all $h, h'\in H$.
\begin{equation}\label{R-moduleproperties}
\begin{array}{ll}
\sqcap^R(h\sqcap^R(h'))=\sqcap^R(h)\sqcap^R(h'),
&
\qquad\sqcap^L(\sqcap^L(h)h')=\sqcap^L(h)\sqcap^L(h'),
\\
\overline{\sqcap}^R(\overline{\sqcap}^R(h)\ h')\, =
\overline{\sqcap}^R(h)\ \overline{\sqcap}^R(h'),
&
\qquad\overline{\sqcap}^L(h\overline{\sqcap}^L(h'))=
\overline{\sqcap}^L(h)\overline{\sqcap}^L(h').
\end{array}
\end{equation}
These maps also obey the so-called counital properties
\begin{equation}\label{counitalpropertiesofmaps}
h_1\sqcap^R(h_2)=\sqcap^L(h_1)h_2=\overline \sqcap^R(h_2) h_1=
h_2 \overline\sqcap^L(h_1) =h
\end{equation}
and the identities 
\begin{equation}\label{piRproduct}
\begin{array}{ll}
\sqcap^R(\sqcap^R(h)h')=\sqcap^R(hh'),
&
\qquad\sqcap^L(h\sqcap^L(h'))=\sqcap^L(hh'),
\\
\overline\sqcap^R(h\overline\sqcap^R(h'))=
\overline\sqcap^R(hh'),
&
\qquad\overline\sqcap^L(\overline\sqcap^L(h)\ h')=
\overline\sqcap^L(hh'),
\end{array}
\end{equation}
for all $h,h'\in H$. In this work we will need a few more identities from
\cite{BohmEtAll:1999}: 
\begin{equation}\label{idpiLdeltah}
\begin{array}{ll}
h_1\ox \sqcap^L(h_2)= 1_1h\ox 1_2,\qquad\ 
&
\qquad\Delta(\sqcap^L(h))=
\sqcap^L(h)1_1\ox 1_2,
\\
h_1\ox \sqcap^R(h_2)= h1_1\ox \sqcap^R(1_2),
&
\qquad \Delta(\sqcap^R(h))=
1_1\ox \sqcap^R(h)1_2,
\end{array}
\end{equation}
for all $h\in H$. 

A weak bialgebra $H$ is a \emph{weak Hopf algebra} if there exists a
$k$--linear map $S:H\rightarrow H$, called the \textit{antipode}, satisfying
the following axioms for all $h\in H$. 
\begin{equation}\label{antipodeaxioms}
h_1S(h_2)=\sqcap^L(h),\qquad 
S(h_1)h_2=\sqcap^R(h),\qquad 
S(h_1)h_2S(h_3)=S(h).
\end{equation}
By \cite[Theorem 2.10]{BohmEtAll:1999}, the antipode is anti-multiplicative
and anti-comultiplicative. That is, for all $h,h'\in H$, 
\begin{equation}\label{eq:S_a_multiplicative}
S(hh')=S(h')S(h)
\qquad\qquad\textrm{and}\qquad\qquad 
S(h)_1\ox S(h)_2=S(h_2)\ox S(h_1).
\end{equation}
By \cite[Lemma 2.9]{BohmEtAll:1999}, the following identities
hold true. 
\begin{equation}\label{piR&antipode}
\begin{array}{ll}
\sqcap^R S=\sqcap^R\sqcap^L=S\sqcap^L,\qquad
&
\overline{\sqcap}^R S=\sqcap^R=S\overline{\sqcap}^L,
\\
\sqcap^L S=\sqcap^L\sqcap^R=S\sqcap^R,\qquad
&
\overline{\sqcap}^L S=\sqcap^L=S\overline{\sqcap}^R.
\end{array}
\end{equation}
For more on weak bialgebras and weak Hopf algebras, we refer to
\cite{BohmEtAll:1999}. 

\section{Categories of bimonoids.}
\label{sec:cat_bmd}

Let $\mathsf{duo}$ denote the category whose objects are duoidal categories
and whose morphisms are functors which are comonoidal with respect to both
monoidal structures. Consider a functor $M : \mathcal S \to \mathsf{duo}$ from
an arbitrary category $\mathcal S$ to $\mathsf{duo}$. In this section we
associate a category (of some bimonoids) to $M$.

\begin{lemma}\label{lem:bmd_morphism}
Let $X$ and $X'$ be objects of $\mathcal S$ and let $H$ and $H'$ be bimonoids
in $MX$ and $MX'$, respectively.
For a morphism $q:X\to X'$ in $\mathcal S$ and a morphism $Q:(Mq)H \to H'$ in
$MX'$, the following assertions are equivalent. 
\begin{itemize}
\item[{(a)}] The functor $Mq:MX\to MX'$ and the natural transformation
$$
\xymatrix{
(Mq)(-\bullet H)\ar[r]^-{(Mq)^\bullet_2}&
(Mq)(-)\bullet\other (Mq)H\ar[rr]^-{(Mq)(-)\bullet\other Q}&&
(Mq)(-)\bullet\other H'}
$$
constitute a morphism of monoidal comonads $((MX,\circ), (-)\bullet H)\to 
((MX',\circ\other), (-)\bullet\other H')$.
\item[{(b)}] The following diagrams commute, for any objects $A,B$ of $MX$.
\smallskip

\hspace{-2.05cm}
\begin{tabular}{ll}
\tagarray\label{diag1lemma}\qquad
\raisebox{1.5cm}{\xymatrix{
(Mq)H\ar[r]^-Q\ar[d]_-{(Mq)\Delta} & 
H'\ar[dd]^-{\Delta'}\\
(Mq)(H\bullet H)\ar[d]_-{(Mq)_2^{\bullet}} &\\
(Mq)H\bullet\other (Mq)H\ar[r]_-{Q\bullet\other Q}& 
H'\bullet\other H'}}\qquad\qquad\quad&
\tagarray\label{diag2lemma}\qquad\qquad
\raisebox{1.5cm}{\xymatrix{
(Mq)H\ar[r]^-Q\ar[d]_-{(Mq)\epsilon} & 
H'\ar[dd]^-{\epsilon'}\\
(Mq)J\ar[d]_-{(Mq)_0^{\bullet}} &\\
J'\ar@{=}[r]& J'
}}
\end{tabular}

\begin{equation}\label{diag3lemma}
\raisebox{3.3cm}{\xymatrix @C=6pt @R=30pt{
(Mq)((A\bullet H)\circ (B \bullet H))\ar[r]^-{(Mq)\gamma} 
\ar[d]_-{(Mq)^\circ_2}& 
(Mq)((A\circ B)\bullet (H \circ H))\ar[d]^-{(Mq)^\bullet_2}\\ 
(Mq)(A\bullet H)\circ\other (Mq)(B \bullet H)
\ar[d]_-{(Mq)^\bullet_2 \circ\other (Mq)^\bullet_2} & 
(Mq)(A\circ B)\bullet\other (Mq)(H \circ H)
\ar[d]^-{(Mq)^\circ_2\bullet\other (Mq)(H \circ H)} \\ 
((Mq)A\bullet\other (Mq)H)\circ\other ((Mq)B \bullet\other (Mq)H)
\ar[d]_-{\gamma'} & 
((Mq)A\circ\other (Mq)B)\bullet\other (Mq)(H \circ H)
\ar[d]^-{((Mq)A\circ\other (Mq)B)\bullet\other (Mq)\mu}\\ 
((Mq)A\circ\other (Mq)B)\bullet\other ((Mq)H \circ\other (Mq)H)
\ar[d]_-{\raisebox{-10pt}{${}_{((Mq)A\circ\other (Mq)B)
\bullet\other (Q\circ\other Q)}$}}& 
((Mq)A\circ\other (Mq)B)\bullet\other (Mq)H
\ar[d]^-{((Mq)A\circ\other (Mq)B)\bullet\other Q}\\ 
((Mq)A\circ\other (Mq)B)\bullet\other (H'\circ\other H')
\ar[r]_-{\raisebox{-10pt}{${}_{((Mq)A\circ\other (Mq)B)\bullet\other \mu'}$}}& 
((Mq)A\circ\other (Mq)B)\bullet\other H'}} 
\end{equation}

\begin{equation}\label{diag4lemma}
\raisebox{2.2cm}{\xymatrix@C=30pt{
(Mq)I \ar[rr]^-{(Mq)^\circ_0}\ar[d]_-{(Mq)\Delta}&&
I'\ar[dd]^-{\Delta\other} \\
(Mq)(I \bullet I)\ar[d]_-{(Mq)^\bullet_2}\\
(Mq)I \bullet\other (Mq)I \ar[d]_-{(Mq) I \bullet\other (Mq)\eta}&& 
I'\bullet\other I'\ar[d]^-{I'\bullet\other \eta'}\\ 
(Mq)I \bullet\other (Mq)H\ar[r]_-{(Mq) I \bullet\other Q}& 
(Mq) I \bullet\other H'\ar[r]_-{(Mq)^\circ_0\bullet\other H'}&
I'\bullet\other H'}}
\end{equation}
\end{itemize}
\end{lemma}

\begin{proof}
Here and throughout, for brevity, we omit explicitly denoting the
associator isomorphisms.

The diagrams \eqref{diag3lemma} and \eqref{diag4lemma} are identical to
the diagrams in \eqref{comonadmorphism} for the functors $F=Mq$, $T=(-)\bullet
H$ and $T'=(-)\bullet' H'$. So we only need to show that the pair in part (a)
is a morphism of comonads if and only if \eqref{diag1lemma} and
\eqref{diag2lemma} commute.

Assume first that the functor $Mq$ and the natural transformation 
$((Mq)(-)\bullet^{\prime} Q)(Mq)_2^{\bullet}$ constitute a morphism of
comonads. This means commutativity of the diagrams
$$
\xymatrix{
(Mq)(-\bullet H)\ar[r]^-{(Mq)(-\bullet \Delta)}
\ar[dd]_-{(Mq)^\bullet_2}&
(Mq)(-\bullet H\bullet H)\ar[d]^-{(Mq)^\bullet_2}&
(Mq)(-\bullet H)\ar[r]^-{(Mq)(-\bullet \epsilon)}
\ar[dd]_-{(Mq)^\bullet_2}&
(Mq)(-\bullet J)\ar[dd]^-{(Mq)\rho^\bullet}\\
&(Mq)(-\bullet H)\bullet' (Mq)H\ar[d]^-{(Mq)(-\bullet H)\bullet' Q}\\
(Mq)(-)\bullet' (Mq)H\ar[dd]_-{(Mq)(-)\bullet' Q}&
(Mq)(-\bullet H)\bullet' H'
\ar[d]^-{(Mq)^\bullet_2\bullet' H'}&
(Mq)(-)\bullet' (Mq)H\ar[dd]_-{(Mq)(-)\bullet' Q}&
Mq\\
&(Mq)(-)\bullet' (Mq)H\bullet' H'\ar[d]^-{(Mq)(-)\bullet' Q\bullet' H'}\\
(Mq)(-)\bullet' H'\ar[r]_-{\raisebox{-8pt}{${}_{(Mq)(-)\bullet' \Delta'}$}}&
(Mq)(-)\bullet' H'\bullet' H'&
(Mq)(-)\bullet' H'\ar[r]_-{\raisebox{-8pt}{${}_{(Mq)(-)\bullet' \epsilon'}$}}&
(Mq)(-)\bullet' J' \ar[uu]_-{\rho^{\bullet'}_{Mq}}.}
$$
Evaluate the equal pairs of paths around these diagrams at the monoidal unit
$J$. Precompose the resulting morphisms with $((Mq)\lambda^\bullet_H)^{-1}$ in
both cases and postcompose them with $\lambda^{\bullet'}_{H'\bullet'
H'}((Mq)^\bullet_0 \bullet' H'\bullet'H')$ in the case of the first 
diagram and postcompose them with $(Mq)^\bullet_0$ in the case of the second
diagram. Using naturality of the morphisms $\lambda^{\bullet}$,
$\lambda^{\bullet'}$, $\rho^{\bullet'}$ and $(Mq)_2^{\bullet}$,
${\bullet}$--comonoidality of $Mq$, the identity
$\rho^\bullet_J=\lambda^\bullet_J$ (holding true in every monoidal 
category) and functoriality of the monoidal product $\bullet'$, they yield
\eqref{diag1lemma} and \eqref{diag2lemma}, respectively. The converse
implication follows by noting first that $(Mq,(Mq)^\bullet_2)$ is a comonad
morphism from the comonad $(-)\bullet H$ to $(-)\bullet'(Mq)H$ by
coassociativity, counitality and naturality of $(Mq)^\bullet_2$. Second, since
$Q:(Mq)H \to H'$ is a morphism of comonoids, $(MX',(-)\bullet' Q)$ is a
comonad morphism from the comonad $(-)\bullet' (Mq)H$ to $(-)\bullet'
H'$. Then also their composite is a comonad morphism.
\end{proof}

\begin{remark}\label{rem:double_comon}
If the functor $Mq:MX\to MX'$ is double comonoidal in the sense of
\cite[Definition 6.55]{Aguiar&Mahajan}, then the last two diagrams in 
part (b) of Lemma \ref{lem:bmd_morphism} turn out to be equivalent to
the following diagrams. 
$$
\xymatrix{
(Mq)(H\circ H)\ar[r]^-{(Mq)^\circ_2} \ar[d]_-{(Mq)\mu}&
(Mq)H\circ\other(Mq)H \ar[r]^-{Q\circ\other Q}&
H'\circ\other H'\ar[d]^-{\mu'}
&
(Mq)I \ar[r]^-{(Mq)^\circ_0}\ar[d]_-{(Mq)\eta}&
I'\ar[d]^-{\eta'}\\
(Mq)H\ar[rr]_-Q&&
H'
&
(Mq)H\ar[r]_-Q&
H'}
$$
However, in our most important example in Section \ref{sec:wba}, the functors
$Mq:MX\to MX'$ are not double comonoidal. So we need to cope with
the more general situation in Lemma \ref{lem:bmd_morphism}.
\end{remark}

\begin{definition}\label{def:mor_bmd}
Let $M$ be a functor from an arbitrary category $\mathcal S$ to the category
$\mathsf{duo}$ of duoidal categories. The associated category
$\mathsf{bmd}(M)$ is defined to have objects which are pairs, consisting of an
object $X$ of $\mathcal S$ and a bimonoid $H$ in $MX$. Morphisms are pairs
consisting of a morphism $q:X\to X'$ in $\mathcal S$ and a morphism
$Q:(Mq)H\to H'$ in $MX'$, obeying the equivalent conditions in Lemma
\ref{lem:bmd_morphism}. 
\end{definition}

The composite of any morphisms of monoidal comonads is a morphism of
monoidal comonads again, see \cite[Definition 3.3]{Szlachanyi:2003}. 
Hence the composition of morphisms in $\mathsf{bmd}(M)$ is well-defined by
their description in part (a) of Lemma \ref{lem:bmd_morphism}. 

If $\mathcal S$ is the singleton category (having only one object and its
identity morphism), then the functors $M:\mathcal S \to \mathsf{duo}$ are in
bijection with the objects of $\mathsf{duo}$; that is, with the duoidal
categories $M$. As kindly pointed out by the referee, in this case
$\mathsf{bmd}(M)$ is the usual category of bimonoids in the duoidal category
$M$: Its objects are the bimonoids and the morphisms are the morphisms in $M$
which are both morphisms of monoids (w.r.t. $\circ$) and morphisms of
comonoids (w.r.t. $\bullet$); see Remark \ref{rem:double_comon}. 

\begin{remark}\label{rem:symm}
Note that Definition \ref{def:mor_bmd} is one choice of several
symmetric possibilities. With this choice, we obtain the adjunction in
Section \ref{sec:functor_k} and Section \ref{sec:adjoint}. An analogous
definition could be based on the monoidal comonad $((MX,\circ), H\bullet
(-))$. If applied to the functor $\mathsf{span}:\mathsf{set}\to \mathsf{duo}$
in Section \ref{sec:cat}, it would lead to the same category of small
categories. If applied to the functor $\mathsf{bim}(-^e):\mathsf{sfr}\to
\mathsf{duo}$ in Section \ref{sec:wba}, however, it would result in a
different notion of morphism between weak bialgebras (related to that in
Section \ref{sec:wba} by interchanging the roles of the `left' and `right'
subalgebras). This symmetric variant of the category of weak bialgebras admits
a symmetric adjunction with the category of small categories, see also Remark
\ref{rem:symm_restored}.

As a further symmetry, one can change the notion of morphism between duoidal 
categories to functors which are monoidal with respect to both monoidal
structures. Then two symmetric variants of morphisms between
bimonoids can be defined in terms of the induced comonoidal monads $((MX,
\bullet),H\circ (-))$ and $((MX,\bullet),(-)\circ H)$. (Note that while
weak bialgebra is a self-dual structure, its morphisms in Section
\ref{sec:wba} are not. A category of weak bialgebras whose morphisms 
are dual to those in Section \ref{sec:wba} can be obtained by this dual
construction. The possibility of finding a contravariant adjunction to
the category of small categories has not been investigated in this case.) 
\end{remark}

\section{Example: $\mathsf{cat}$ as a category of bimonoids.} 
\label{sec:cat}

\begin{aussage}{\bf The category $\mathsf{span}(X)$ \cite[Example
6.17]{Aguiar&Mahajan}. } 
For any set $X$, a {\em span} over $X$ is a triple $(A,s,t)$
where $A$ is a set and $s,t:A\to X$ is a pair of maps, called the {\em source}
and {\em target} maps, respectively. 
A morphism between the spans $(A,s,t)$ and $(A',s',t')$ over $X$ is a map
$f:A\rightarrow A'$ such that the diagrams
$$
\xymatrix{
A\ar[rr]^-{f}\ar[rd]_-{s} && 
A'\ar[ld]^-{s'}\\
&X&} \qquad
\xymatrix{
A\ar[rr]^-{f}\ar[rd]_-{t} && 
A'\ar[ld]^-{t'}\\
&X&}$$
commute. For brevity, we write $A$ instead $(A,s,t)$, understanding 
that $s$ and $t$ are given. We denote by $\mathsf{span}(X)$ the category of
spans over $X$. For any spans $A$ and $B$, define the sets
\begin{eqnarray*}
A\circ B &:=&\{(a,b)\in A\times B : s(a)=t(b)\}\\
A\bullet B &:=&\{(a,b)\in A\times B: s(a)=s(b) \text{ and } t(a)=t(b)\}.
\end{eqnarray*}
We turn $A\circ B$ and $A\bullet B$ into spans over $X$ by
defining, for $(a,b)\in A\circ B$, 
\begin{equation*}
 s(a,b):=s(b)\qquad \text{and} \qquad t(a,b):=t(a),
\end{equation*}
and for $(a,b)\in A\bullet B$,
\begin{equation*}
 s(a,b):=s(a)=s(b) \qquad \text{and}\qquad t(a,b):=t(a)=t(b).
\end{equation*}
Each one of these operations is functorial and endows the
category $\mathsf{span}(X)$ with a monoidal structure, with the obvious
associators. The unit object $I$ of $(\mathsf{span}(X),\circ)$ is the discrete
span $(X,\mathsf{id},\mathsf{id})$ and the unit object $J$ of
$(\mathsf{span}(X),\bullet)$ is the complete span $(X\times X, p_1,p_2)$ with
$p_1(x,y)=x$ and $p_2(x,y)=y$. Furthermore, $(\mathsf{span}(X),
\circ,I,\bullet,J)$ is a duoidal category with the structure below. Let
$A,B,C,D$ be spans over $X$. The interchange law 
$$
\gamma_{A,B,C,D}:(A\bullet B)\circ (C\bullet D)\rightarrow (A\circ C)\bullet
(B\circ D)
$$
simply sends $(a,b,c,d)$ to $(a,c,b,d)$.
The structure map $\Delta_{I}:I\rightarrow I\bullet I$ is the identity and
$\mu_J:J\circ J\rightarrow J$ and $\tau:I\rightarrow J$ are uniquely
determined since the object $J$ is terminal in the category
$\mathsf{span}(X)$. 

A bimonoid in the duoidal category $\mathsf{span}(X)$ is, equivalently, a
small category (see \cite[Example 6.43]{Aguiar&Mahajan}).
\end{aussage} 

Consider the following functor $\mathsf{span}$ from the category
$\mathsf{set}$ of (small) sets to $\mathsf{duo}$. It sends
a set $X$ to the duoidal category $\mathsf{span}(X)$ above. Regarding its
action on a map of sets $q:X\to X'$, note that $q$ induces a morphism
$\mathsf{span}(q)$ in $\mathsf{duo}$ from $\mathsf{span}(X)$ to
$\mathsf{span}(X')$: The functor $\mathsf{span}(q)$ takes an object
$t:X\leftarrow A \rightarrow X:s$ to $qt:X'\leftarrow A 
\rightarrow X':qs$ and it acts on the morphisms as the identity map. It
is easily seen to be comonoidal with respect to both monoidal structures
$\circ$ and $\bullet$, via
$$
\begin{array}{lcl}
\mathsf{span}(q)_2^{\circ}:A\circ B\to A\circ' B,\quad
&(a,b)&\mapsto (a,b)\\
\mathsf{span}(q)_0^{\circ}:X\to X', \quad 
&x&\mapsto q(x)\\
\mathsf{span}(q)_2^{\bullet}:A\bullet B\to A\bullet' B,\quad
&(a,b)&\mapsto (a,b)\\
\mathsf{span}(q)_0^{\bullet}:X\times X\to X'\times X' \quad 
&(x,y)&\mapsto (q(x),q(y)).
\end{array}
$$

The final aim of this section is to prove the following. 

\begin{theorem}
The category $\mathsf{bmd}(\mathsf{span})$ is isomorphic to the category of
small categories. 
\end{theorem}

\begin{proof}
Since there is exactly one comonoid structure (the `diagonal' one) on
any object $(A,s,t)$ of $(\mathsf{span}(X),\bullet,J)$, it follows that
objects in $\mathsf{bmd(span)}$ are pairs $(X,A)$ of a set $X$ and a monoid
$A$ in $\mathsf{span}(X)$ or, equivalently, a small category $A$ with
object set $X$, see \cite{Aguiar&Mahajan}. 

The morphisms in $\mathsf{bmd(span)}$ are pairs $(q:X\rightarrow 
X',Q:A\rightarrow A')$ of maps for which $qs=s'Q$, $qt=t'Q$ and 
which render commutative 
the four diagrams in Lemma \ref{lem:bmd_morphism}. 
Evaluating these diagrams on elements of the appropriate set, we see that
\eqref{diag1lemma} commutes for any maps $(q:X\rightarrow X',Q:A\rightarrow
A')$; \eqref{diag2lemma} commutes if and only if $Q$ is a morphism of spans,
\eqref{diag3lemma} commutes if and only if $Q$ preserves composition; and 
\eqref{diag4lemma} commutes if and only if $Q$ preserves identity
morphisms. Shortly, these diagrams commute if and only if there is a functor
with object map $q$ and morphism map $Q$.
\end{proof}

Applying the above construction to the restriction of the functor
$\mathsf{span}$ to the full subcategory of finite sets in $\mathsf{set}$, we
obtain the full subcategory $\mathsf{cat}$ of small categories with finitely
many objects.

\section{Example: $\cat{wba}$ as a category of bimonoids.}
\label{sec:wba}

Let $\mathsf{sfr}$ denote the category whose objects are
separable Frobenius (co)algebras (over a given base field $k$), and whose
morphisms are defined as follows. Given separable Frobenius algebras $R$ and
$R'$ with respective Nakayama automorphisms $\theta$ and $\theta'$, a morphism
in $\mathsf{sfr}$ from $R$ to $R'$ is a coalgebra map $f: R \to R'$ such that
$\theta' f = f \theta$. In what follows we construct a functor
$\mathsf{bim}(-{}^e)$ from $\mathsf{sfr}$ to $\mathsf{duo}$. 

In the monoidal category of $R^e:=R\ox R^{op}$--bimodules, we will not 
explicitly denote the associator constraints. However, since we work
simultaneously with various $R^e$--actions, the corresponding unit constraints
will be written out for an easier reading. 

\begin{aussage}\textbf{The first monoidal structure.}
Let $R$ be an object in $\cat{sfr}$ with idempotent Frobenius element $e_i\ox
f_i$, Frobenius functional $\psi:R\rightarrow k$ and Nakayama automorphism
$\theta: R\rightarrow R$. The category $\mathsf{bim}(R^e)$ of $R^e$--bimodules
is monoidal via the monoidal product $ \circ = \ox_{R^e}$, and unit $I = R^e$
with the $R^e$--bimodule structure given by its multiplication as a
$k$--algebra; that is, with the actions 
\begin{equation}\label{eq:I_left}
(s\ox r)(x\ox y)=sx\ox yr\qquad
\textrm{and} \qquad
(x\ox y)(s\ox r)=xs\ox ry.
\end{equation}
Given $R^e$--bimodules $M$ and $N$, the unit constraints are
\begin{eqnarray*}
& \lambda_M^{\circ}: I\circ M\rightarrow M,& 
\qquad (x\otimes y)\circ m\mapsto(x\otimes y)m\\
&\rho_M^{\circ}:M\circ I\rightarrow M,& 
\qquad m\circ (x\otimes y)\mapsto m(x\otimes y).
\end{eqnarray*}
The product $M\circ N$ is an $R^e$--bimodule via the actions
\begin{equation}\label{eq:circ_bimod}
(s\ox r)(m\circ n)=(s\ox r)m\circ n\qquad \textrm{and}\qquad 
(m\circ n)(s\ox r)=m\circ n(s\ox r).
\end{equation}
The canonical $R^e$--bimodule epimorphism 
\begin{equation*}
\pi^{\circ}_{M,N}: M\ox N\rightarrow M\circ N, \qquad m\ox n\mapsto m\circ n
\end{equation*}
is split by 
\begin{equation*}
\iota^{\circ}_{M, N}: M\circ N\rightarrow M\ox N, \qquad m\circ
 n\mapsto m(e_i\ox f_j)\ox (f_i\ox e_j)n. 
\end{equation*}
Thus, $M\circ N$ is isomorphic to the vector subspace (in fact
$R^e$--subbimodule) 
$$
\iota^{\circ}_{M,N}(M\circ N)=M(e_i\ox f_j)\ox (f_i\ox e_j)N
$$
of $M\otimes N$. Alternatively, 
$$
M\circ N\cong
\{\ x\in M\ox N\ : \ \iota^{\circ}_{M,N}\pi^{\circ}_{M,N}x=x\ \}. 
$$

Recall from \cite[Lemma 2.2]{Bohm:2009} that the monoids in this
monoidal category $(\mathsf{bim}(R^e),\circ,I)$ can be identified with pairs
consisting of a $k$--algebra $A$ and a $k$--algebra homomorphism $\eta:R^e\to
A$. Via this identification, the morphisms of monoids correspond to
$k$--algebra homomorphisms $f:A\to A'$ such that $f\eta=\eta'$.

Let $H$ be a weak bialgebra over a field. Then its `right' subalgebra
$R:=\sqcap^R(H)$ is a separable Frobenius algebra and there is an algebra
homomorphism $R^e\to H$, $s\ox r\mapsto s\overline \sqcap^L(r)$. Thus $H$ is
in particular a monoid in $(\mathsf{bim}(R^e),\circ,I)$.
\end{aussage}

\begin{aussage}\textbf{The second monoidal structure.}
As we have seen in the previous paragraph, the multiplication of a weak
bialgebra $H$, with base algebra $R:=\sqcap^R(H)$, is $R^e$--balanced with
respect to the $R^e$--actions $(s\ox r)h(s'\ox r')=s\overline
\sqcap^L(r)hs'\overline \sqcap^L(r')$ on $H$. The comultiplication of $H$
factorizes through another $R^e$--module tensor product with respect to the
twisted $R^e$--actions $(s\ox r)\cdot h\cdot (s'\ox r')= r'\sqcap^L(r)h
s'\overline \sqcap^L(s)$ on $H$ (note the occurrence of the Nakayama
automorphism $\sqcap^R\sqcap^L$ of $R$ in $\sqcap^L=\overline
\sqcap^L\sqcap^R\sqcap^L$). 

By this motivation, for any separable Frobenius algebra $R$, 
define an automorphism functor $F:\mathsf{bim}(R^e)\rightarrow
\mathsf{bim}(R^e)$ as follows. For any $R^e$--bimodule $M$, the underlying
vector space of $F(M)$ is $M$ endowed with the $R^e$--actions 
\begin{equation}
(s\ox r)\·m = (1\ox \theta(r))m(1\ox s) \label{eq:dot_left}\qquad m\·(s\ox r)
= (r\ox 1)m(s\ox 1). 
\end{equation}
This gives the object map of the functor $F$. On morphisms $F$ acts as the
identity map. The (strict) inverse of $F$ sends an $R^e$--bimodule $M$ to the
$R^e$--bimodule whose underlying vector space is $M$, and whose actions are
$$
(s\ox r)\cdotdot m = (1\ox \theta^{-1}(r))m(1\ox s)
\qquad
m\cdotdot (s\ox r) = (r\ox 1)m(s\ox 1), 
$$
where juxtaposition denotes the original untwisted actions. 

Any automorphism functor on a monoidal category can be used to twist the
monoidal structure. In particular, we can use the above functor $F$ to twist
the monoidal category $(\mathsf{bim}(R^e), \circ, I)$ to a new monoidal 
category $(\mathsf{bim}(R^e), \bullet, J)$. The new monoidal product and unit
are 
$$
M\bullet N=F^{-1}(F(M)\circ F(N))\qquad \textrm{and} \qquad J=F^{-1}(I).
$$
The underlying vector space of the $R^e$--bimodule $M\bullet N$ is the tensor
product $F(M)\ox_{R^e} F(N)$; that is, the factor space of
$M\ox N$ with respect to the relations 
\begin{equation}\label{eq:bullet_quotient}
 \{(r\ox 1)m(s\ox 1)\ox n-m\ox (1\ox \theta(r))n(1\ox s)\}.
\end{equation}
The $R^e$--bimodule structure of $M\bullet N$ comes out as 
\begin{equation}\label{eq:bullet_left}
(s\ox r)\cdotdot (m\bullet n)=
(1\ox \theta^{-1}(r))\cdot m\bullet n \cdot (1\ox s)=
(1\ox r)m\bullet (s\ox 1)n
\end{equation}
and, analogously,
\begin{equation}\label{eq:bullet_right}
(m\bullet n)\cdotdot (s\ox r) =m(1\ox r) \bullet n(s\ox 1).\hspace{1cm}
\end{equation}
The $R^e$--bimodule structure of $J=R^e$ is given by the actions 
\begin{equation}\label{eq:LbimodJ}
(s \ox r)\cdotdot (x \ox y) =
x \ox sy\theta^{-1}(r)
\quad \textrm{and}\quad
(x \ox y) \cdotdot (s \ox r) =
rxs \ox y.
\end{equation}
The left and right unit constraints for the monoidal product $\bullet$ are
given by
\begin{equation}\label{eq:lambda_rho_bullet}
\begin{array}{ll}
\lambda_M^{\bullet}:J\bullet M\rightarrow M,
& \qquad (x\otimes y)\bullet m\mapsto 
(x\ox y)\cdot m = 
(1\otimes \theta(y))m(1\otimes x)\\ 
\rho_M^{\bullet}:M\bullet J\rightarrow M,
& \qquad m\bullet (x\otimes y)\mapsto 
m\cdot (x\ox y) = 
(y\otimes 1)m(x\otimes 1).
\end{array}
\end{equation}
The canonical epimorphism
\begin{equation*}
\pi^{\bullet}_{M,N}: M\ox N\rightarrow M\bullet N, \qquad m\ox n\mapsto
m\bullet n 
\end{equation*}
is a homomorphism of $R^e$--bimodules if 
$M \ox N$ is considered as the bimodule $F^{-1}(F(M) \ox F(N))$.
Since $R^e$ is a separable Frobenius algebra, $\pi^{\bullet}_{M,N}$
admits an $R^e$--bimodule section 
\begin{equation}\label{eq:iota_bullet}
 \iota^{\bullet}_{M, N}: M\bullet N\rightarrow M\ox N, \quad
 m\bullet n\mapsto 
m\cdot (e_j\ox f_i)\ox (f_j\ox e_i)\cdot n = 
(e_i\ox 1)m(e_j\ox 1)\ox (1\ox f_i)n(1\ox f_j).\\ 
\end{equation} 
Thus $M\bullet N$ is isomorphic to the $R^e$--subbimodule 
$$
\iota^{\bullet}_{M,N}(M\bullet N)=(e_i\ox 1)M(e_j\ox 1)\ox (1\ox
f_i)N(1\ox f_j) 
$$
of $M \ox N$ (with the aforementioned structure). Alternatively, 
$$
 M\bullet N\cong\{\ x\in M\ox N\ : \
 \iota^{\bullet}_{M,N}\pi^{\bullet}_{M,N}(x)=x\ \}. 
$$
Note that by \eqref{eq:lambda_rho_bullet}, \eqref{eq:iota_bullet},
\eqref{eq:LbimodJ}, \eqref{Nakayamaexplicitform}, \eqref{eq:Frob4} and
\eqref{eq:Frob-1}, the following diagrams commute, for any $R^e$--bimodule $M$.
\begin{equation}\label{eq:lambda_rho_iota}
\raisebox{.8cm}{\xymatrix@C=40pt{
J\bullet M\ar[r]^-{\lambda^\bullet_M}\ar[d]_-{\iota^\bullet_{J,M}}&
M\ar@{=}[d]&&
M\bullet J\ar[r]^-{\rho^\bullet_M}\ar[d]_-{\iota^\bullet_{M,J}}&
M\ar@{=}[d]\\
J\ox M \ar[r]_-{\psi\ox \psi \ox M}&
M&&
M\ox J \ar[r]_-{M\ox \psi\ox \psi}&
M}}
\end{equation}
\end{aussage}

\begin{theorem}\label{thm:bim_obj}
$(\mathsf{bim}(R^e), \circ,I,\bullet,J)$ possesses the structure of a duoidal
category. 
\end{theorem}

\begin{proof}
Given $R^e$--bimodules $A, B, C, D$, we define the interchange law
\eqref{duoidalinterchangelaw} by 
\begin{equation}\label{bim(Re)interchangelaw}
\gamma ((a \bullet b) \circ (c \bullet d)) = (a(e_i \ox 1) \circ c) \bullet (b
\circ (1 \ox f_i)d) 
\end{equation}
and the morphisms in \eqref{duoidalmorphisms} by
\begin{equation}\label{bim(Re)morphisms}
\begin{array}{rlll}
 \tau&:& I\rightarrow J, & (x\otimes y)\mapsto (y f_i\otimes x e_i)\\
\mu_{J}&:& J \circ J\rightarrow J, & (x\otimes y)\circ(p\otimes q)\mapsto
\psi(xq)p\otimes y\\ 
\Delta_{I}&:& I\rightarrow I\bullet I, & (x\otimes y)\mapsto (1\otimes
y)\bullet (x\otimes 1).
\end{array}
\end{equation}

In order to show that $\gamma$ is well defined, we should check
that the map 
$$
\widetilde{\gamma}:
A \ox B \ox C \ox D \to
(A \circ C) \bullet (B \circ D), \qquad
a \ox b \ox c \ox d \mapsto
(a(e_i \ox 1) \circ c) \bullet (b \circ (1 \ox f_i)d)
$$
is $R^e$--balanced in all of the three occurring tensor products. 
This follows by computations of the kind 
\begin{eqnarray*}
\widetilde{\gamma}[a \cdot (
1\ox r) \ox b \ox c \ox d] & = & 
\left[(r \ox 1)a
(e_i \ox 1) \circ c\right] \bullet [b \circ (1 \ox f_i)d]\\ 
& = & [(a(e_i \ox 1) \circ c) \cdot (1 \ox r)] \bullet [b
\circ (1 \ox f_i)d] \\ 
& = & [a(e_i \ox 1) \circ c] \bullet [(1 \ox r) \cdot (b
\circ (1 \ox f_i)d)] \\ 
& = & [a(e_i \ox 1) \circ c] \bullet [(1 \ox \theta(r))b
\circ (1 \ox f_i)d] \\ 
& = & \widetilde{\gamma}[a \ox (
1\ox r)\cdot b \ox c \ox d],
\end{eqnarray*}
and similarly in the other five cases. By similar steps one can also see that
$\mu_J$ is well defined and that $\gamma$, $\tau$, $\mu_J$ and $\Delta_I$ are
morphisms of $R^e$--bimodules. For instance, let us show that $\tau$ is a
morphism of left $R^e$--modules; the compatibilities of the other maps with
the respective $R^e$--actions are checked similarly. 
\begin{eqnarray*}
(s\ox r)\cdotdot \tau(x \ox y) & = & 
(s \ox r) \cdotdot(yf_i \ox xe_i) 
\stackrel{\eqref{eq:LbimodJ}}= 
yf_i \ox sxe_i\theta^{-1}(r) \\
& \stackrel{\eqref{eq:Frob2}}= & 
yrf_i \ox sxe_i = \tau(sx \ox yr) = \tau\left((s \ox r)(x\ox y)\right)
\end{eqnarray*}

We turn to checking the compatibility between both monoidal structures. This
amounts to showing that the just defined maps satisfy the associativity,
unitality and compatibility of units conditions from Section
\ref{duoidalcatdefinition}.  
The computations are fairly straightforward so we only illustrate them on some
chosen examples. For example, coassociativity of $\Delta_I$ and associativity
of $\mu_J$ are obvious. The (left) counitality of $\Delta_I$ and the (left)
unitality of $\mu_J$ are checked by the computations 
$$
\xymatrix@C=38pt{
x\ox y \ar@{|->}[r]^-{\Delta_I}\ar@{=}[d]_(.4){\eqref{eq:Frob4}}
_(.6){\eqref{separability}}&
(1\ox y)\bullet (x\ox 1)\ar@{|->}[d]^-{\tau\bullet \mathsf{id}}&
(x\ox y)\circ (u\ox v)\ar@{|->}[r]^-{\tau\circ\mathsf{id}}
\ar@{|->}[d]_-{\lambda_J^{\circ}}&
(yf_i\ox xe_i)\circ (u\ox v)\ar@{|->}[d]^-{\mu_J}\\
x\ox y f_j\theta(e_j)&
(yf_j\ox e_j)\bullet(x\ox 1)\ar@{|->}[l]^-{\lambda^{\bullet}_{I}}&
u\ox xv\theta^{-1}(y)\ar@{=}[r]_-{\eqref{eq:Frob-1} \eqref{eq:Frob2}}&
\psi(yf_iv)u\ox xe_i,}
$$
for any $x,y,u,v\in R$.
Commutativity of \eqref{diag:1associativity} and
\eqref{diag:2associativity} is immediate from \eqref{eq:circ_bimod},
\eqref{eq:bullet_left} and \eqref{eq:bullet_right}. The following
computation, for any $R^e$--bimodules $A$ and $B$, and any $a\in A$ and $b\in
B$, proves the commutativity of \eqref{diag:unit1}. Commutativity of
\eqref{diag:unit2} is proven analogously. 
$$
\xymatrix@R=5pt{
a\bullet b\ar@{|->}[r]^-{(\lambda_{A\bullet B}^{\circ})^{-1}}
\ar@{|->}[dd]_-{(\lambda_A^{\circ})^{-1}\bullet (\lambda_B^\circ)^{-1}}&
(1\ox 1)\circ (a\bullet b)\ar@{|->}[dddd]^-{\Delta_{I}\circ \mathsf{id}}\\
\\
((1\ox 1)\circ a)\bullet ((1\ox 1)\circ b)\ar@{=}[d]\\
((1\ox 1)\circ a)\cdot (1\ox e_j)\bullet (1\ox \theta^{-1}(f_j))\cdot 
((1\ox 1)\circ b)\ar@{=}[d]\\
((e_j\ox 1)\circ a)\bullet ((1\ox 1)\circ(1\ox f_j) b)&
((1\ox 1)\bullet (1\ox 1))\circ(a\bullet b)\ar@{|->}[l]^-\gamma}
$$
Reading from the top to the bottom, the first equality at the bottom
left corner follows by \eqref{eq:Frob4} and \eqref{separability}. The
computation below, for any $R^e$--bimodules $A$ and $B$, and any $a\in
A$ and $b\in B$, verifies the commutativity of \eqref{diag:unit3}, and
\eqref{diag:unit4} is checked analogously. 
$$
\xymatrix{
a\circ b
\ar@{|->}[rr]^-{(\lambda_A^{\bullet})^{-1}\circ (\lambda_B^{\bullet})^{-1}}
\ar@{|->}[d]_-{(\lambda_{A\circ B}^{\bullet})^{-1}}&&
((1\ox 1)\bullet a)\circ ((1\ox 1)\circ b)
\ar@{|->}[d]^-\gamma\\
(1\ox 1)\bullet (a\circ b)\ar@{=}[r]_-{\eqref{eq:Frob-1}}&
\psi(e_i)(1\ox 1)\bullet (a\circ (1\ox f_i)b)&
((e_i\ox 1)\circ (1\ox 1))\bullet (a\circ (1\ox f_i)b)
\ar@{|->}[l]^-{\mu_J\bullet \mathsf{id}}}
$$
\end{proof}

\begin{remark}\label{rem:sF}
For a separable Frobenius (co)algebra $R$, and $R^e$--bi(co)modules $M$ and $N$,
the $R$--(co)module tensor product
$$
M\ox N/\{m(s\ox 1)\ox n- m\ox n(1\ox s)\} \cong 
\{m(e_j\ox 1)\ox n(1\ox f_j)\}
$$
is an $R$--bimodule via
$$
r[m(e_j\ox 1)\ox n(1\ox f_j)]r'=
(r\ox 1)m(e_j\ox 1)\ox (1\ox r')n(1\ox f_j).
$$
The monoidal product $M\bullet N$ in Theorem \ref{thm:bim_obj} can be
interpreted as the center of this bimodule. That is, $\bullet$ is isomorphic
to the so-called {\em Takeuchi product} over $R$ \cite{Takeuchi}. 

The Takeuchi product of $R^e$--bimodules is defined for {\em any ring}
$R$. However, at this level of generality it does not define a monoidal
product on the category of $R^e$--bimodules (only a lax monoidal one, see
\cite{Day&Street:1997}). It is a consequence of the separable Frobenius
structure of $R$ that allows us to write the Takeuchi product over it
as a (co)module tensor product, what is more, as a split (co)equalizer. 
\end{remark}

\begin{remark}
For any commutative algebra $R$ over a field $k$, a duoidal category
$\mathsf{bim}(R)$ of $R$--bimodules was constructed in \cite[Example
6.18]{Aguiar&Mahajan}. Although the constructions in \cite[Example
6.18]{Aguiar&Mahajan} and in the current section are similar in flavor,
they yield inequivalent categories for a commutative separable Frobenius
$k$--algebra $R$ (in which case both can be applied). Indeed, an equivalence
$\mathsf{bim}(R) \cong \mathsf{bim}(R^e)$ would imply the Morita equivalence
of $R^e$ and $R^e\ox R^e$; hence $R^e\cong R\cong k$. To say a bit more about
the relationship between the categories $\mathsf{bim}(R)$ and
$\mathsf{bim}(R^e)$, let $R$ be a commutative separable Frobenius
$k$--algebra. Any $R$--bimodule $M$ with left and right actions $r\ox m\mapsto 
r\triangleright m$ and $m\ox r \mapsto m\triangleleft r$ can be regarded as an
$R^e\cong R\ox R$--bimodule putting $(s\ox r)m:=r\triangleright m\triangleleft
s =:m(s\ox r)$. This is the object map of a fully faithful embedding (acting
on the morphisms as the identity map) from the category $\mathsf{bim}(R)$ in 
\cite[Example 6.18]{Aguiar&Mahajan} to the category $\mathsf{bim}(R^e)$ in
Theorem \ref{thm:bim_obj} --- but it is not an equivalence. It is strict
monoidal with respect to the monoidal products $\diamond$ in \cite[Example
6.18]{Aguiar&Mahajan} and $\circ$ in Theorem \ref{thm:bim_obj} --- but not
with respect to $\star$ in \cite[Example 6.18]{Aguiar&Mahajan} and $\bullet$
in Theorem \ref{thm:bim_obj}. In fact, it takes the monoidal product $\star$
to $\bullet$ but it does not preserve its monoidal unit. The image of the
$\star$--monoidal unit $R$ in \cite[Example 6.18]{Aguiar&Mahajan} does not
serve as a $\bullet$--monoidal unit in our $\mathsf{bim}(R^e)$, while our
$\bullet$--monoidal unit $R^e$ does not lie in the image of the above embedding
$\mathsf{bim}(R) \to \mathsf{bim}(R^e)$.
\end{remark}

\begin{remark}
Recall (from \cite[Appendix C.5.3]{Aguiar&Mahajan}) that for a commutative
$k$--algebra $R$, the duoidal category $\mathsf{bim}(R)$ in \cite[Example
6.18]{Aguiar&Mahajan} arises via the so-called `looping principle'. This means
the following. If $(\mathcal V,\times, \mathbbm{1})$ is a monoidal 2-category
and $\mathcal C$ is a $\mathcal V$--enriched bicategory, then for any object
$R$ of $\mathcal C$, the hom object $\mathcal C(R,R)$ is a pseudo-monoid in
$\mathcal V$. By \cite[Appendix C.2.4]{Aguiar&Mahajan}, duoidal categories can
be regarded as pseudo-monoids in the monoidal 2-category $\mathsf{coMon}$ of
monoidal categories, comonoidal functors and comonoidal natural
transformations (with monoidal structure provided by the Cartesian 
product). So via the looping principle, hom objects in a
$\mathsf{coMon}$--enriched bicategory are duoidal categories. 

Below we claim that also the duoidal category $\mathsf{bim}(R^e)$ in Theorem
\ref{thm:bim_obj} can be obtained via the looping principle. For this purpose, 
we sketch the construction of a $\mathsf{coMon}$--enriched bicategory
$\mathcal C$ whose objects are separable Frobenius $k$--algebras, and
for any object $R$, $\mathcal C(R,R)\cong \mathsf{bim}(R^e)$. For any
separable Frobenius $k$--algebras $S$ and $R$, let $\mathcal C (R,S)$ be the
category of $R^e$-$S^e$--bimodules. As in \eqref{eq:dot_left}, we can regard
any $R^e$-$S^e$--bimodule $M$ as an $S\ox R$--bimodule via the actions 
$$
(s\ox r)\cdot m\cdot(s'\ox r')=(r'\ox \theta(r))m(s'\ox s).
$$
Hence $\mathcal C (R,S)$ is a monoidal category via the $S\ox R$--module
tensor product 
$$
M\bullet N:=M\ox N/ \{(r\ox 1)m(s\ox 1)\ox n-m\ox (1\ox \theta(r))n(1\ox s)\},
$$
cf. \eqref{eq:bullet_quotient}. The product $M\bullet N$ is an
$R^e$-$S^e$--bimodule as in (\ref{eq:bullet_left}-\ref{eq:bullet_right}). The
monoidal unit is $R\ox S$ with the actions $(r\ox r')(x\ox y)(s\ox
s')=rx\theta^{-1}(r')\ox s'ys$ (which becomes isomorphic to the
$R^e$--bimodule $J$ in \eqref{eq:LbimodJ} if $S=R$). For any separable
Frobenius $k$--algebra $R$, there is a comonoidal functor $I_R$ from the
singleton category $\mathbbm 1$ to $\mathcal C(R,R)$, sending the single
object of $\mathbbm 1$ to the $R^e$--bimodule $I$ in \eqref{eq:I_left}. Its
comonoidal structure is given (up-to isomorphism) by the $R^e$--bimodule maps
$\tau:I\to J$ and $\Delta_I:I\to I \bullet I$ in
\eqref{bim(Re)morphisms}. Coassociativity and counitality of this comonoidal
functor follows by coassociativity and counitality of $\Delta_I$. Furthermore,
for any separable Frobenius $k$--algebras $S$, $R$ and $T$, there is a
comonoidal functor $\circ_{S,R,T}:\mathcal C(S,R) \times \mathcal C(R,T) \to
\mathcal C(S,T)$ given by the usual $R^e$--module tensor product. Its
comonoidal structure is given by the maps 
$$
\begin{array}{ll}
(S\ox R)\circ (R\ox T)\to S\ox T,\qquad
&(s\ox r)\circ (r'\ox t)\mapsto \psi(rr')(s\ox t) 
\quad \textrm{and}\\ 
(A\bullet B)\circ (C\bullet D)\to (A\circ C)\bullet (B\circ D),\qquad
&(a\bullet b)\circ (c\bullet d)\mapsto 
(a(e_i\ox 1) \circ c)\bullet (b\circ (1\ox f_i)d),
\end{array}
$$
for any $S^e$-$R^e$--bimodules $A$ and $B$ and $R^e$-$T^e$--bimodules $C$ and
$D$ (compare them with $\mu_J$ in \eqref{bim(Re)morphisms} and $\gamma$ in
\eqref{bim(Re)interchangelaw}). These maps are checked to be bimodule maps in
the same way as $\mu_J$ and $\gamma$ are in the proof of Theorem
\ref{thm:bim_obj}. Naturality of the binary part is immediate. Coassociativity
and counitality of the comonoidal functor $\circ_{S,R,T}$ is verified by the same
computations used to check that \eqref{diag:2associativity},
\eqref{diag:unit3} and \eqref{diag:unit4} hold in the proof of Theorem
\ref{thm:bim_obj}. The unitors and the associator for the module tensor
product $\circ$ give rise to 2-cells 
$$
\xymatrix{
\mathcal C(R,S)\times \mathbbm 1\ar[d]_-{\mathcal C(R,S)\times I_S}
\ar@{}[rd]|-{\Rightarrow}&
\mathcal C(R,S)\ar[l]_-{\cong}\ar@{=}[d]\ar[r]^-\cong
\ar@{}[rd]|-{\Leftarrow}&
\mathbbm 1 \times \mathcal C(R,S)\ar[d]^-{I_R \times \mathcal C(R,S)}\\
\mathcal C(R,S)\times \mathcal C(S,S)\ar[r]_-{\circ_{R,S,S}}&
\mathcal C(R,S)&
\mathcal C(R,R)\times \mathcal C(R,S)\ar[l]^-{\circ_{R,R,S}}\\
(\mathcal C(Z,R)\times \mathcal C(R,S))\times \mathcal C(S,T)
\ar[rr]^-\cong \ar@{}[rrdd]|-{\Rightarrow}
\ar[d]_-{\circ_{Z,R,S} \times \mathcal C(S,T)}&&
\mathcal C(Z,R)\times (\mathcal C(R,S)\times \mathcal
C(S,T))\ar[d]^-{\mathcal C(Z,R)\times \circ_{R,S,T}}\\
\mathcal C(Z,S)\times \mathcal C(S,T)\ar[d]_-{\circ_{Z,S,T}}&&
\mathcal C(Z,R)\times \mathcal C(R,T)\ar[d]^-{\circ_{Z,R,T}}\\
\mathcal C(Z,T)\ar@{=}[rr]&&
\mathcal C(Z,T)}
$$
in $\mathsf{coMon}$, for any separable Frobenius algebras $S,R,Z,T$. Indeed,
they are shown to be comonoidal natural transformations by computations
similar to those verifying the associativity and the unitality of $\mu_J$ and
the validity of \eqref{diag:1associativity}, \eqref{diag:unit1} 
and \eqref{diag:unit2} in the proof of Theorem \ref{thm:bim_obj}. They
clearly obey the Mac Lane type coherence conditions. This proves that 
$\mathcal C$ is a $\mathsf{coMon}$--enriched bicategory hence
$\mathcal C(R,R)\cong \mathsf{bim}(R^e)$ is a duoidal category. 
\end{remark}

\begin{aussage}\textbf{The functor $\mathsf{bim}(-{}^e)$}. 
Let $R$ and $R'$ be separable Frobenius (co)algebras. Associated to any 
coalgebra homomorphism $q:R\to R'$, there is a functor
$\mathsf{bim}(q^e):\mathsf{bim}(R^e) \to \mathsf{bim}(R^{\prime e})$ (where
$q^e : R^e \to R'^e$ is defined by $q^e(s \ox r) = q(s) \ox q(r)$). It acts on
the morphisms as the identity map. It takes an $R^e$--bi(co)module $P$ with
coactions $\lambda:P\to R^e\ox P$ and $\rho: P \to P\ox R^e$ to the $R^{\prime
e}$--bi(co)module $P$ with the coactions $(q^e\ox P)\lambda$ and $(P\ox
q^e)\rho$. The $R^{\prime e}$--actions on $P$ are induced from the
$R^e$--actions by the dual forms of $q$; that is, by the algebra maps 
$$
\widetilde q: R'\to R, \ r'\mapsto \psi'(r'q(e_i))f_i
\qquad \textrm{and}\qquad
\hat q: R'\to R, \ r'\mapsto e_i\psi'(q(f_i)r')
$$
as 
$$
(r'\ox s')p(u'\ox v')=
(\widetilde q(r')\ox \hat q(s'))p(\hat q(u') \ox \widetilde q(v')),
\quad \textrm{for} \ p\in P,\ r',s',u',v'\in R'.
$$
Note that
\begin{equation}\label{eq:qhat_tr}
\qquad \hat q(e'_i)\ox f'_i=e_j\ox q(f_j)
\qquad \textrm{and}\qquad
\qquad e'_i\ox \widetilde q(f'_i)=q(e_j)\ox f_j.
\end{equation}
The maps $\widetilde q$ and $\hat q$ are equal if and only if $q$
commutes with the Nakayama automorphisms; that is,
$\theta'q=q\theta$. 
\end{aussage}

\begin{proposition}\label{prop:bim_mor}
Let $R$ and $R'$ be separable Frobenius (co)algebras and $q:R\to R'$ be a
coalgebra homomorphism which commutes with the Nakayama automorphisms. The
induced functor $\mathsf{bim}(q^e):\mathsf{bim}(R^e)\to \mathsf{bim}(R^{\prime
 e})$ is comonoidal with respect to both monoidal structures. 
\end{proposition}

\begin{proof}
The coalgebra homomorphisms $q:R\rightarrow R'$ are in bijective
correspondence with the algebra homomorphisms $\widetilde{q}:R'\rightarrow R$
via transposition (or duality)
\begin{eqnarray*}
q\mapsto \widetilde{q}=\psi'(-q(e_i))f_i\qquad 
\widetilde{q}\mapsto q=e'_i\psi(\widetilde{q} (f'_i)-). 
\end{eqnarray*}
In particular, the Nakayama automorphism and its dual $\widetilde{\theta}$ 
satisfy 
\begin{equation*}
 \widetilde{\theta}(r)=\psi(r\theta(e_i))f_i=\psi(r
 f_i)e_i=\theta^{-1}(r), 
\end{equation*}
cf. \eqref{eq:Frob4} and \eqref{Nakayamaexplicitform}.
Thus the assumption $\theta'q=q\theta$ can be written equivalently as
$\widetilde{q}\theta'=\theta \widetilde{q}$. 

The candidate for the binary part of the comonoidal structure with
respect to $\circ$ is the ($R^e$--bimodule) map 
\begin{equation*}
\mathsf{bim}(q^e)_2^{\circ}: \mathsf{bim}(q^e)(M\circ N)\to
\mathsf{bim}(q^e)M\circ' \mathsf{bim}(q^e)N, \qquad 
m\circ n\mapsto m(e_i\ox f_j)\circ' (f_i\ox e_j)n.
\end{equation*}
It is evidently coassociative. The nullary part of the $\circ$--comonoidal
structure is 
\begin{equation*}
\mathsf{bim}(q^e)^\circ_0=q^e :
R^e\rightarrow R^{ \prime e},\qquad 
x\ox y \mapsto q(x)\ox q(y).
\end{equation*}
Its $R'^e$--bimodule map property; that is, 
\begin{equation*} 
 sq(x)s'\ox rq(y)r'=q(\widetilde{q}(s) x \widetilde{q}(s'))\ox
 q(\widetilde{q}(r) y \widetilde{q}(r')) 
\end{equation*}
is proven by 
\begin{equation}\label{eq:qqtilde}
\begin{array}{rl}
q(\widetilde{q}(r')x\widetilde{q}(s'))
=& \psi'(r'q(e_i))q(f_ixe_j)\psi'(q(f_j)s')
=\psi'(r' q(e_i))q(f_ix)e^{\prime}_k \psi^{\prime}(f^{\prime}_k s') \\
=& \psi'(r' q(e_i))q(f_ix)s'
=\psi'(r'e^{\prime }_l)f^{\prime}_lq(x)s'
=r'q(x)s',
\end{array}
\end{equation}
for all $r',s'\in R'$, where in the second and the penultimate equalities we
used that $q$ is comultiplicative. Right counitality; that is, 
commutativity of 
$$
\xymatrix@R=2pt@C=10pt{
\mathsf{bim}(q^e)M \ar[dd]_-{\mathsf{bim}(q^e)M}
\ar[r]^-{\raisebox{6pt}{${}_{\mathsf{bim}(q^e)((\rho_M^{\circ})^{-1})}$}} &
\mathsf{bim}(q^e)(M\circ I)\ar[dddd]^-{\mathsf{bim}(q^e)_2^{\circ}}&
m\ar@{|->}[r]\ar@{=}[dd]&
m\circ (1\ox 1)\ar@{|->}[dddd]\\
\\
\mathsf{bim}(q^e)M&&
m(e_i\ox f_j)(\widetilde q q(f_i)\ox \widetilde qq(e_j))\\
\\
\mathsf{bim}(q^e)M\circ' I'\ar[uu]^-{\rho_{\mathsf{bim}(q^e)M}^{\circ'}} &
\mathsf{bim}(q^e)M\circ' \mathsf{bim}(q^e)I
\ar[l]^-{\raisebox{-7pt}{${}_{\mathsf{bim}(q^e)M\circ' \mathsf{bim}(q^e)_0^{\circ}}$}}&
m(e_i\ox f_j)\circ' (q(f_i)\ox q(e_j))\ar@{|->}[uu]&
m(e_i\ox f_j)\circ' (f_i\ox e_j)\ar@{|->}[l]
}
$$
follows from
\begin{equation}\label{eq:qqtilde1}
e_j\widetilde q q(f_j)
\stackrel{\eqref{eq:qhat_tr}}=
\widetilde q(e'_i)\widetilde q(f'_i)=\widetilde
q(e'_if'_i)=\widetilde q(1')=1 
\quad \textrm{and}\quad 
\widetilde q q(e_j)f_j
\stackrel{\eqref{eq:qhat_tr}}=
\widetilde q(e'_i)\widetilde q(f'_i)=\widetilde
q(e'_if'_i)=\widetilde q(1')=1, 
\end{equation}
where in both cases, in the second and the last equalities we used that
$\widetilde q$ is an algebra homomorphism.
A similar computation shows counitality on the other side.

The binary part of the $\bullet$--comonoidal structure is given by the
$R^e$--bimodule maps 
\begin{equation*}
\mathsf{bim}(q^e)_2^{\bullet}: \mathsf{bim}(q^e)(M\bullet N) \to
\mathsf{bim}(q^e)M\bullet^{\prime} \mathsf{bim}(q^e)N, \qquad 
m\bullet n\mapsto (e_i\ox 1)m(e_j\ox 1)
 \bullet' (1 \ox f_i)n(1 \ox f_j). 
\end{equation*}
Its coassociativity is obvious. The nullary part is given by
$\mathsf{bim}(q^e)_0^{\bullet} =q^e:R^e\rightarrow R'^e$. Its
$R'^e$--bilinearity; that is, 
\begin{equation*}
 r'q(x)s'\ox sq(y)\theta'^{-1}(r)=q(\widetilde q(r')x\widetilde{q}(s'))\ox
 q(\widetilde{q}(s)y\theta^{-1}\widetilde{q}(r)) 
\end{equation*}
follows by 
\eqref{eq:qqtilde} and
 $\widetilde{q}\theta'=\theta\widetilde{q}$.
Right counitality; that is, commutativity of
\begin{equation*}
\scalebox{.98}{
\xymatrix@R=2pt@C=7pt{
\mathsf{bim}(q^e)M \ar[dd]_-{\mathsf{bim}(q^e)M}
\ar[r]^-{\raisebox{5pt}{${}_{\mathsf{bim}(q^e)((\rho_M^{\bullet})^{-1})}$}} &
\mathsf{bim}(q^e)(M\bullet J)\ar[dddd]^-{\mathsf{bim}(q^e)_2^{\bullet}}&
m\ar@{|->}[r]\ar@{=}[dd]&
m\bullet (1\ox 1)\ar@{|->}[dddd]\\
\\
\mathsf{bim}(q^e)M
&& (\widetilde qq(e_i)f_i\ox 1)m(e_j\widetilde qq(f_j)\ox 1)\\
\\
\mathsf{bim}(q^e)M\bullet' J'\ar[uu]^-{\rho_{\mathsf{bim}(q^e)M}^{\bullet'}} & 
\mathsf{bim}(q^e)M\bullet' \mathsf{bim}(q^e)J
\ar[l]^-{\raisebox{-7pt}{
${}_{\mathsf{bim}(q^e)M\bullet' \mathsf{bim}(q^e)_0^{\bullet}}$}}&
(f_i\!\ox\! 1)m(e_j\!\ox\! 1)\!\bullet'\!
(q(f_j)\! \ox\! q(e_i)) \ar@{|->}[uu]&
(f_i\!\ox\! 1)m(e_j\!\ox\! 1)\!\bullet' \!
(f_j \! \ox \! e_i) \ar@{|->}[l]}}
\end{equation*}
follows by \eqref{eq:qqtilde1} and similarly on the other side.
\end{proof}

By Theorem \ref{thm:bim_obj} and Proposition \ref{prop:bim_mor}
there is a functor $\mathsf{bim}(-{}^e):\mathsf{sfr}\to \mathsf{duo}$. Our 
next aim is to describe the corresponding category
$\mathsf{bmd}(\mathsf{bim}(-{}^e))$ as a category of weak bialgebras over
$k$. We begin with identifying in the next two paragraphs the objects of
$\mathsf{bmd}(\mathsf{bim}(-{}^e))$ with weak bialgebras; that is, the
bimonoids in $\mathsf{bim}(R^e)$ with weak bialgebras of `right' subalgebras
isomorphic to $R$.

\begin{punto}\textbf{From weak bialgebras to bimonoids.}\label{frombimtowba} 
Let $(H,\mu,\eta,\Delta,\epsilon)$ be a weak bialgebra over a field $k$. By
\cite[Proposition 4.2]{Schauenburg:2003}, 
$R := \sqcap^R (H)$ is a separable Frobenius algebra with $1_1 \ox
\sqcap^R(1_2) \in R \ox R$ as separability idempotent, and the restriction of
the counit $\epsilon$ to $R$ as the corresponding Frobenius functional. In
this case, the Nakayama automorphism and its inverse (see
\eqref{Nakayamaexplicitform}) are the (co)restrictions of $\sqcap ^R
{\sqcap}^L$ and $\overline{\sqcap}^R\overline{\sqcap}^L$ to $R$,
respectively. In this paragraph we equip $H$ with the structure of a
bimonoid in $\mathsf{bim}(R^e)$. 

First we construct on $H$ a monoid structure in $\mathsf{bim}(R^e)$. By
\cite[Lemma 2.2]{Bohm:2009}, this amounts to the construction of an algebra
homomorphism $R^e\to H$: Since $\sqcap^R(H)$ and $\overline{\sqcap}^L (H)$ are
commuting subalgebras in $H$, and since $\overline{\sqcap}^L$ restricts to an
anti-algebra isomorphism between them (cf. \cite[Proposition
1.18]{BohmEtAll:}), it follows that the map $\widetilde{\eta} : R^e \to H$
defined as $\widetilde{\eta}(s \ox r) = s\overline{\sqcap}^L (r)$ is a desired
homomorphism of algebras. It induces an $R^e$--bimodule structure on $H$. We
denote the resulting actions by juxtaposition. By virtue of \cite[Lemma
2.2]{Bohm:2009}, the multiplication $\mu$ factorizes through an
$R^e$--bilinear associative multiplication $\widetilde{\mu} : H \circ H \to H$
with unit $\widetilde{\eta}$, so that $(H,\widetilde{\mu},\widetilde{\eta})$ 
has a structure of monoid in $\mathsf{bim}(R^e)$. 

In order to equip $H$ with the structure of a comonoid in $\mathsf{bim}(R^e)$,
note that $\Delta:H \to H\ox H$ factorizes through $H\bullet H$ (via the
inclusion $\iota^\bullet_{H,H}:H\bullet H \to H\ox H$). That is, for any $h\in
H$,
\begin{equation}\label{eq:Delta_fac}
\begin{array}{rl}
\Delta(h)=& h_1\ox h_2
= 
1_1h_{ 1} 1_{1'}\ox 1_2h_21_{2'}
=
1_1h_{ 1} 1_{1'}\ox\overline{\sqcap}^L\sqcap^R(1_2)h_2
\overline{\sqcap}^L\sqcap^R(1_{2'})\\ 
=&
(1_1\ox 1)h_1(1_{1'}\ox 1)\ox (1\ox \sqcap^R(1_2))h_2 (1\ox \sqcap^R(1_{2'})) 
=
\iota^\bullet_{H,H}(h_1\bullet h_2),
\end{array} 
\end{equation}
where \eqref{piLbarpiR=piLbar} and \eqref{delta1} has been used. As the
comultiplication for the bimonoid associated to the weak bialgebra $H$,
consider the corestriction $\widetilde{\Delta}:H \to H\bullet H$ of
$\Delta$. As the counit, put
$\widetilde{\epsilon}=(\sqcap^R\ox\overline{\sqcap}^R)\Delta=  
(\sqcap^R\ox\overline{\sqcap}^R)\Delta^{op}:H \to R\ox R$. 
(The two forms are equal, indeed, since for any $h\in H$, 
\begin{eqnarray*}
 \sqcap^R(h_1)\ox \overline{\sqcap}^R(h_2) &=&
1_1\epsilon(h_11_2)\ox 1_{1'}\epsilon(1_{2'}h_2)
 =
1_1\ox 1_{1'}\epsilon(1_{2'}h1_2)\\ &=& 
1_1\epsilon(h_21_2)\ox 1_{1'}\epsilon(1_{2'}h_1)
= 
\sqcap^R(h_2)\ox \overline{\sqcap}^R(h_1).)
\end{eqnarray*}
The comultiplication $\widetilde{\Delta}$ is $R^e$--bilinear by the
$R$--module map properties of $\Delta$, cf. \eqref{idpiLdeltah}. 
Right $R^e$--linearity of $\widetilde{\epsilon}$ follows by 
\begin{eqnarray*}
\widetilde{\epsilon}(h(s\ox r))&=&
\widetilde{\epsilon}(hs\overline \sqcap^L(r))
=
1_1\ox 1_{1'}\epsilon(1_{2'}hs\overline \sqcap^L(r)1_2)
=
r1_1\ox 1_{1'}\epsilon(1_{2'}hs 1_2)\\
&\stackrel{\eqref{piRpiLcommute}}=&
r1_1\ox 1_{1'}\epsilon(1_{2'}h1_2 s)
\stackrel{\eqref{piRproduct}}=
r1_1\ox 1_{1'}\epsilon(1_{2'}h 1_2\sqcap^L(s))\\
&=&
r1_1s\ox 1_{1'}\epsilon(1_{2'}h 1_2)
\stackrel{\eqref{eq:LbimodJ}}=
\widetilde{\epsilon}(h)\cdotdot(s\ox r),
\end{eqnarray*}
for $h\in H$ and $s\ox r\in R^e$.
The third and the sixth equalities follow by the Frobenius property of $R$:
just apply $\mathsf{id}\ox \overline \sqcap^L$ to the identities
$$
1_1\ox \sqcap^R(1_2)\sqcap^R(h)=\sqcap^R(h)1_1\ox \sqcap^R(1_2)
\quad\textrm{and}\quad
1_1\ox \sqcap^R\sqcap^L(h)\sqcap^R(1_2)=1_1\sqcap^R(h)\ox \sqcap^R(1_2)
$$
holding true for all $h\in H$ by (\ref{eq:Frob1}-\ref{eq:Frob2}). 
Left $R^e$--linearity of $\widetilde{\epsilon}$ is checked
symmetrically. 
Coassociativity of $\widetilde{\Delta}$ is obvious. The computation 
$$
h_1\cdot(\sqcap^R(h_2)\ox \overline \sqcap^R(h_3))=
(\overline \sqcap^R(h_3)\ox 1)h_1(\sqcap^R(h_2)\ox 1)=
\overline \sqcap^R(h_3)h_1\sqcap^R(h_2)=h,
$$
for any $h\in H$, shows its right counitality and left
counitality is checked symmetrically. This proves that $(H,\widetilde
\Delta,\widetilde \epsilon)$ is a comonoid in $\mathsf{bim}(R^e)$. 

Our next aim is to show that the compatibility conditions --- expressed
by diagrams \eqref{diag:bimon1}, \eqref{diag:bimon2}, \eqref{diag:bimon3}
and \eqref{diag:bimon4} --- hold between the above monoid and comonoid
structures of $H$. Commutativity of \eqref{diag:bimon1} follows by
commutativity of 
$$
\xymatrix@C=60pt{
h\circ h'\ar@{|->}[r]^-{\widetilde{\Delta}\circ\widetilde{\Delta}}
\ar@{|->}[d]_-{\widetilde{\mu}}&
(h_1\bullet h_2)\circ (h'_1\bullet h'_2)\ar@{|->}[r]^-\gamma&
(h_11_1 \circ h'_1)\bullet (h_2\circ 1_2h'_2)
\ar@{|->}[d]^-{\widetilde{\mu}\bullet\widetilde{\mu}}\\
hh'\ar@{|->}[r]_-{\widetilde{\Delta}}&
(hh')_1\bullet (hh')_2\ar@{=}[r]&
h_11_1h'_1\bullet h_21_2h'_2,}
$$
for any $h,h'\in H$, where the equality in the bottom row follows by the
comultiplicativity and the unitality of the multiplication $\mu:H\ox H \to H$.
Commutativity of \eqref{diag:bimon2} follows by commutativity of
$$
\xymatrix@C=15pt{
h\circ h'\ar@{|->}[rr]^-{\widetilde{\epsilon}\circ\widetilde{\epsilon}}
\ar@{|->}[d]_-{\widetilde{\mu}}&&
(\sqcap^R(h_1)\ox \overline{\sqcap}^R(h_2))\circ 
(\sqcap^R(h'_1)\ox\overline{\sqcap}^R(h'_2))
\ar@{|->}[d]^-{\mu_J}\\
hh'\ar@{|->}[r]_-{\widetilde{\epsilon}}&
\sqcap^R((hh')_1)\ox \overline \sqcap^R((hh')_2)\ar@{=}[r]&
\epsilon(\sqcap^R(h_1)\overline{\sqcap}^R(h'_2))
\sqcap^R(h'_1)\ox\overline{\sqcap}^R(h_2),}
$$
for any $h,h'\in H$. In order to verify the equality in the bottom row,
observe that for all $h,h'\in H$, 
$$
\overline \sqcap^R(hh')
\stackrel{\eqref{piRproduct}}=\overline \sqcap^R(h\overline \sqcap^R(h'))
\stackrel{\eqref{piLbarpiR=piLbar}}=\overline \sqcap^R(h\overline
\sqcap^R\sqcap^L(h')) 
\stackrel{\eqref{piRproduct}}=\overline \sqcap^R(h\sqcap^L(h')).
$$
Using this identity in the penultimate equality, 
\begin{eqnarray*}
\epsilon(\sqcap^R(h_1)\overline{\sqcap}^R(h'_2))\hspace{-.4cm}&&\hspace{-.5cm}
\sqcap^R(h'_1)\ox\overline{\sqcap}^R(h_2)\stackrel{\eqref{piRproduct}}= 
\epsilon(\sqcap^R(h_1)h'_2)
\sqcap^R(h'_1)\ox\overline{\sqcap}^R(h_2)\\
&\stackrel{\eqref{idpiLdeltah}}=& 
\sqcap^R(\sqcap^R(h_1) h')\ox\overline{\sqcap}^R(h_2)
\stackrel{\eqref{piRproduct}}=
\sqcap^R(h_1 h')\ox\overline{\sqcap}^R(h_2) \\
&=& \sqcap^R(h_1 1_{1} h')\ox\overline{\sqcap}^R(h_21_{2})
\stackrel{\eqref{idpiLdeltah}}= 
\sqcap^R(h_1 h'_1)\ox
\overline{\sqcap}^R(h_2 \sqcap^L(h'_2))\\
&=& 
\sqcap^R(h_1h'_1)\ox \overline{\sqcap}^R(h_2h'_2)
= 
\sqcap^R((hh')_1)\ox \overline \sqcap^R((hh')_2).
\end{eqnarray*}

Commutativity of \eqref{diag:bimon3} and \eqref{diag:bimon4} follows by
commutativity of 
$$
\xymatrix{
s\ox r \ar@{|->}[rrr]^-{\widetilde{\eta}}\ar@{|->}[d]_-{\Delta_I}&&&
s\overline{\sqcap}^L(r)\ar@{|->}[d]^-{\widetilde{\Delta}}\\
(1\ox r)\bullet (s\ox 1)\ar@{|->}[r]_-{\widetilde{\eta}\bullet\widetilde{\eta}}&
\overline{\sqcap}^L(r)\bullet s\ar@{=}[r]_-{\eqref{counitalpropertiesofmaps}}&
\overline{\sqcap}^L(r)\cdot (1_1 \ox 1)\bullet (\sqcap^R(1_2)\ox 1)\cdot s
\ar@{=}[r]&
\overline{\sqcap}^L(r)1_1 \bullet s1_2}
$$
and
$$
\xymatrix@C=50pt{
s\ox r \ar@{|->}[rr]^-{\widetilde{\eta}}\ar@{|->}[d]_-{\tau}&&
s\overline{\sqcap}^L(r)\ar@{|->}[d]^-{\widetilde\epsilon}\\
r\sqcap^R(1_2)\ox s1_1\ar@{=}[r]&
r1_1\ox s \overline{\sqcap}^R(1_2)
\ar@{=}[r]_-{\eqref{R-moduleproperties}\ \eqref{piLbarpiR=piLbar}}&
\sqcap^R(\overline{\sqcap}^L(r)1_1)\otimes \overline{\sqcap}^R(s1_2),
}
$$
respectively, for any $s,r\in R$. Therefore, we conclude that
$(H,\widetilde{\mu},\widetilde{\eta},\widetilde{\Delta},\widetilde{\epsilon})$
is a bimonoid in $\mathsf{bim}(R^e)$. 
\end{punto}

\begin{punto}\textbf{From bimonoids to weak bialgebras.}\label{fromwbatobim}
Take now a bimonoid
$(H,\widetilde{\mu},\widetilde{\Delta},\widetilde{\eta},\widetilde{\epsilon})$ 
in $\mathsf{bim}(R^e)$, for some separable Frobenius (co)algebra
$R$ over the field $k$. In this paragraph we equip $H$ with the
structure of a weak bialgebra over $k$, whose `right' subalgebra is
isomorphic to $R$. 

First we construct an associative and unital $k$--algebra structure on
$A$, via the multiplication and the unit defined by
\begin{eqnarray*}
 \xymatrix{\mu:H\ox H\ar@{->>}[r]^-{\pi^{\circ}_{H,H}} & H\circ
 H\ar[r]^-{\widetilde{\mu}} & H}\quad \textrm{and}\quad 
&&\xymatrix{\eta: k\ar[r]^-{{\eta}_{R^e}} & R^e\ar[r]^-{\widetilde{\eta}} & H,}\\
\end{eqnarray*}
where ${\eta}_{R^e}$ denotes the unit of the $k$--algebra
$R^e$. 

Next, we can make $H$ to be a $k$--coalgebra via the comultiplication and the
counit
\begin{equation*}
\xymatrix{\
\Delta:H\ar[r]^-{\widetilde{\Delta}} & H\bullet H\ 
\ar@{>->}[r]^-{\iota^{\bullet}_{H,H}} & H\ox H} \quad \textrm{and}\quad 
\xymatrix{\epsilon: H\ar[r]^-{\widetilde{\epsilon}}& R^e\ar[r]^-{\psi\ox\psi}&
k.} 
\end{equation*}
Indeed, $\Delta$ is evidently coassociative and it is counital by
commutativity of
$$
\xymatrix @R=15pt@C=35pt{
H\ar[r]^-{\widetilde{\Delta}} \ar@/_2pc/@{=}[rrdd]& 
H\bullet H\ar[d]^-{\widetilde \epsilon \bullet H}\ 
\ar@{>->}[r]^-{\iota^{\bullet}_{H,H}} & 
H\ox H\ar[d]^-{\widetilde \epsilon \ox H}\\
&R^e \bullet H\ \ar@{>->}[r]^-{\iota^{\bullet}_{R^e,H}} 
\ar[rd]^(.7){\lambda^\bullet_H}&
R^e\ox H\ar[d]^-{(\psi\ox \psi)\ox H}\\
&&H}
$$
--- where the triangle at the bottom right commutes by
\eqref{eq:lambda_rho_iota} --- and similarly on the other side. 

Our next aim is to show that the above algebra and coalgebra structures
of $H$ combine into a weak bialgebra. In doing so, we use both Sweedler
notations $\widetilde\Delta(h)=h_{\grtilde{1}}\bullet h_{\grtilde{2}}$ and
$\Delta(h)=h_1\ox h_2$, for any $h\in H$. 

We begin with checking the multiplicativity of the comultiplication $\Delta$;
that is, axiom \eqref{multiplicativitycomultiplication}. For any $h\in
H$, $\Delta(h)=
\iota_{H,H}^{\bullet}\widetilde{\Delta}(h)=
(e_j\ox 1) h_{\grtilde{1}} (e_i\ox 1) \ox 
(1\ox f_j) h_{\grtilde{2}} (1\ox f_i)$, hence
\begin{eqnarray*}
\Delta(h)\Delta(h')&=&
(e_j\ox 1) h_{\grtilde{1}} (e_i\ox 1) h'_{\grtilde{1}}(e_k\ox 1)
 \ox 
(1\ox f_j) h_{\grtilde{2}} (1\ox f_i) h'_{\grtilde{2}} (1\ox f_k) \\
&=& \iota_{H,H}^{\bullet}(\widetilde{\mu}\bullet\widetilde{\mu})\gamma
(\widetilde{\Delta}\circ\widetilde{\Delta})
\pi_{H,H}^{\circ}(h\ox h')
\stackrel{\eqref{diag:bimon1}}=
\iota_{H,H}^{\bullet}\widetilde{\Delta}\widetilde{\mu}
\pi_{H,H}^{\circ}(h\ox h') =
\Delta\mu(h\ox h')=\Delta(hh'),
\end{eqnarray*} 
for all $h,h'\in H$. 

Next we check axiom \eqref{weakcomultiplicativityunit}, expressing
weak comultiplicativity of the unit. From \eqref{diag:bimon3} on the bimonoid
$H$ it follows that
\begin{equation}
{\Delta}\widetilde{\eta}(r\ox s) =
\iota_{H,H}^{\bullet}\widetilde{\Delta}\widetilde{\eta}(r\ox s)=
\iota_{H,H}^{\bullet}(\widetilde{\eta} \bullet \widetilde{\eta})
\Delta_I(r\ox s)=
\widetilde{\eta}(e_i\otimes s)\ox 
\widetilde{\eta}(r \ox f_i).
\label{eq:bimon2} 
\end{equation}
With this identity at hand, the weak comultiplicativity of the unit is checked
by
\begin{eqnarray*}
(H\ox {\Delta})\Delta(1)&=&
\widetilde{\eta}(e_i\otimes 1)\ox 
\Delta \widetilde{\eta}(1\ox f_i)
=\widetilde{\eta}(e_i\otimes 1)\ox \widetilde{\eta}(e_j\ox f_i)
\ox \widetilde{\eta}(1\ox f_j)\\
&=&\widetilde{\eta}(e_i\otimes 1)\ox 
\widetilde{\eta}(1 \ox f_i)\widetilde{\eta}(e_j\otimes 1)
\ox \widetilde{\eta}(1\ox f_j)
=(\Delta(1)\ox 1)(1\ox \Delta(1)).
\end{eqnarray*}
Since $\widetilde{\eta}(1\ox r)
\widetilde{\eta}(s\otimes 1)=\widetilde{\eta}(s\ox r)=
\widetilde{\eta}(s\otimes 1)\widetilde{\eta}(1\ox r)$, for all $s,r\in R$,
also $(1\ox \Delta(1))(\Delta(1)\ox 1)=(H\ox {\Delta})\Delta(1)$. 

Finally, we check that axiom
\eqref{weakmultiplicativitycounitalternativeaxioms} --- expressing weak 
multiplicativity of the counit --- holds. This starts with proving the 
equality
\begin{equation}\label{bimonoidcounit}
\widetilde{\epsilon}=
(\sqcap\ox \overline{\sqcap})\iota^{\bullet}_{H,H}\widetilde{\Delta}
=(\sqcap\ox \overline{\sqcap})\Delta
\end{equation}
in terms of the maps
\begin{equation*}\label{pi&pibardef}
\xymatrix{\sqcap:=(H\ar[r]^-{\widetilde{\epsilon}} & 
R\ox R^{op}\ar[r]^-{R\ox \psi} & R)}\qquad \text{and}\qquad 
\xymatrix{\overline{\sqcap}:=(H\ar[r]^-{\widetilde{\epsilon}} & R\ox
R^{op}\ar[r]^-{\psi\ox R^{op}} & R^{op}).} 
\end{equation*}
Equality \eqref{bimonoidcounit} is proven by commutativity of the following
diagram, noting that $\rho^{\bullet}_H$ is an isomorphism. 
\begin{equation*}
\xymatrix{
H\ox H \ar[rr]^-{\widetilde{\epsilon}\ox \widetilde{\epsilon}} && 
R^e\ox R^e \ar[rr]^-{R\ox \psi\ox \psi\ox R^{op}} && R^e\\
H\bullet H \ar@{>->}[u]^-{\iota_{H,H}^{\bullet}} 
\ar[r]^-{H\bullet \widetilde{\epsilon}} & 
H\bullet R^e\ar[ld]_-{\rho^{\bullet}_H}\ar[r]^{\widetilde{\epsilon}\bullet R^e}& 
R^e\bullet R^e
\ar@{>->}[u]^-{\iota_{R^e,R^e}^{\bullet}}\ar[rr]^-{\rho^{\bullet}_{R^e}} &&
R^e\ar@{=}[u]\\
H\ar[u]^-{\widetilde{\Delta}}\ar@/_1.5pc/[rrrru]^-{\widetilde{\epsilon}}}
\end{equation*}
The bottom-right region and the top-left region commute by the
$R^e$--bimodule map property of $\widetilde{\epsilon}$. The bottom-left region
commutes by counitality $\widetilde \Delta$. Commutativity of the top-right
region follows similarly to \eqref{eq:lambda_rho_iota}. 
From \eqref{diag:bimon2} on the bimonoid $H$ and from
\eqref{bimonoidcounit}, it follows that 
\begin{equation}
\sqcap((hh')_{ 1})\ox \overline{\sqcap}((hh')_{ 2})=
\psi(\sqcap(h_{ 1})\overline{\sqcap}(h'_{ 2})) \sqcap(h'_{ 1})\ox
\overline{\sqcap}(h_{ 2}), \label{eq:bimon3} 
\end{equation}
for all $h,h'\in H$. Using the $R^e$--bilinearity of $\widetilde
\epsilon$ together with \eqref{bimonoidcounit} in the second equality, 
\begin{eqnarray*}
\epsilon(h 1_1)\epsilon(1_2 h')
&=&(\psi\ox\psi)\widetilde{\epsilon}(h (e_i\ox 1))(\psi\ox\psi)
\widetilde{\epsilon}((1\ox f_i) h')\\
&=&\psi(\sqcap(h _{ 1})e_i)
\psi\overline{\sqcap}(h _{ 2})
\psi\sqcap(h' _{ 1})
\psi(\overline{\sqcap}(h' _{ 2})\theta^{-1}(f_i))\\
&\stackrel{\eqref{eq:Frob-1}\eqref{eq:Frob0}}=&
\psi(\sqcap(h _{ 1})\overline{\sqcap}(h' _{ 2}))
\psi\sqcap(h' _{ 1})
\psi\overline{\sqcap}(h _{ 2})\\ 
&\stackrel{\eqref{eq:bimon3}}=& 
\psi\sqcap((h h')_{ 1})
\psi\overline{\sqcap}((h h')_{ 2})
\stackrel{\eqref{bimonoidcounit}}=
(\psi\ox\psi)\widetilde{\epsilon}(hh')= 
\epsilon(h h'),
\end{eqnarray*}
where in the first equality we used \eqref{eq:bimon2} and that the
multiplication $\mu$ of the $k$--algebra $H$ is $R^e$--balanced and
$R^e$--bilinear. A symmetrical computation verifies
$\epsilon(h1_2)\epsilon(1_1 h')=\epsilon(hh')$, for all $h,h'\in H$. 

We have so far constructed a weak bialgebra structure on $H$. It remains to
check that $\sqcap^R(H)$ is isomorphic to the given separable Frobenius
algebra $R$. With this purpose, consider the map 
\begin{equation}\label{isoRHR}
\sigma: R \rightarrow H, \qquad 
r\mapsto \widetilde{\eta}(r\ox 1).
\end{equation} 
Since for any $s,r\in R$
\begin{equation}\label{eq:epsilon.tildeeta}
\epsilon\widetilde{\eta}(r\ox s)=
(\psi\ox\psi)\widetilde{\epsilon}\widetilde{\eta}(r\ox s)=
(\psi\ox\psi)\tau(r\ox s)=
\psi(sf_i)\psi(re_i)=
\psi(sr),
\end{equation}
and by \eqref{eq:bimon2}, 
$$
\sqcap^R\sigma (r)=
\widetilde{\eta}(e_i\ox 1)
\epsilon\widetilde{\eta}(r\ox f_i)
\stackrel{\eqref{eq:epsilon.tildeeta}}=\widetilde{\eta}(e_i\ox 1)\psi(f_ir)
=
\widetilde{\eta}(r\ox 1)
=
\sigma (r).
$$
This proves that $\sigma$ corestricts to a map $R\to \sqcap^R(H)$,
to be denoted also by $\sigma$. This restricted map $\sigma: R\to
\sqcap^R(H)$ is our candidate to establish the desired isomorphism of
separable Frobenius algebras. Since $\widetilde \eta$ is a $k$--algebra
homomorphism, so is $\sigma$. Comultiplicativity of $\sigma$ is proven using
the identity 
\begin{equation}\label{eq:PiR.tildeeta}
\sqcap^R\widetilde\eta(1\ox r)=
\widetilde\eta(e_i\ox 1)
\epsilon\widetilde \eta(1\ox f_ir)
\stackrel{\eqref{eq:epsilon.tildeeta}}=
\widetilde\eta(e_i\ox 1)\psi(f_ir)=
\widetilde\eta(r\ox 1).
\end{equation}
With this identity at hand, 
$$
\sigma(r)\widetilde{\eta}(e_i\ox 1) \ox \sqcap^R\widetilde{\eta}(1\ox f_i)
\stackrel{\eqref{eq:PiR.tildeeta}}=
\widetilde{\eta}(re_i\ox 1) \ox \widetilde\eta(f_i\ox 1)
=\sigma(re_i)\ox \sigma(f_i).
$$
Finally, $\sigma$ is also counital by applying \eqref{eq:epsilon.tildeeta} for
$s=1$. Since any map between Frobenius algebras, which is both an algebra and
a coalgebra homomorphism, is an isomorphism (cf. \cite[Proposition
A.3]{PastroStreet}), this proves that $\sigma$ is an isomorphism of
separable Frobenius algebras. 
\end{punto}

\begin{theorem}\label{thm:wba_obj}
Let $R$ be a separable Frobenius (co)algebra over a field $k$. A
bimonoid in the duoidal category $\mathsf{bim}(R^e)$ in Theorem
\ref{thm:bim_obj} is, equivalently, a weak bialgebra over $k$ whose right
subalgebra is isomorphic to $R$ (as a separable Frobenius (co)algebra). 
\end{theorem}

\begin{proof}
In light of Paragraphs \ref{frombimtowba} and \ref{fromwbatobim},
we only have to prove the bijectivity of the correspondence described in
them. Starting with a weak bialgebra $(H,\mu,\eta,\Delta,\epsilon)$, and
applying to it the above constructions, the resulting weak bialgebra has the
same structure as $H$, as the following computations show. The resulting
multiplication is the unique map which yields $\mu \pi^\circ_{H,H}$ if
composed with $\pi^\circ_{H,H}$. Hence it is equal to $\mu$. The resulting
unit map multiplies an element of $k$ by $1 \overline \sqcap^L(1)=1$ hence it
is equal to $\eta$. The resulting comultiplication is equal to $\Delta$ by
\eqref{eq:Delta_fac}. The resulting counit sends $h\in H$ to 
$$
(\epsilon_{|R} \sqcap^R\ox \epsilon_{|R} \overline{\sqcap}^R)\Delta(h)
=(\epsilon\ox \epsilon)\Delta(h) 
=\epsilon(h).
$$
 
Conversely, consider a bimonoid
$(H, \widetilde{\mu},\widetilde{\eta},\widetilde{\Delta},
\widetilde{\epsilon})$ in $\mathsf{bim}(R^e)$ and the bimonoid obtained by
applying to it the constructions in Paragraphs \ref{frombimtowba} and
\ref{fromwbatobim}. By construction, they have identical multiplications and
comultiplications. Concerning the unit and the counit, note that in the weak
bialgebra in Paragraph \ref{fromwbatobim}, 
\begin{eqnarray*}
\sqcap^R(h)&=& 
\widetilde{\eta}(e_i\ox 1)(\psi\ox \psi)\widetilde{\epsilon}(h
\widetilde \eta(1\ox f_i))\\
&=&\widetilde{\eta}(e_i\ox 1) \psi(f_i\sqcap(h_1)) 
\psi\overline{\sqcap}(h_2) 
=\widetilde{\eta}(\sqcap(h_1)\ox 1) \epsilon(h_2) 
=\widetilde{\eta}(\sqcap(h)\ox 1),
\end{eqnarray*}
for all $h\in H$, where in the second equality we used the right
$R^e$--linearity of $\widetilde{\epsilon}$, \eqref{bimonoidcounit} and
\eqref{eq:LbimodJ}; and in the penultimate equality we used \eqref{eq:Frob-1}
and $\psi\overline{\sqcap}=\epsilon$. 
By similar computations also the idempotent maps $\overline\sqcap^L$
and $\overline\sqcap^R$ --- in the weak bialgebra associated in 
Paragraph \ref{fromwbatobim} to the bimonoid
$(H,\widetilde{\mu},\widetilde{\eta},\widetilde{\Delta},\widetilde{\epsilon})$ 
--- can be expressed as 
$$
\overline\sqcap^L=\widetilde\eta(1\ox \sqcap(-))
\qquad \textrm{and}\qquad 
\overline\sqcap^R=\widetilde\eta(\overline{\sqcap}(-)\ox 1).
$$
So the counits differ by the isomorphism $\sigma \otimes \sigma$ by
\eqref{bimonoidcounit}. Finally, in the bimonoid obtained by applying both
constructions, the unit map takes $s\ox r\in R\ox R^{op}$ to
$\sigma(s)\overline \sqcap^L\sigma(r)=\widetilde\eta(s\ox
1)\widetilde\eta(1\ox r)=\widetilde\eta(s\ox r)$. 
 \end{proof}

By Theorem \ref{thm:wba_obj}, an object of $\mathsf{bmd}(\mathsf{bim}(-^e))$
is given by a weak bialgebra. We make no notational distinction between a weak
bialgebra $H$ and the corresponding bimonoid in the bi(co)module category
$\mathsf{bim}(R^e)$, where $R$ is the `right' subalgebra $\sqcap^R(H)$. 

By \cite{Szlach:2001,Schauenburg:2003}, a weak bialgebra of `right'
subalgebra $R$ can be regarded as a right $R$--bialgebroid (or
`$\times_R$--bialgebra' in \cite{Takeuchi}) supplemented by a separable
Frobenius structure on $R$. However, since for arbitrary algebras $R$ we
cannot equip the category of $R^e$--bimodules with a duoidal structure (see
Remark \ref{rem:sF}), we cannot extend Theorem \ref{thm:wba_obj} to
interpret arbitrary bialgebroids as bimonoids in an appropriate duoidal
category. 

\begin{theorem}\label{thm:wba_mor}
Let $H$ and $H'$ be weak bialgebras with respective right subalgebras $R$
and $R'$. A morphism in $\mathsf{bmd}(\mathsf{bim}(-^e))$ from
$(R,H)$ to $(R',H')$ is, equivalently, a coalgebra map $Q:H\to H'$, rendering
commutative the diagrams
$$
\xymatrix @R=1pt {
H\ar[r]^-Q\ar[dddddddd]^(.4){\sqcap^R}&
H'\ar[dddddddd]_(.6){\sqcap^{\prime R}}
&
H\ar[r]^-Q\ar[dddddddd]^(.4){\overline\sqcap^R}&
H'\ar[dddddddd]_(.6){\overline\sqcap^{\prime R}}
&
{H\ar[r]^-{Q}\ar[dddddddd]^(.4){\sqcap^R\sqcap^L}}&
{ H'\ar[dddddddd]_(.6){\sqcap^{\prime R}\sqcap^{\prime L}}}
&
H\ox H \ar[r]^-E\ar[dddddddd]^-{\mu}&
H\ox H\ar[r]^-{Q\ox Q}&
H'\ox H'\ar[dddddddd]_-{\mu'}
\\ \\ \\ \\ 
&&&&&&
\\ \\ \\ \\
H\ar[r]_-Q&
H'
&
H\ar[r]_-Q&
H'
&
{H\ar[r]_-{Q}}&
{H'}
&
H\ar[rr]_-Q&&
H',}
$$
where
$E(h\ox h'):= h1_1\ox \sqcap^R (1_2)h'$.
\end{theorem}

\begin{proof}
Let us take first a morphism in $\mathsf{bmd(bim}(-^e))$, and see that
it obeys the properties in the claim. A morphism in $\mathsf{bmd(bim}(-^e))$
is given by a morphism $q:R\rightarrow R'$ in $\mathsf{sfr}$ and a morphism
$Q:\mathsf{bim}(q^e)H\rightarrow H'$ in $\mathsf{bim}(R'^e)$, rendering
commutative the four diagrams in part (b) of Lemma
\ref{lem:bmd_morphism}. 

Let us see first that $Q$ is a coalgebra map. In order to prove that it is
comultiplicative, we need to see that the top row of 
$$
\xymatrix @C=40pt{
H\ar[d]_-Q\ar[r]^-{
\Delta}
\ar@{}[rrrd]|-{\eqref{diag1lemma}}&
H\ox H
\ar[r]^-{
\pi^\bullet_{H,H}}
&
H\bullet H
\ar[r]^-{
\mathsf{bim}(q^e)^\bullet_2}
&
H\bullet' 
H\ar[d]^-{Q\bullet'Q}
\ar[r]^-{
\iota^{\bullet'}_{
H, 
H}
}
&
H\ox 
H\ar[d]^-{Q\ox Q}\\
H'\ar[rr]^-{\Delta'} \ar@/_1.5pc/[rrrr]_-{\Delta'}^-{\eqref{eq:Delta_fac}}&&
H'\ox H'\ar[r]^-{\pi^{\bullet'}_{H',H'}}&
H'\bullet' H'\ar[r]^-{\iota^{\bullet'}_{H',H'}}&
H'\ox H'}
$$
is equal to the comultiplication $\Delta$ of $H$. Computing its value on $h\in
H$, we get 
$$
(\widetilde q (e'_k)e_i\ox 1)h_1(e_j\widetilde q (e'_l)\ox 1)\ox
(1\ox f_i \widetilde q (f'_k))h_2(1\ox \widetilde q (f'_l)f_j).
$$
It is equal to 
$$
(e_i\ox 1)h_1(e_j\ox 1)\ox (1\ox f_i)h_2(1\ox f_j)=
1_1h_11_{1'}\ox 1_2h_21_{2'}=\Delta(h)
$$ 
by 
\begin{equation}\label{eq:qtildeid}
\begin{array}{rl}
\widetilde q (e'_k)e_i\ox f_i \widetilde q (f'_k) 
=&\widetilde q (e'_k)\widetilde q (f'_k)e_i\ox f_i 
= \widetilde q (e'_kf'_k)e_i\ox f_i 
= \widetilde q (1')e_i\ox f_i 
= e_i\ox f_i 
\quad \textrm{and}\\
e_j\widetilde q (e'_l)\ox \widetilde q (f'_l)f_j
=& e_j\ox \widetilde q (f'_l) \theta \widetilde q (e'_l)f_j
= e_j\ox \widetilde q (f'_l) \widetilde q \theta'(e'_l)f_j
=e_j\ox \widetilde q (e'_l) \widetilde q (f'_l)f_j
= e_j\ox f_j.
\end{array}
\end{equation}
This proves comultiplicativity of $Q$.
In order to see that $Q$ is counital as well, observe that condition
\eqref{diag2lemma} takes now the form 
$$
\xymatrix@C=70pt{
H\ar[d]_-Q \ar[r]^-{\raisebox{7pt}{${}_{(\sqcap^R\ox \overline \sqcap^R)\Delta}$}}&
R\ox R^{op}
\ar[r]^-{q\ox q^{op}}&
R'\ox R^{\prime op}\ar@{=}[d]\\
H'
\ar[rr]_-{(\sqcap^{\prime R}\ox \overline \sqcap^{\prime R})\Delta'}
&&
R'\ox R^{\prime op}.}
$$
Composing both paths around it with $R'\ox \epsilon'_{|R'}$ and with
$\epsilon'_{|R'}\ox R'$, respectively, we obtain 
\begin{equation}\label{eq:qPi}
\sqcap^{\prime R}Q(h)=q \sqcap^R(h)\qquad \textrm{and}\qquad 
\overline
\sqcap^{\prime R}Q(h)=q\overline \sqcap^R(h);
\end{equation}
and composing either one of these equalities with $\epsilon'_{|R'}$ we have the
counitality of $Q$ proven. 

Let us check now that $Q$ satisfies the required weak multiplicativity
condition; that is, it renders commutative the last diagram in the
claim. Since $(q,Q)$ is a morphism in $\mathsf{bmd(bim}(-^e))$ by
assumption, it renders commutative diagram \eqref{diag3lemma} for any 
$R^e$--bimodules $A$ and $B$. Evaluating both paths around it
on an arbitrary element $(a\bullet h) \circ (b\bullet h')$, and using
the commutativity of $\widetilde{q}$ with the Nakayama automorphisms, 
the $R'^e$--bilinearity of $Q$, \eqref{eq:Frob4} and that $\widetilde{q}$ is
an algebra map together with \eqref{eq:Frob1}, yields the equivalent
form 
\begin{equation}\label{eq:diagr_2.3}
\begin{array}{l}
((e_i\ox 1)a(e_je_p\ox f_l)\circ' (f_p\ox e_l)b(e_{j'}\ox 1))
\bullet' Q((1\ox f_i)h(1\ox f_j)h'(1\ox f_{j'}))=\\
((e_i\ox 1)a(e_j\ox f_l)\circ' 
(e_{i'}\ox e_l)b(e_{j'}\ox 1))\bullet' 
Q((1\ox f_i)h(e_k\ox f_j))Q((f_k\ox f_{i'})h'(1\ox f_{j'}))
\end{array}
\end{equation}
of condition \eqref{diag3lemma} on $(q,Q)$. Taking $A=B=R^e\ox R^e$ with the
$R^e$--actions
$$
(r\ox s)((x\ox y)\ox (v\ox w))(r'\ox s'):=(rx\ox ys) \ox (vr'\ox s'w),
$$
putting $a=b=1\ox1\ox 1\ox 1$, and applying $(\iota^{\circ'}\ox
H')\iota^{\bullet'}$ to the resulting equality, by the $R^{\prime
e}$--bilinearity of $Q$ and \eqref{eq:qtildeid} we obtain
\begin{eqnarray*}
&&
e_i\ox 1\ox e_je_p\ox f_l\ox f_p\ox e_l\ox e_{j'}
\ox 1\ox Q((1\ox f_i)h(1\ox f_j)h'(1\ox f_{j'}))
\\ &=&
e_i\ox 1\ox e_j
\ox 
f_l
\ox 
e_{i'}\ox e_l\ox e_{j'}
\ox 1
\ox 
Q((1\ox f_i)h(e_k\ox f_j))
Q((f_k\ox f_{i'})h'(1\ox f_{j'})).
\end{eqnarray*}
Applying $\psi$ to the first, third, fifth and seventh tensorands in the last
equality, we get 
\begin{equation*}
 1\ox f_l\ox e_l \ox 1 \ox Q(hh')= 
 1\ox f_l\ox e_l \ox 1 \ox Q(h(e_k\ox 1))Q((f_k\ox 1)h').
\end{equation*}
This is equivalent to
$$
Q(hh')= 
Q(h1_1)Q(\sqcap^R(1_2)h'),
$$
that is, commutativity of the last diagram in the claim. 

Next we check that $q$ can be uniquely reconstructed from $Q$ ---
namely, it is the (co)restriction to $R\to R'$ of $Q:H\rightarrow 
H'$. Evaluating the equal paths around 
$$
\xymatrix@R=25pt{
R^e\ar[d]_-{\Delta_I}\ar[r]^-{q^e}\ar@{}[rrdd]|-{\eqref{diag4lemma}}&
R^{\prime e}\ar[r]^-{\Delta_{I'}}&
R^{\prime e} \bullet' R^{\prime e} 
\ar[r]^-{\iota^{\bullet'}_{R^{\prime e},R^{\prime e}}}
\ar[dd]^-{R^{\prime e}\bullet'\widetilde \eta'}&
R^{\prime e}\ox R^{\prime e}\ar[ddd]^-{R^{\prime e}\ox \widetilde \eta'}\\
R^e\bullet R^e \ar[d]_-{\mathsf{bim}(q^e)^\bullet_2}\\
R^e\bullet' R^e \ar[r]^-{R^e\bullet' \widetilde \eta}
\ar[d]_-{\iota^{\bullet'}_{R^e,R^e}}&
R^e\bullet' H\ar[r]^-{q^e\bullet' Q}\ar[d]^-{\iota^{\bullet'}_{R^e,H}}&
R^{\prime e}\bullet' H'\ar[rd]^-{\iota^{\bullet'}_{R^{\prime e},H'}}\\
R^e\ox R^e \ar[r]_-{R^e\ox\widetilde \eta}&
R^e\ox H \ar[rr]_-{q^e\ox Q}&&
R^{\prime e}\ox H'}
$$
at $1_1\ox \sqcap^R(1_2)r\in R^e$, and composing the result with
$\epsilon'_{|R'}\ox \epsilon^{\prime}_{|R^{\prime}} \ox H'$, we conclude
$q(r)=Q(r)$. 

Comparing this identity $q(r)=Q(r)$ with \eqref{eq:qPi}, the 
compatibility of $Q$ with $\sqcap^R$ and $\overline \sqcap^R$
--- that is, commutativity of the first two diagrams in the claim ---
follows. Commutativity of the third diagram --- that is, compatibility of $Q$
with $\sqcap^R\sqcap^L$ --- is equivalent to the assumed commutativity of $q$
with the Nakayama automorphisms. 

Conversely, assume that $Q:H\rightarrow H'$ is a coalgebra map rendering
commutative the four diagrams in the statement. 
We construct its mate $q:R\to R'$ together with whom they constitute a
morphism in $\mathsf{bmd}(\mathsf{bim}(-^e))$. 

By commutativity of any of the first two diagrams, $Q$ restricts to a map
$q:R\to R'$. Let us see that the restriction $q:R\rightarrow R'$ of $Q$ is
a morphism in $\mathsf{sfr}$. First of all, that it is a coalgebra map. Take
$y\in R$. Since $Q$ respects the counits, $\epsilon'_{\vert R'}q(y)=
\epsilon'Q(y)=\epsilon_{\vert R}(y)$. Moreover, $q$ is comultiplicative by 
\begin{eqnarray*}
Q(y)1'_1\ox \sqcap'^R(1'_2)
&\stackrel{\eqref{idpiLdeltah}}=&
Q(y)_1\ox \sqcap'^R(Q(y)_2)
=
Q(y_1)\ox \sqcap'^RQ(y_2)\\
&=& Q(y_1)\ox Q\sqcap^R(y_2)
\stackrel{\eqref{idpiLdeltah}}=
Q(y1_1)\ox Q\sqcap^R(1_2).
\end{eqnarray*}
By commutativity of the third diagram in the claim, $q$ commutes with the
Nakayama automorphisms. Hence it is a morphism in $\mathsf{sfr}$, as needed. 

In order for $Q$ to be a morphism in $\mathsf{bim}(R^{\prime e})$, it
has to be an $R^{\prime e}$--bimodule map. We check that it is a right
$R'$--module map; its compatibility with the other three $R'$--actions is
similarly proven. 
\begin{eqnarray*}
Q(h\widetilde{q}(r'))&=&
Q(h1_1)\epsilon'(q\sqcap^R(1_2)r')\stackrel{\eqref{idpiLdeltah}}=
Q(h_1)\epsilon'(q\sqcap^R(h_2)r')\\
&=&Q(h_1)\epsilon'(\sqcap'^RQ(h_2)r')
\stackrel{\eqref{piRproduct}}=Q(h_1)\epsilon'(Q(h_2)r')
=Q(h)_1\epsilon'(Q(h)_2r')
=Q(h)r',
\end{eqnarray*}
illustrating that $Q$ is a morphism in $\mathsf{bim}(R^{\prime e})$. 

It remains to show that the morphisms $q:R\to R'$ in $\mathsf{sfr}$ and
$Q:H\to H'$ in $\mathsf{bim}(R^{\prime e})$ obey the conditions in part
(b) of Lemma \ref{lem:bmd_morphism}. The following commutative diagrams
show that \eqref{diag1lemma} and \eqref{diag2lemma} hold.
$$
\xymatrix@R=15pt@C=30pt{
H\ar[ddd]_-{Q}\ar[rd]_-{\Delta}
\ar[r]^-{\widetilde{\Delta }} & 
H\bullet H \ar[r]^-{\mathsf{bim}(q^e)^\bullet_2}
\ar@{}[d]|-{\eqref{eq:qtildeid}}& H\bullet' H\ar[ddd]^-{Q\bullet^{\prime} Q} 
&&
H\ar[ddd]_-{Q}
\ar[r]^-{\Delta}&
H\ox H\ar[r]^-{\sqcap^R\ox \overline \sqcap^R}
\ar[ddd]^-{Q\ox Q}& 
R\ox R^{op}\ar[ddd]^-{q\ox q^{op}}\\
&H\ox H\ar[ur]_-{\pi^{\bullet'}_{H,H}}\ar[d]^-{Q\ox Q}
& &&& 
\\
&H'\ox H' \ar[dr]^-{\pi^{\bullet'}_{H',H'}}
&& && 
\\
H'\ar[rr]_-{\widetilde{\Delta}'}
\ar[ur]^-{\Delta'} && 
H'\bullet^{\prime} H'
&& 
H'\ar[r]_-{\Delta'}&
H'\ox H'\ar[r]_-{\sqcap^{\prime R}\ox \overline \sqcap^{\prime R}}& 
R'\ox R^{\prime op}} 
$$ 
Commutativity of diagram \eqref{diag3lemma} was seen to be equivalent
to \eqref{eq:diagr_2.3}. It holds by the following computation, 
for all $h,h'\in H$, $a\in A$ and $b\in B$ for any
$R^e$--bimodules $A$ and $B$. 
\begin{eqnarray*}
&&\hspace{-3pt}
((e_i\ox 1)a(e_j\ox f_l)\circ' (e_{i'}\ox e_l)b(e_{j'}\ox 1))\bullet' 
Q((1\ox f_i)h(e_k\ox f_j))Q((f_k\ox f_{i'})h'(1\ox f_{j'}))\\
&\!\!=&\!\!((e_i\ox 1)a(e_j\ox f_l)\circ' (e_{i'}\ox e_l)b(e_{j'}\ox 1))\bullet' 
Q((1\ox f_i)h(1\ox f_j)(1\ox f_{i'})h'(1\ox f_{j'}))\\
&\!\!=&\!\!((e_i\ox 1)a(e_je_{i'}\ox f_l)\circ' (f_{i'}\ox e_l)b(e_{j'}\ox 1))
\bullet' Q((1\ox f_i)h(1\ox f_j)h'(1\ox f_{j'}))
\end{eqnarray*}
In the first equality we used the weak multiplicativity of $Q$, holding true
by assumption. In the second equality we used \eqref{eq:Frob4} and
\eqref{eq:Frob2}. 

Commutativity of diagram \eqref{diag4lemma} is checked by the computation
\begin{eqnarray*}
(q(1_1)\ox q(y))\bullet' Q\widetilde{\eta}(x\ox \sqcap^R(1_2))
&=& (1\ox q(y))\bullet' Q\widetilde{\eta}(x\ox \sqcap^R(1_2))
(1\ox q(1_1))\\
&=& (1\ox q(y))\bullet' Q(\widetilde{\eta}(x\ox \sqcap^R(1_2))
(1\ox \widetilde qq(1_1)))\\
&=& (1\ox q(y))\bullet' 
Q\widetilde{\eta}(x\ox \widetilde{q}q(1_1)\sqcap^R(1_2))\\
&=&(1\ox q(y))\bullet' 
Q\widetilde{\eta}(x\ox \widetilde{q}(1'_1)\widetilde{q}\sqcap'^R(1'_2))\\
&=& (1\ox q(y))\bullet' 
Q\widetilde{\eta}(x\ox 1)\\
&=&(1\ox q(y))\bullet'\widetilde{\eta}'(q(x)\ox 1),
\end{eqnarray*}
for any $x,y\in R$. In the first equality we used the definition of $\bullet'$ 
(cf. \eqref{eq:bullet_quotient}). In the second and third equalities we used 
the right $R'^e$--linearity of $Q$ and the right $R^e$--linearity of
$\widetilde \eta$, respectively. In the fourth equality we used
\eqref{eq:qhat_tr}; in the penultimate equality we used that
$\widetilde{q}$ is an algebra map together with
\eqref{counitalpropertiesofmaps}; and in the last equality we used that $Q$
restricts to $q$ on $R$. 
\end{proof}

We conclude by Theorem \ref{thm:wba_obj} and Theorem \ref{thm:wba_mor} that
the category $\mathsf{bmd}(\mathsf{bim}(-^e))$ has weak bialgebras as its
objects and morphisms as in Theorem \ref{thm:wba_mor}. Thus we can regard it
as the category of weak bialgebras and introduce the notation $\mathsf{wba}$
for it.

Applying results from \cite{Szlachanyi:2003}, we know from Lemma
\ref{lem:bmd_morphism} that the morphisms in $\mathsf{wba}$ are closed under 
the composition. But it is also easy to see this directly. Indeed, if both
morphisms $Q:H\to H'$ and $Q':H'\to H''$ render commutative the
first three diagrams in Theorem \ref{thm:wba_mor}, then so does their
composite evidently. If $Q$ and $Q'$ make commutative the last diagram Theorem
\ref{thm:wba_mor}, then so does their composite since $Q$ is a morphism of
$R'$--bimodules and \eqref{eq:qtildeid} holds: for any $h,h'\in H$,
\begin{eqnarray*}
Q'Q(hh')&=&
Q'[Q(h1_1)Q(\sqcap^R(1_2)h')]=
Q'[Q(h1_1)1'_1]Q'[\sqcap^{\prime R}(1'_2)Q(\sqcap^R(1_2)h')]\\
&=&Q'Q(h1_1\widetilde q(1'_1))
Q'Q(\widetilde q\sqcap^{\prime R}(1'_2)\sqcap^R(1_2)h')
\stackrel{\eqref{eq:qtildeid}}=
Q'Q(h1_1)Q'Q(\sqcap^R(1_2)h').
\end{eqnarray*}

While the notion of weak bialgebra is self-dual, the morphisms in Theorem
\ref{thm:wba_mor} are not. (They are coalgebra homomorphisms but not algebra
homomorphisms.) The dual counterpart of $\mathsf{wba}$; that is, a category of
weak bialgebras with the dual notion of morphisms, would be obtained from
a construction based on a symmetric form of Definition \ref{def:mor_bmd} (see
the discussion in Remark \ref{rem:symm}). 

The morphisms in Theorem \ref{thm:wba_mor} look different from all other
kinds of morphisms between weak bialgebras discussed previously in
\cite[Section 1.4]{Szlach:Gal_fin}. However, if we restrict to morphisms
$Q:H\to H'$ whose (co)restriction $q:\sqcap^R(H)\to \sqcap^{\prime R}(H')$ is
the identity map, they are in particular unit preserving
$\sqcap^R(H)=\sqcap^{\prime R}(H')$--bimodule maps; hence also homomorphisms
of algebras (see also Remark \ref{rem:double_comon}). That is to say,
they are {\em `strict morphisms'} of weak bialgebras in the sense of
\cite[Section 1.4]{Szlach:Gal_fin}. For usual (non-weak) 
bialgebras $H$ and $H'$ over the field $k$, any morphism $H\to H'$ in Theorem
\ref{thm:wba_mor} (co)restricts to the identity map $\sqcap^R(H)\cong k\to
\sqcap^{\prime R}(H')\cong k$. Hence $\mathsf{wba}$ contains the usual
category of $k$--bialgebras --- in which morphisms are algebra and coalgebra
homomorphisms --- as a full subcategory. 

\section{The ``free vector space'' functor.}
\label{sec:functor_k}

Let $k$ be a field. For any small category $A$ with finite object set $X$, let
$kA$ denote the free $k$--vector space spanned by the set of morphisms in
$A$. Consider the unique $k$--coalgebra structure $(kA,\Delta,\epsilon)$ for
which the elements of $A$ are group-like, that is, $\Delta (a) = a \ox a$,
$\epsilon (a) = 1$ for all $a \in A$. Let $t: X \leftarrow A \rightarrow X :s$
be the target and source maps, respectively, in the category $A$. The vector
space $kA$ is an algebra with the multiplication determined by the rule $ab =
\delta_{s(a),t(b)} a . b$, where $\delta $ denotes Kronecker's `delta
operator' and $.$ is the composition in $A$. The unit of $kA$ is given by $1 =
\sum_{x \in X} x$ (where $x$ denotes also the identity morphism at
$x$). With these algebra and coalgebra structures $kA$ turns out to be a weak
bialgebra over $k$ (see for example \cite[Section 3.2.2, page
187]{Bohm:2009}, or \cite[Section 2.5]{NykVai} for the case when $A$ is a
groupoid hence $kA$ is a weak Hopf algebra). For any $a \in A$, we get
$\sqcap^R (a) = {s(a)} = \overline{\sqcap}^L(a)$ and $\overline{\sqcap}^R(a) =
{t(a)} = \sqcap^L(a)$. 

This assignment gives the object map of a functor $\mathsf k
:\mathsf{cat} \to \mathsf{wba}$ --- from the category $\mathsf{cat}$ of small
categories with finitely many objects to $\mathsf{wba}$ --- as Proposition
\ref{fAA'} shows. 

\begin{proposition}\label{fAA'}
Let $A$ and $A'$ be small categories with finite object sets $X$
and $X'$, respectively. For any functor $f:A \to A'$, the linear extension
$kf:kA\to kA'$ is a morphism in $\mathsf{wba}$. 
\end{proposition}

\begin{proof}
First, note that $kf$ is a morphism of $k$--coalgebras because it sends
group-like elements to group-like elements; and group-like elements
provide a basis in $kA$. 
We need to show that the four diagrams in Theorem \ref{thm:wba_mor} commute
for $Q = kf$. As for the first two concerns, for any basis element $a \in A$ 
\begin{eqnarray*}
(kf)\sqcap^R(a)&=&
(kf) {s(a)}=
f {s(a)} = 
{s^{ \prime}f(a)}=
\sqcap^{ \prime R}f(a)=
\sqcap^{ \prime R}(kf)(a)\\
(kf)\ \overline{\sqcap}^R(a)&=&
(kf){t(a)}\, =
f{t(a)} = 
{t^{ \prime}f(a)}= 
\overline{\sqcap}^{ \prime R}f(a)=
\overline{\sqcap}^{ \prime R}(kf)(a).
\end{eqnarray*}
The commutativity of the third diagram in Theorem \ref{thm:wba_mor} becomes
redundant by $\sqcap^L=\overline \sqcap^R$. In order to check that
the fourth diagram commutes, let us first note that any element 
in the range of the map $E=(-)1_1\ox \sqcap^R(1_2)(-):kA \ox kA\to kA\ox
kA$ is of the form 
\begin{equation*}
\sum_{x\in X} (\sum_{a\in A} \lambda_a a) x \ox 
x(\sum_{a'\in A}\lambda_{a'}a')=
\sum_{x\in X}(\sum_{a : s(a)=x} \lambda_a a)
\ox (\sum_{ a': t(a')=x} \lambda_{a'} a')=
\sum_{ a,a': s(a)=t(a')} \lambda_a\lambda_{a'}a\ox a',
\end{equation*}
and if $s(a)=t(a')$ then
$$
(kf)\mu(
a\ox a')=f(a . a')=f(a) . f(a')=
\mu'(kf\ox kf)(a\ox a').
$$
\end{proof}

\section{On group-like elements in a weak bialgebra.}
\label{sec:gr-like}

In forthcoming Section \ref{sec:adjoint} we are going to construct the right
adjoint $\mathsf g$ of the ``free vector space'' functor $\mathsf k$ in Section
\ref{sec:functor_k}. Recall that for any small category $A$, the set of
morphisms is in a bijective correspondence with the set of functors from
the interval category $\mathbbm{2}=\xymatrix{S\ar@(dl,ul)[]^-{}\ar[r]^-{a} &
T\ar@(dr,ur)[]^-{}}$ to $A$. 
So if the right adjoint $\mathsf g$ of $\mathsf k$ exists, then for any weak
bialgebra $H$ over the field $k$, the set of morphisms in $\mathsf g 
(H)$ is isomorphic to $\mathsf{cat}(\mathbbm{2},\mathsf g (H))\cong
\mathsf{wba}(k\mathbbm{2},H)$. This motivates the study
of the set $\mathsf{wba}(k\mathbbm{2},H)$ for any weak bialgebra $H$, with the
aim of finding the way to look at it as the set of morphisms in an appropriate
category. 

\begin{definition}\label{groupoidlike-set}
For any weak bialgebra $H$, define the subset 
$$
\mathsf g (H) := \{g\in H\ : \ 
\Delta (g)=g\ox g,\ \epsilon (g) =1,\ 
\Delta \sqcap^R (g)=\sqcap^R (g) \ox \sqcap^R (g),\ 
\Delta \overline \sqcap^R (g)=\overline \sqcap^R (g) \ox 
\overline \sqcap^R (g)\}
$$ 
of the set of group-like elements in $H$. 
\end{definition}

\begin{remark}\label{rem:smaller}
Let us stress that for a general weak bialgebra $H$, the set $\mathsf g (H)$
is strictly smaller than the set $\{g\in H\ :\ \Delta (g)=g\ox g,\ \epsilon
(g) =1\}$ of group-like elements. 

For example, let us consider the free $k$--vector space on the basis provided
by the morphisms of the interval category $\mathbbm{2}$. It is a
weak bialgebra via the dual of the weak bialgebra structure in Section
\ref{sec:functor_k}. In terms of Kronecker's delta, it has the unique
multiplication such that $p q=\delta_{p,q}p$, for all $p,q\in \{S,T,a\}$, the
unit $S+T+a$, the unique comultiplication for which 
$$
\Delta(S)= S \ox S,\qquad
\Delta(T)= T \ox T,\qquad
\Delta(a)= T \ox a + a \ox S
$$
and the unique counit for which $\epsilon(S)=\epsilon(T)=1$ and
$\epsilon(a)=0$. In this weak bialgebra 
$$
\sqcap^R(S)=\overline\sqcap^R(S)= S+ a
\qquad
\sqcap^R(T)=\overline\sqcap^R(T)= T
\qquad
\sqcap^R(a)=\overline\sqcap^R(a)= 0.
$$
Thus there are two group-like elements $S$ and $T$ but only
$T$ belongs to $\mathsf g(k\mathbbm{2})$. 

As we shall see below, there are some distinguished classes of weak
bialgebras $H$, however, in which $\mathsf g(H)$ coincides with the set of
group-like elements in $H$. 

In contrast to usual bialgebras, where the unit element is always group-like,
there are weak bialgebras $H$ in which the set of group-like elements (and
therefore the subset $\mathsf g(H)$) is empty. 
Consider, for example, the groupoid with two objects $S$ and $T$ and only one
non-identity isomorphism $a:S\to T$. The free $k$--vector space on the
basis provided by its morphisms, is a weak bialgebra via the dual of the weak
bialgebra structure in Section \ref{sec:functor_k}. It has the unique
multiplication such that $pq=\delta_{p,q}p$, for all $p,q\in
\{S,T,a,a^{-1}\}$, the unit $S+T+a+a^{-1}$, the unique comultiplication for
which
$$
\Delta(S)= S \ox S+a^{-1}\ox a,\ 
\Delta(T)= T \ox T+a\ox a^{-1},\ 
\Delta(a)= T \ox a + a \ox S,\ 
\Delta(a^{-1})= S \ox a^{-1} + a^{-1} \ox T,
$$
and the unique counit for which $\epsilon(S)=\epsilon(T)=1$ and
$\epsilon(a)=\epsilon(a^{-1})=0$. In this weak bialgebra there is
no group-like element.
\end{remark}

\begin{lemma}\label{lem:gr-like}
For a weak bialgebra $H$, any element $g\in H$ such that $\Delta(g)=g\ox g$
obeys the following identities. 
\begin{itemize}
\item[{(i)}] $g\sqcap^R(g)=g=\overline \sqcap^R(g)g$ and
 $\sqcap^L(g)g=g=g\overline \sqcap^L(g)$.
\item[{(ii)}] All elements $\sqcap^R(g)$, $\overline \sqcap^R(g)$,
 $\sqcap^L(g)$, $\overline \sqcap^L(g)$ are idempotent.
\item[{(iii)}] If in addition $g\in \mathsf g(H)$, then $\sqcap^R
\sqcap^L(g)= \overline \sqcap^R(g)$ and $\sqcap^L\sqcap^R(g) =\overline
\sqcap^L(g)$. 
\end{itemize}
\end{lemma}

\begin{proof}
The equalities in (i) follow from $\Delta(g)=g\ox g$ and
\eqref{counitalpropertiesofmaps}. The statements in (ii) are obtained by
applying $\sqcap^R$, $\overline 
\sqcap^R$, $\sqcap^L$ and $\overline \sqcap^L$, respectively, to the equalities
in (i), and taking into account the module map properties
\eqref{R-moduleproperties}. For $g\in
\mathsf{g}(H)$, 
\begin{equation}\label{eq:deltapibarg}
\overline \sqcap^R(g) \ox \overline \sqcap^R(g)=
\Delta\overline \sqcap^R(g)=
1_1\ox 1_2 \overline \sqcap^R(g).
\end{equation}
Applying to both sides $\mathsf{id}\ox \sqcap^R$ and multiplying on the right
the result by $g\ox 1$, by the application of part (i) we get
$$
g\ox \overline \sqcap^R(g)=
1_1g\ox \sqcap^R(1_2)\overline \sqcap^R(g).
$$
Application of $\epsilon \ox \mathsf{id}$ to both sides of this equality yields
\begin{equation}\label{eq:pibarg1}
\overline \sqcap^R(g)=\sqcap^R\sqcap^L(g)\overline \sqcap^R(g).
\end{equation}
On the other hand, applying to both sides of \eqref{eq:deltapibarg}
$\sqcap^L\ox \sqcap^R\sqcap^L$ and multiplying on the right the result by
$g\ox 1$, we obtain
$$
g\ox \sqcap^R\sqcap^L(g) = 1_2 g \ox \sqcap^R\sqcap^L(g)1_1,
$$
where we used \eqref{piLbarpiR=piLbar}, part (i), \eqref{delta1},
\eqref{R-moduleproperties}, 
anti-multiplicativity of $\sqcap^R:\sqcap^L(H)\to \sqcap^R(H)$, and
\eqref{delta1}. Thus by applying $\epsilon \ox \mathsf{id}$, we
get 
\begin{equation}\label{eq:pibarg2}
\sqcap^R\sqcap^L(g) = \sqcap^R\sqcap^L(g)\overline \sqcap^R(g).
\end{equation}
Comparing \eqref{eq:pibarg1} and \eqref{eq:pibarg2}, we conclude on the first
equality in (iii). The other equality in (iii) is proven symmetrically.
\end{proof}

\begin{proposition}\label{prop:g(H)cocommutativeH}
For a cocommutative weak bialgebra $H$, the set of group-like elements and
the set $\mathsf g(H)$ 
are equal; that is, $\mathsf{g}(H)=\{g\in H:\Delta(g)=g\ox g,
\epsilon(g)=1\}$. 
\end{proposition}

\begin{proof}
It follows immediately from the cocommutativity of $H$ that
$\sqcap^L=\overline{\sqcap}^R$ and $\sqcap^R=\overline{\sqcap}^L$, so that
$\sqcap^R(H)$ and $\sqcap^L(H)$ are coinciding commutative 
separable Frobenius subalgebras in $H$, with separability element $1_1\ox
1_2$. 
Hence if $\Delta(g)=g\ox g$, then 
\begin{equation*}
\begin{array}{rll}
\Delta\sqcap^R(g)&=& 1_1\ox \sqcap^R(g)1_2
=
1_1\ox \sqcap^R(g)\sqcap^R(g)1_2\\
&=&\sqcap^R(g)1_1 \ox \sqcap^R(g)1_2
=
\sqcap^R(g1_1)\ox \sqcap^R(g1_2)
=
\sqcap^R(g)\ox \sqcap^R(g).
\end{array}
\end{equation*}
In the first equality we used \eqref{idpiLdeltah} and in the second one we
used part (ii) of Lemma \ref{lem:gr-like}. In the third equality
we used that $\Delta(1)$ is a separability element for the commutative algebra
$\sqcap^R(H)$. In the fourth equality we used $\Delta(1)\in \sqcap^R(H)\ox
\sqcap^R(H)$ and \eqref{R-moduleproperties}. 
In the last equality we used the multiplicativity of the
comultiplication (cf. \eqref{multiplicativitycomultiplication}) and that
$\Delta(g)=g\ox g$. 
The identity $\Delta\overline\sqcap^R(g)=\overline\sqcap^R(g)\ox
\overline\sqcap^R(g)$ follows symmetrically.
\end{proof}

\begin{lemma}\label{lem:gr-Hopf}
Let $H$ be a weak Hopf algebra 
and $g\in H$ such that $\Delta(g)=g\ox g$. Then the following assertions
hold. 
\begin{itemize}
\item[{(i)}] $\Delta\sqcap^L(g)=\sqcap^L(g) \ox \sqcap^L(g)$ and
$\Delta\sqcap^R(g)=\sqcap^R(g) \ox \sqcap^R(g)$.
\item[{(ii)}] $S^2\sqcap^R(g)=\sqcap^R(g)$ and $S^2\sqcap^L(g)=\sqcap^L(g)$.
\item[{(iii)}] $\overline \sqcap^L(g)=\sqcap^R(g)$ and $\sqcap^L(g)=\overline
 \sqcap^R(g)$. 
\item[{(iv)}] $S^2(g)=g$.
\item[{(v)}] $\sqcap^R S(g)=\overline \sqcap^R(g)$ and $\overline \sqcap^R
 S(g)=\sqcap^R(g)$; $\sqcap^L S(g)=\overline \sqcap^L(g)$ and $\overline
 \sqcap^L S(g)=\sqcap^L(g)$. 
\item[{(vi)}] $\Delta\overline\sqcap^L(g)=\overline\sqcap^L(g)\ox
 \overline\sqcap^L(g)$ and
 $\Delta\overline\sqcap^R(g)=\overline\sqcap^R(g)\ox \overline\sqcap^R(g)$.
\end{itemize}
\end{lemma}

\begin{proof}
(i). Since $\Delta(g)=g\ox g$, it follows by the multiplicativity of 
$\Delta$ in \eqref{multiplicativitycomultiplication} and the
anti-co\-multiplicativity of $S$ in \eqref{eq:S_a_multiplicative} that 
$$
\Delta\sqcap^L(g)=
\Delta(g_1S(g_2))=
\Delta(gS(g))=
g_1S(g_{2'})\ox g_2S(g_{1'})=
gS(g)\ox gS(g)=
\sqcap^L(g)\ox \sqcap^L(g),
$$
and symmetrically for $\sqcap^R(g)$.

(ii). By the weak Hopf algebra axioms \eqref{antipodeaxioms} and part (i), 
$$
\sqcap^L\sqcap^R(g)=
\sqcap^R(g)_1S(\sqcap^R(g)_2)=
\sqcap^R(g)S\sqcap^R(g)
\stackrel{\eqref{piR&antipode}}=
\sqcap^R(g)\sqcap^L\sqcap^R(g).
$$
Symmetrically, 
$$
\sqcap^R(g)=
\sqcap^R\sqcap^R(g)=
S(\sqcap^R(g)_1)\sqcap^R(g)_2=
S\sqcap^R(g)\sqcap^R(g)
\stackrel{\eqref{piR&antipode}}=
\sqcap^L\sqcap^R(g)\sqcap^R(g).
$$
The right hand sides are equal by \eqref{piRpiLcommute}, proving
\begin{equation}\label{eq:gr-like-lemma}
\sqcap^L\sqcap^R(g)=\sqcap^R(g). 
\end{equation}
Applying $\sqcap^R$ to both sides of \eqref{eq:gr-like-lemma} and using
\eqref{piR&antipode}, we conclude on $S^2\sqcap^R(g)=\sqcap^R(g)$. The other
equality is proven symmetrically. 

(iii). By \eqref{eq:gr-like-lemma} and some weak Hopf algebra identities
in Section \ref{sec:prelims},
$$
\sqcap^R(g)\stackrel{\eqref{eq:gr-like-lemma}}=
\sqcap^L\sqcap^R(g)\stackrel{\eqref{piLbarpiR=piLbar}}=
\overline\sqcap^L\sqcap^L\sqcap^R(g)\stackrel{\eqref{eq:gr-like-lemma}}=
\overline\sqcap^L\sqcap^R(g)\stackrel{\eqref{piLbarpiR=piLbar}}=
\overline\sqcap^L(g).
$$
The other equality is proven symmetrically.

(iv). If $\Delta(g)=g\ox g$, then
\begin{equation}\label{eq:Sg}
gS(g)g=
g_1S(g_2)g_3=
g_1\sqcap^R(g_2)=
g.
\end{equation}
Hence 
\begin{eqnarray*}
g&=&gS(g)g=gS(gS(g)g)g=gS(g)S^2(g)S(g)g=
g_1S(g_2)S^2(g)S(g_{1'})g_{2'}\\
&=&\sqcap^L(g)S^2(g)\sqcap^R(g)=
S^2\sqcap^L(g)S^2(g)S^2\sqcap^R(g)=
S^2(\sqcap^L(g)g\sqcap^R(g))=
S^2(g).
\end{eqnarray*}
In the first and the second equalities we used \eqref{eq:Sg}. In the third and
the penultimate equalities we used anti-multiplicativity of $S$,
cf. \eqref{eq:S_a_multiplicative}. In the fourth equality we used
$\Delta(g)=g\ox g$, in the fifth equality we used the weak Hopf algebra axioms
\eqref{antipodeaxioms} and in the sixth equality we used part (ii). 
The last equality follows by part (i) of Lemma \ref{lem:gr-like}. 

(v). The first claim follows by 
$
\overline\sqcap^R(g)=
\overline\sqcap^R S^2(g)=
\sqcap^R S(g)
$,
cf. part (iv) and \eqref{piR&antipode}. The second claim is immediate by
\eqref{piR&antipode}. The remaining two claims follow symmetrically.

(vi). This is immediate by parts (i) and (iii). 
\end{proof}

From parts (i) and (vi) of Lemma \ref{lem:gr-Hopf} we obtain the following.

\begin{corollary}
In any weak Hopf algebra $H$, $\mathsf{g}(H)=\{g\in H\ :\ \Delta(g)=g\ox g,
\epsilon(g)=1\}$.
\end{corollary}

Our motivation of the study of the set $\mathsf g(H)$ in a weak
bialgebra $H$ comes from the following. 

\begin{proposition}\label{prop:gr-like}
For any weak bialgebra $H$ over a field $k$, there is a bijection
between the sets $\mathsf{wba}(k\mathbbm{2},H)$ and $\mathsf g (H)$.
\end{proposition}

\begin{proof}
 Let $\gamma\in \mathsf{wba}(k\mathbbm{2},H)$ and consider $g:=\gamma(a)$
 (where $a$ stands for the only non-identity morphism in $\mathbbm{2}$). Let 
 us see that $g\in\mathsf{g}(H)$: 
\begin{eqnarray*}
\Delta(g)&=&\Delta\gamma(a)=(\gamma\ox \gamma)\Delta_{k\mathbbm{2}}(a)=
\gamma(a)\ox\gamma(a)=g\ox g,\\
\epsilon(g)&=&\epsilon\gamma(a)=\epsilon_{k\mathbbm{2}} (a)=1,\\
\Delta\sqcap^R(g)&=
&\Delta\sqcap^R\gamma(a)=
(\gamma\ox \gamma)\Delta_{k\mathbbm{2}}\sqcap^R_{k\mathbbm{2}}(a)=
\gamma\sqcap^R_{k\mathbbm{2}}(a)\ox 
\gamma\sqcap^R_{k\mathbbm{2}}(a)\\
&=&
\sqcap^R\gamma(a)\ox \sqcap^R\gamma(a)=
\sqcap^R(g)\ox \sqcap^R(g),\\
\Delta\overline{\sqcap}^R(g)&=
&\Delta\overline{\sqcap}^R\gamma(a)=
(\gamma\ox \gamma)\Delta_{k\mathbbm{2}}
\overline{\sqcap}^R_{k\mathbbm{2}}(a)=\gamma\overline{\sqcap}^R_{k\mathbbm{2}}(a)
\ox \gamma\overline{\sqcap}^R_{k\mathbbm{2}}(a)\\
&=&
\overline{\sqcap}^R\gamma(a)\ox \overline{\sqcap}^R\gamma(a)=
\overline{\sqcap}^R(g)\ox \overline{\sqcap}^R(g).
\end{eqnarray*}
Conversely, let $g\in \mathsf{g}(H)$ and consider the linear map $\gamma:
k\mathbbm{2}\rightarrow H$, given by
$$ 
\gamma({{S}})=\sqcap^R(g),\qquad
\gamma({{T}})=\overline{\sqcap}^R(g),\qquad 
\gamma(a)=g,
$$ 
(where $S$ and $T$ are the objects of the category $\mathbbm{2}$ and the
same symbols stand for their unit morphisms). 
By Theorem \ref{thm:wba_mor}, to check that $\gamma$ is a morphism in
$\mathsf{wba}(k\mathbbm{2}, H)$ it should be proven first that $\gamma$ is
a coalgebra map. This follows by noting that --- since $\epsilon
\sqcap^R=\epsilon$ and $\epsilon \overline \sqcap^R=\epsilon$ --- for any morphism
$c$ in $\mathbbm{2}$, 
$$
\Delta\gamma(c)=
\gamma(c)\ox \gamma(c)=
(\gamma\ox \gamma)\Delta_{k\mathbbm{2}}(c)
\qquad \textrm{and}\qquad 
\epsilon\gamma(c)=
\epsilon(g)=
1=\epsilon_{k\mathbbm{2}}(c).
$$
Next, $\gamma$ can be seen to commute with $\sqcap^R$ as
\begin{eqnarray*}
\sqcap^R\gamma(S)&=&
\sqcap^R\sqcap^R(g)=
\sqcap^R(g)=\gamma(S)=
\gamma\sqcap^R_{k\mathbbm{2}}(S)\\
\sqcap^R\gamma(T)&=&
\sqcap^R\overline{\sqcap}^R(g)=
\overline{\sqcap}^R(g)=
\gamma(T)=
\gamma\sqcap^R_{k\mathbbm{2}}(T)\\
\sqcap^R\gamma(a)&=&
\sqcap^R(g)=\gamma(S)=
\gamma{s(a)}=
\gamma\sqcap^R_{k\mathbbm{2}}(a).
\end{eqnarray*}
Commutativity with $\overline \sqcap^R$ is checked symmetrically. 
Commutativity with the Nakayama automorphism $\sqcap^R\sqcap^L$ follows by
part (iii) of Lemma \ref{lem:gr-like} as 
\begin{eqnarray*}
\sqcap^R\sqcap^L\gamma(S)&=&
\sqcap^R\sqcap^L\sqcap^R(g)=
\sqcap^R\overline \sqcap^L(g)
\stackrel{\eqref{piLbarpiR=piLbar}}=
\sqcap^R(g)=
\gamma(S)=
\gamma \sqcap^R_{k\mathbbm{2}} \sqcap^L_{k\mathbbm{2}}(S)\\
\sqcap^R\sqcap^L\gamma(T)&=&
\sqcap^R\sqcap^L\overline \sqcap^R(g)
\stackrel{\eqref{piLbarpiR=piLbar}}=
\sqcap^R\sqcap^L(g)=
\overline \sqcap^R(g)=
\gamma(T)=
\gamma \sqcap^R_{k\mathbbm{2}} \sqcap^L_{ k\mathbbm{2}}(T)\\
\sqcap^R\sqcap^L\gamma(a)&=&
\sqcap^R\sqcap^L(g)=
\overline \sqcap^R(g)=
\gamma(T)=
\gamma \sqcap^R_{k\mathbbm{2}} \sqcap^L_{ k\mathbbm{2}}(g).
\end{eqnarray*}
Finally, the weak multiplicativity condition in Theorem
\ref{thm:wba_mor} translates to four equalities in parts (i) and
(ii) of Lemma \ref{lem:gr-like}, see 
$$
\begin{array}{ll}
\gamma(S)\gamma(S)=\sqcap^R(g)\sqcap^R(g)
=\sqcap^R(g)=\gamma(S)&\quad
\gamma(a)\gamma(S)=g\sqcap^R(g)
=g=\gamma(a)\\
\gamma(T)\gamma(T)=\overline \sqcap^R(g)\,\overline\sqcap^R(g)\,
=\overline\sqcap^R(g)=\gamma(T)&\quad
\gamma(T)\gamma(a)=\overline \sqcap^R(g)\,g\,
=g=\gamma(a).
\end{array}
$$
These constructions clearly yield mutually inverse maps between the
sets $\mathsf g(H)$ and $\mathsf{wba}(k\mathbbm{2},H)$.
\end{proof}

\begin{proposition}\label{prop:g_obj}
For any weak bialgebra $H$, there is a category with morphism set $\mathsf g
(H)$ in Definition \ref{groupoidlike-set}. The object set is $\{r\in 
\sqcap^R(H)=\overline \sqcap^R (H)\ :\ \Delta (r) =r\ox
r,\ \epsilon (r) =1\}$ and the identity morphisms are given by the evident
inclusion into $\mathsf g(H)$. The source map is given by the restriction of
$\sqcap^R$ and the target map is given by the restriction of $\overline
\sqcap^R$. The composition is given by the restriction of the multiplication
in $H$. 
\end{proposition}

\begin{proof}
First we check that $\mathsf{g}(H)$ is closed under the composition. Let
$g,g'\in \mathsf{g}(H)$ such that
$\sqcap^R(g)=\overline{\sqcap}^R(g^{\prime})$. Then 
\begin{eqnarray*}
\Delta(gg^{\prime})&=&
\Delta(g)\Delta(g')=(g\ox g)(g'\ox g')=gg'\ox gg'
\qquad \textrm{and}\\ 
\epsilon(gg')&=&
\epsilon
(g\overline{\sqcap}^R(g'))=
\epsilon(g\sqcap^R(g))=\epsilon(g)=1.
\end{eqnarray*}
Since
\begin{equation}\label{eq:s&t}
\begin{array}{rl}
\sqcap^R(gg')\stackrel{\eqref{piRproduct} }=&
\sqcap^R(\sqcap^R(g)g')=
\sqcap^R(\overline\sqcap^R(g')g')=
\sqcap^R(g')
\qquad \textrm{and}\\
\overline\sqcap^R(gg')\stackrel{\eqref{piRproduct}}=&
\overline\sqcap^R(g\overline\sqcap^R(g'))=
\overline \sqcap^R(g\sqcap^R(g))=
\overline \sqcap^R(g),
\end{array}
\end{equation}
also 
\begin{eqnarray*}
\Delta\sqcap^R(gg')&=&
\Delta\sqcap^R(g')=
\sqcap^R(g') \ox \sqcap^R(g')=
\sqcap^R(gg')\ox \sqcap^R(gg')
\qquad \textrm{and}\\
\Delta\, \overline\sqcap^R(gg')\ &=&
\Delta\ \overline\sqcap^R(g)\ =
\, \overline\sqcap^R(g)\, \ox \, \overline\sqcap^R(g)\, =
\overline\sqcap^R(gg') \ox \overline\sqcap^R(gg')
\end{eqnarray*}
hold 
and we conclude that $gg'\in \mathsf{g}(H)$. Associativity of the
composition is evident because of associativity of the
multiplication. 
The object set is clearly a subset of the morphism set; and for any $g\in
\mathsf{g}(H)$, both $\sqcap^R(g)$ and $\overline \sqcap^R(g)$ belong to the
object set. 
The restrictions of $\sqcap^R$ and $\overline{\sqcap}^R$ give the source and
target maps, respectively, by part (i) of Lemma
\ref{lem:gr-like}. It follows by \eqref{eq:s&t} that the composition is
compatible with the source and target maps. 
\end{proof}

The category in Proposition \ref{prop:g_obj} is also denoted by $\mathsf g
(H)$. 

\begin{remark}\label{rem:symm_restored}
For an arbitrary weak bialgebra $H$, the construction of the category $\mathsf
g(H)$ in Proposition \ref{prop:g_obj} is not symmetric under the simultaneous
replacements $\sqcap^R \leftrightarrow \overline \sqcap^L$, $\overline
\sqcap^R \leftrightarrow \sqcap^L$. This is a consequence of the choice we
made in the definition of morphisms between bimonoids (so in particular in the
definition of morphisms in $\mathsf{wba}$), see Remark \ref{rem:symm}. In
light of part (iii) of Lemma \ref{lem:gr-Hopf}, the symmetry of
the category $\mathsf g(H)$ under the simultaneous replacements $\sqcap^R
\leftrightarrow \overline \sqcap^L$, $\overline \sqcap^R \leftrightarrow
\sqcap^L$ is restored whenever $H$ is a weak Hopf algebra. 
\end{remark}

\begin{proposition}\label{prop:g_mor}
Any morphism $H\to H'$ in $\mathsf{wba}$ restricts to a functor $\mathsf g (H)
\to \mathsf g (H')$.
\end{proposition}
\begin{proof}
Let $Q:H\rightarrow H'$ be a morphism in $\mathsf{wba}$. 
First we need to see that it restricts to a map $\mathsf
g(Q)=Q_{|\mathsf g (H)}:\mathsf g (H)\rightarrow\mathsf g (H')$. Since $Q$ is
in particular a coalgebra map, it follows for all $g\in \mathsf g(H)$ that 
$$
\Delta' Q(g)=
(Q\ox Q)\Delta(g)=
Q(g)\ox Q(g)
\quad \textrm{and}\quad 
\epsilon' Q(g)=
\epsilon(g)=
1.
$$
Since $Q$ commutes also with $\sqcap^R$ and $\overline \sqcap^R$,
\begin{eqnarray*}
\Delta' \sqcap^{\prime R} Q (g)&=&
(Q\ox Q)\Delta \sqcap^R(g) =
Q\sqcap^R(g)\ox Q\sqcap^R(g)=
\sqcap^{\prime R}Q(g)\ox \sqcap^{\prime R}Q(g)\\
\Delta'\ \overline{\sqcap}^{\prime R} Q(g)&=&
(Q\ox Q)\Delta \,\overline{\sqcap}^R\ (g)=
Q\,\overline{\sqcap}^R(g)\,\ox\, Q\,\overline{\sqcap}^R(g)\,=
\overline{\sqcap}^{\prime R}Q(g)\ox \overline{\sqcap}^{\prime R}Q(g).
\end{eqnarray*}
This proves $Q(g)\in \mathsf{g}(H')$. 
Also from the compatibility of $Q$ with $\sqcap^R$ and $\overline
\sqcap^R$, it follows that $\mathsf g(Q)$
respects the source and target maps as well as the unit morphisms.
It preserves the composition by the weak multiplicativity condition; 
that is, by 
\begin{eqnarray*}
Q(gg')&=&
Q(g1_1)Q(\sqcap^R(1_2)g')\stackrel{\eqref{idpiLdeltah}}= 
Q(g_1)Q(\sqcap^R(g_2)g')\\
&=&Q(g)Q(\sqcap^R(g)g')=
Q(g)Q(\overline \sqcap^R(g')g')=
Q(g)Q(g'),
\end{eqnarray*}
for all $g,g'\in \mathsf{g}(H)$ such that $\sqcap^R(g)=\overline{\sqcap}^R(g')$.
\end{proof}

Clearly, the group-like elements in any coalgebra over a field are linearly
independent, see \cite[Theorem 2.1.2]{Abe}. Hence the elements of $\mathsf
g(H)$ in a weak bialgebra $H$ are linearly independent. Since the right
subalgebra $\sqcap^R(H)$ of $H$ is finite dimensional, this proves that the
cardinality of the object set of $\mathsf g (H)$ --- that is, of the set
$\mathsf g (H)\cap \sqcap^R(H)$ --- is finite. So we conclude by Proposition
\ref{prop:g_obj} and Proposition \ref{prop:g_mor} that there is a functor
$\mathsf g$ from $\mathsf{wba}$ to the category $\mathsf{cat}$ of small
categories with finitely many objects.
 
\section{The right adjoint of the ``free vector space'' functor.}
\label{sec:adjoint}

The aim of this section is to show that the 
functor $\mathsf g$ in Section \ref{sec:gr-like} is right adjoint of the
``free vector space'' functor $\mathsf k$ in Section \ref{sec:functor_k}. That
is, to prove the following. 
\begin{theorem}
For any small category $A$ with finitely many objects, and for any weak
bialgebra $H$ over a given field $k$, there is a bijection
$\mathsf{wba}(\mathsf k(A),H)\cong \mathsf{cat}(A,\mathsf g (H))$ which
is natural in $A$ and $H$. 
Moreover, the image of $1_{\mathsf k(-)}$ under this bijection (that is,
the unit of the adjunction $\mathsf k \dashv \mathsf g$) is a natural
isomorphism. 
\end{theorem}
\begin{proof}
We use the same symbol $A$ to denote the set of morphisms in the
category $A$.

First we show that the to-be-unit of the adjunction $\mathsf k \dashv
\mathsf g$ is a natural isomorphism. That is, for any category $A$ (with
finitely many objects) the functor $A\to \mathsf{gk}(A)$, $a\mapsto a$ is 
an isomorphism. This amounts to checking its bijectivity on the sets of
morphisms. Injectivity is obvious. In order to see
its surjectivity, let us take some $p\in \mathsf g\mathsf k(A)$. Let
us write $p=\sum_{a\in A} \lambda_a a$, with $\lambda_a\in k$ non-zero
at most for finitely many $a\in A$. Then from the requirement that $p$ is 
group-like, 
$$ 
\Delta(p)=p\ox p=\sum_{a,b\in A}\lambda_a\lambda_b a\ox b.
$$
On the other hand, by linearity of $\Delta$, 
$$
\Delta(p)=\sum_{a\in A}\lambda_a\Delta(a)=\sum_{a\in A}\lambda_a a\ox a.
$$
Since $\{a\ox b\}_{a,b\in A}$ is a linearly independent subset in 
$kA \otimes kA$, 
we conclude that $\lambda_a$ is non-zero at most for one element $a\in
A$. 
On the other hand, since 
\begin{equation*}
 1=\epsilon(p)=\lambda_a\epsilon(a)=\lambda_a,
\end{equation*}
we have $p=a\in A$.

We claim next that the desired bijection $\phi_{A,H}:
\mathsf{wba}(\mathsf k (A),H)\to \mathsf{cat}(A,\mathsf{g}(H))$ 
takes any morphism $Q:kA\to H$ to $Q_{|A}$, its restriction to $A\cong
\mathsf{gk}(A)$. 
By Proposition \ref{prop:g_mor}, $Q$ restricts to a functor $A\cong
\mathsf{gk} (A)\to \mathsf g(H)$; so that $\phi_{A,H}$ is well defined. 
Naturality of $\phi_{A,H}$ is evident.
Since $A$ is a basis of the vector space $kA$, the map $\phi_{A,H}$ is
injective. In order to show surjectivity of $\phi_{A,H}$, consider some
functor $h:A\rightarrow \mathsf{g} (H)$. Since $A$ is a basis of
the vector space $kA$, it can be extended to a unique linear map
$\widetilde{h}: kA\rightarrow H$. Let us see that $\widetilde{h}$ is a
morphism of weak bialgebras and hence $h=\phi_{A,H}(\widetilde
h)$. For any $a\in A$, $h(a)\in \mathsf g(H)$ so $\Delta h(a)=h(a)\ox h(a)$
and $\epsilon h(a)=1$. Thus $h$ extends to a coalgebra map $\widetilde
h$. The weak multiplicativity of $\widetilde h$ follows from the fact
that $h$ preserves the composition. Indeed, for $a,b\in A$,
$$
\widetilde h(a(1_{kA})_1)\widetilde h(\sqcap^R((1_{kA})_2)b)=
\delta_{s(a),t(b)}h(a)h(b)=
\delta_{s(a),t(b)}h(a.b)=
\widetilde h(ab).
$$
Since $h$ preserves the source and target maps, $\widetilde h$ commutes
with $\sqcap^R$ and $\overline{\sqcap}^R$. Finally, by part (iii) of
Lemma \ref{lem:gr-like}, 
$$
\sqcap^R\sqcap^L h (a)=
\overline \sqcap^R h (a)= 
h{{t(a)}}=
h\sqcap^R_{kA}\sqcap^L_{kA}(a) \qquad \forall a\in A,
$$
hence $\sqcap^R\sqcap^L \widetilde h=\widetilde h\sqcap^R_{kA}\sqcap^L_{kA}$
follows by linearity. 
\end{proof}

The counit of the above adjunction $\mathsf{k}\dashv \mathsf{g}$ is not
an isomorphism in general (as it is not so for usual, non-weak bialgebras; see
for example \cite{Abe}). Consider for example the weak bialgebra on the 
vector space $k\mathbbm{2}$ from Remark \ref{rem:smaller}. This weak bialgebra
$k\mathbbm{2}$ is three dimensional, while applying to it the functor
$\mathsf{kg}$ we get a one dimensional weak bialgebra. So they cannot
 be isomorphic. Another counterexample was kindly suggested by the
referee: For any (non-zero) weak bialgebra $H$ for which there are no
group-like elements in $\sqcap^R(H)$, $\mathsf{kg}(H)$ is the zero dimensional
weak bialgebra. 

\begin{proposition}\label{prop:counitadjunction}
The component $\phi_{\mathsf g(H),H}^{-1}(\mathsf g(H))
:\mathsf{k}\mathsf{g}(H)\to H$ of the counit of the
adjunction $\mathsf k \dashv \mathsf g:\mathsf{wba}\to \mathsf{cat}$ is an
isomorphism if and only if $H$ is a pointed cosemisimple weak bialgebra. 
\end{proposition}
\begin{proof}
Assume that $H$ is a pointed cosemisimple weak bialgebra; that is, 
that $H\cong k\{g\in H \ :\ \Delta(g)=g\ox g, \ \epsilon(g)=1\}$.
Since then $H$ is cocommutative, it follows by Proposition
\ref{prop:g(H)cocommutativeH} that $H\cong \mathsf{kg}(H)$.
The converse is clear since $\mathsf{kg}(H)$ is obviously a pointed
cosemisimple coalgebra. 
\end{proof}

\begin{corollary}\label{catwba-equivalence}
The functors $\mathsf k$ and $\mathsf g$ induce an equivalence between the
category of all small categories with finitely many objects, and the
full subcategory of $\mathsf{wba}$ of all pointed cosemisimple weak
bialgebras over a given field $k$.
\end{corollary}

Since over an algebraically closed field every cocommutative coalgebra is
pointed (see for example \cite[Theorem 2.3.3]{Abe}), we get the
following alternative form of Corollary \ref{catwba-equivalence}.

\begin{corollary}\label{cor:algclosed}
If $k$ is an algebraically closed field, then the
functors $\mathsf k$ and $\mathsf g$ induce an equivalence between the
category of all small categories with finitely many objects, and the full
subcategory of $\mathsf{wba}$ of all cocommutative cosemisimple weak
bialgebras. 
\end{corollary}

\section{Restriction to Hopf bimonoids.}
\label{sec:Hopf}

The aim of this section is to study and compare the full subcategories of Hopf
monoids in the categories in Section \ref{sec:cat} and in Section \ref{sec:wba}.

\begin{definition}\label{def:hopfmonoid}(cf. 
\cite[pages 193-194]{Booker&Street})
Let $(\mathsf{C},\circ,I,\bullet,J)$ be a duoidal category. We say that a
bimonoid $H$ in $\mathsf{C}$ is a {\em Hopf monoid} if the induced monoidal
comonad $(-)\bullet H$ is a right Hopf comonad; that is, if 
\begin{equation}\label{eq:can}
\xymatrix@C=10pt{
(A\bullet H)\circ (B\bullet H)
\ar[rr]^-{\raisebox{7pt}{${}_{
(A\bullet \Delta)\circ (B\bullet H)}$}}&& 
(A\bullet H\bullet H)\circ (B\bullet H)\ar[r]^-\gamma&
((A\bullet H)\circ B)\bullet (H\circ H)
\ar[rr]^-{\raisebox{7pt}{${}_{
((A\bullet H)\circ B)\bullet \mu}$}}&& 
((A\bullet H)\circ B)\bullet H}
\end{equation}
--- to be denoted by $\beta_{A,B}$ --- 
is a natural isomorphism.
\end{definition}

\begin{proposition}
For any set $X$, a Hopf monoid in $\mathsf{span}(X)$ is precisely
a groupoid with object set $X$. 
\end{proposition}

\begin{proof}
Let $H$ be a Hopf monoid in $\mathsf{span}(X)$ and consider the induced
monoidal comonad $(-)\bullet H$. By assumption, the map \eqref{eq:can}
is an isomorphism for any objects $A,B$ in $\mathsf{span}(X)$. So in
particular, for $A=B=J = X\times X$, it is an isomorphism
from $((X\times X)\bullet H)\circ ((X\times X)\bullet
H)\cong H\circ H$ to 
$(((X\times X)\bullet H)\circ (X\times X)) \bullet H
\cong \{(h,h')\in H\times H :
t(h)=t(h')\}$. 
It sends $(h,h')$ to $(h,
hh')$. 
We can write its inverse 
in the form $(h,h')\mapsto (l(h,h'),r(h,h'))$, in terms of some maps $l$
and $r$ from $H\times H$ to $H$ satisfying the conditions
\begin{eqnarray}
sl(h,h')&=&tr(h,h')\nonumber\\
sr(h,h')&=&s(h')\nonumber\\
tl(h,h')&=&t(h)\nonumber\\
l(h,h')&=&h\label{eq:l3}\\
l(h,h')r(h,h')&=&h'\label{eq:l4}
\end{eqnarray}
for all $h,h'\in H$ such that $t(h)=t(h')$ and 
\begin{equation}
r(h,hh')=h'\label{eq:r1}
\end{equation}
for all $h,h'\in H$ such that $s(h)=t(h')$. 
Using \eqref{eq:l3} to simplify \eqref{eq:l4} and substituting 
$h'={t(h)}$ in it, we obtain
\begin{equation}\label{rightinverse}
 hr(h,{t(h)})={t(h)}
\end{equation}
so that $r(h,{t(h)})$ is a right inverse of $h$. As the following
computation proves, it is also its left inverse. 
$$
r(h,{{t(h)}})h\stackrel{\eqref{eq:r1}}{=}
r(h,hr(h,{{t(h)}}) h)\stackrel{\eqref{rightinverse}}{=}
r(h,h)\stackrel{\eqref{eq:r1}}{=}
{{s(h)}}.
$$
Since this construction is valid for every $h\in H$, we showed that $H$ is a
groupoid. 

Conversely, if $H$ is a groupoid with object set $X$,
then $\beta_{A,B}: (a,h,b,h')\mapsto (a,h,b,hh')$ is an isomorphism with
the inverse $\beta_{A,B}^{-1}: (a,h,b,h')\mapsto (a,h,b, h^{-1}h')$. Therefore, 
by Definition \ref{def:hopfmonoid}, $H$ is a Hopf monoid. 
\end{proof}

\begin{proposition}
For any separable Frobenius (co)algebra $R$, a Hopf monoid in
$\mathsf{bim}(R^e)$ is precisely a weak Hopf algebra with right
subalgebra isomorphic to $R$. 
\end{proposition}

\begin{proof}
By Theorem \ref{thm:wba_obj}, a bimonoid in $\mathsf{bim}(R^e)$ is precisely a
weak bialgebra $H$ whose right subalgebra is isomorphic to $R$.
Assume that $H$ is a weak Hopf algebra with the antipode $S:H\rightarrow
H$. Then \eqref{eq:can} --- which takes now the explicit form 
$$
\beta_{A,B}((a\bullet h)\circ (b\bullet h'))=
((a\bullet h_1)\circ b)\bullet h_2h'
$$
--- is an isomorphism with the inverse 
$$
\beta^{-1}_{A,B}(((a\bullet h)\circ b)\bullet h')=
(a\bullet h_1)\circ (b\bullet S(h_2)h').
$$
This map is checked to be well-defined --- that is, $R^e$--balanced in all
of the occurring tensor products --- by computations similar to those in the
proof of Theorem \ref{thm:bim_obj}. Moreover,
\begin{eqnarray*}
\beta^{-1}_{A,B}\beta_{A,B}((a\bullet h)\circ (b\bullet h'))&=&
(a\bullet h_1)\circ (b\bullet S(h_2)h_3h')
\stackrel{\eqref{antipodeaxioms}}=
(a\bullet h_1)\circ (b\bullet \sqcap^R(h_2)h')\\
&\stackrel{\eqref{eq:bullet_left}}=& 
(a\bullet h_1)\circ (\sqcap^R(h_2)\ox 1) \cdotdot (b\bullet h')\\
&=& (a\bullet h_1)\cdotdot (\sqcap^R(h_2)\ox 1) \circ (b\bullet h')\\
&\stackrel{\eqref{eq:bullet_right}}=& 
(a\bullet h_1\sqcap^R(h_2)) \circ (b\bullet h')
\stackrel{\eqref{counitalpropertiesofmaps}}=
(a\bullet h) \circ (b\bullet h').
\end{eqnarray*}
A similar computation verifies $\beta_{A,B}\beta^{-1}_{A,B}=\mathsf{id}$.

Conversely, assume that $\beta_{A,B}$ is an isomorphism, for any 
objects $A,B$ in $\mathsf{bim}(R^e)$. Then it is an isomorphism,
in particular, for $A=B=R^e\ox R^e$ with the $R^e$--actions
$$
(r\ox s)((x\ox y)\ox (v\ox w))(r'\ox s'):=(rx\ox ys) \ox (vr'\ox s'w).
$$
Using the isomorphisms
$$
\begin{array}{lcl}
R\ox R\ox R\ox H1_1\ox \sqcap^R(1_2) H&\to&
((R^e\ox R^e)\bullet H)\circ ((R^e\ox R^e)\bullet H)\\
x\ox y\ox z\ox h1_1\ox \sqcap^R(1_2) h'&\mapsto&
(((1\ox x)\ox (1\ox 1))\bullet h)\circ
(((1\ox y)\ox (1\ox z))\bullet h')\\
\\
R\ox R\ox R\ox 1_1H\ox 1_2H&\to&
(((R^e\ox R^e)\bullet H)\circ (R^e\ox R^e))\bullet H \\
x\ox y\ox z\ox 1_1h \ox 1_2 h'&\mapsto&
((((1\ox x)\ox (1\ox 1))\bullet h)\circ
((1\ox y)\ox (1\ox z)))\bullet h',
\end{array}
$$
we obtain that
\begin{equation*}
\begin{array}{lll}
R\ox R\ox R\ox H1_1\ox \sqcap^R(1_2)H 
&\rightarrow& R\ox R\ox R\ox 1_1H\ox 1_2H\\
x\ox y\ox z\ox h1_1\ox \sqcap^R(1_2)h'&\mapsto& x\ox y\ox z\ox h_1\ox h_2h'
\end{array}
\end{equation*}
is an isomorphism. Then also the Galois map $H1_1\ox
\sqcap^R(1_2)H\to 1_1H\ox 1_2H$, $h1_1\ox \sqcap^R(1_2)h'\mapsto h_1\ox
h_2h'$ is an isomorphism. This means equivalently that $H$ is a
weak Hopf algebra (see \cite[Corollary 6.2]{Schauenburg:2003} for the details
of this equivalent characterization of weak Hopf algebras among week
bialgebras). 
\end{proof}

Let us take the full subcategory $\mathsf{grp}$ of groupoids in the category
of small categories with finitely many objects. The morphisms in
$\mathsf{grp}$ are functors (so that they are compatible with the inverse
operation on the morphisms). 
Similarly, let us take the full subcategory $\mathsf{wha}$ of weak Hopf
algebras in $\mathsf{wba}$. Its morphisms are the coalgebra maps $H\to
H'$ rendering commutative the diagrams in Theorem \ref{thm:wba_mor}. Note that 
there is no reason to expect that all of them will be compatible with the
antipodes (that is, the equality $S'Q=QS$ will hold). In fact,
compatibility with the antipodes is equivalent to $\sqcap^{\prime
L}Q=Q\sqcap^L$ holding true. 

\begin{theorem}\label{thm:grpwha}
The adjunction in Section \ref{sec:adjoint} restricts to an iso unit
adjunction between $\mathsf{grp}$ and $\mathsf{wha}$.
\end{theorem}

\begin{proof}
First we check that $\mathsf k:\mathsf{cat}\to \mathsf{wba}$ restricts
to a functor $\mathsf{grp} \to \mathsf{wha}$. 
If $A$ is a groupoid, then $kA$ has a weak Hopf algebra structure via
the antipode $S:kA\rightarrow kA$, sending every $a\in A$ to $a^{-1}$. 
The antipode axioms hold by
\begin{eqnarray*}
 a_1S(a_2)&=& 
aS(a)=a.a^{-1}={t(a)}=\sqcap^L_{kA}(a)\\
S(a_1)a_2&=& 
S(a)a=a^{-1}.a={s(a)}=\sqcap^{R}_{kA}(a)\\
S(a_1)a_2S(a_3)&=&S(a)aS(a)=a^{-1}.a.a^{-1}=
a^{-1}=S(a),
\end{eqnarray*}
see \cite[Section 2.5]{NykVai}. On the other hand, also $\mathsf
g:\mathsf{wba} \to \mathsf{cat}$ restricts to a functor $\mathsf{wha} \to
\mathsf{grp}$. That is, if $H$ is a weak Hopf algebra, then $\mathsf g(H)$ is
a groupoid (with many finitely objects) with the inverse operation $\mathsf
g(H)\rightarrow \mathsf g(H)$, $g\mapsto S(g)$. 
In order to see that $S(g)$ is an element of $\mathsf g(H)$ 
indeed, note that $\Delta
S(g)=(S\ox S)\Delta^{op}(g)=S(g)\ox S(g)$ and $\epsilon S(h)=\epsilon (h)=1$
follow from the fact that $S$ is an anti-coalgebra map. By part (v) of Lemma \ref{lem:gr-Hopf} also the other two conditions on
elements of $\mathsf g(H)$ hold true and the to-be-inverse
operation $g\mapsto S(g)$ is compatible with the source and target
maps. Moreover, it works as an inverse by 
\begin{eqnarray*}
g.g^{-1}&=&gS(g)=g_1S(g_2)=\sqcap^L(g)=\overline \sqcap^R(g)=
{t_{\mathsf g (H)}(g)} \qquad \textrm{and}\\
g^{-1}.g&=&S(g)g=S(g_1)g_2=\sqcap^R(g)={s_{\mathsf g (H)}(g)},
\end{eqnarray*}
where the penultimate equality in the first line follows by part (iii)
of Lemma \ref{lem:gr-Hopf}. 
\end{proof}
 
The following corollaries are immediate consequences of Corollary
\ref{catwba-equivalence} and Corollary \ref{cor:algclosed},
respectively. 

\begin{corollary}\label{grpdwha-equivalence}
The functors $\mathsf k$ and $\mathsf g$ induce an equivalence between the
category of all small groupoids with finitely many objects, and the
full subcategory of $\mathsf{wha}$ of all pointed cosemisimple weak Hopf
algebras over a given field $k$.
\end{corollary}

\begin{corollary}
If $k$ is an algebraically closed field, then the functors $\mathsf k$ and
$\mathsf g$ induce an equivalence between the category of all small groupoids
with finitely many objects, and the full subcategory of $\mathsf{wha}$ of all
cocommutative cosemisimple weak Hopf algebras. 
\end{corollary}

\begin{example}
Assume $k$ to be a field of characteristic $0$, and let $N$ be a positive
integer. The `algebraic quantum torus'; that is, the algebra $H =
k\langle U,V, V^{-1} \vert U^N=1,VU=qUV\rangle$, with $q \in k$ such that $q^N
= 1$, is a double crossed product weak Hopf algebra of the group 
Hopf algebra $k\langle V, V^{-1} \rangle$ and the $N$-dimensional weak Hopf
algebra $B:= k\langle U \vert U^N = 1 \rangle $ with the
comultiplication 
$$
\Delta(U^n) = \frac 1 N \sum_{j=1}^N (U^{ j+n}\ox U^{ -j}),
$$ 
 the counit defined by $\epsilon(1) = N, \epsilon (U^n) = 0
\hbox{ if } U^n \neq 1$ and the antipode $S = \mathsf{id}$ (see
\cite[Example 9]{Bohm/Gomez:2013}). 

For any $N$th root of unity $\omega \in k$ (possibly, different from $q$),
we have a group-like element $g_\omega = \frac 1 N \sum_{j=1}^{N}
\omega^{ j} U^{ j}$ of $B$. Thus,
if $k$ contains a primitive $N$th root of unity (so that the set
$T:=\{\omega\in k : \omega^N=1\}$ has $N$ elements) then, as
coalgebras, 
$$
B = \bigoplus_{ \omega \in T}
kg_\omega \qquad \textrm{and} \qquad
H = \bigoplus_{{ \omega \in T},
\ m \in \mathbb{Z}} 
kg_\omega V^m.
$$

We deduce from Corollary \ref{grpdwha-equivalence} that in this case
 $H$ is isomorphic to the groupoid weak Hopf algebra $k
\mathsf{g}$, where $\mathsf{g} = \{ g_{\omega}V^m \, \vert \, \omega\in
T, m \in \mathbb{Z} \}$. This groupoid has $N$ objects $\{g_{\omega}
\, \vert \, \omega \in T \}$, but it is not finite. Since
$g_{\omega}g_{\omega'} = 0$ if $\omega \neq \omega'$, and $g_{\omega}^2 =
g_{\omega}$, we get that two morphisms $g_{\omega}V^m, g_{\nu}V^n$ of
$\mathsf{g}$ are composable if and only if $\omega = \nu q^m$, 
and, in this case, $g_{\omega}V^mg_{\nu}V^n = g_{\omega}V^{m+n}$. 
\end{example}

\end{document}